%% file: main.tex
\documentclass[11pt]{article}
\input{Preamb.tex}
\usepackage{multirow}
\usepackage{caption}
\usepackage{subcaption}

\usepackage[numbers,merge,sort&compress]{natbib}
\begin{document}

\title{
 \bf \Large 
 A Proximal Descent Method for Minimizing Weakly Convex Optimization 
 \thanks{This work is supported by NSF ECCS 2154650, NSF CMMI 2320697, and NSF CAREER 2340713. Emails: fliao@ucsd.edu; zhengy@ucsd.edu.}
 }
\author[1]{Feng-Yi Liao}
\author[1]{Yang Zheng}
\affil[1]{\small Department of Electrical and Computer Engineering, University of California San Diego}
\date{\small \today \vspace{-5ex}} 

\maketitle

\begin{abstract}
    We study the problem of minimizing a $m$-weakly convex and possibly nonsmooth function. Weak convexity provides a broad framework that subsumes convex, smooth, and many composite nonconvex functions. In this work, we propose a \emph{proximal descent method}, a simple and efficient first-order algorithm that combines the inexact proximal point method with classical convex bundle techniques. Our analysis establishes explicit non-asymptotic convergence rates in terms of $(\eta,\epsilon)$-inexact stationarity. In particular, the method finds an $(\eta,\epsilon)$-inexact stationary point using at most 
$
    \mathcal{O}\!\left( \Big(\tfrac{1}{\eta^2} + \tfrac{1}{\epsilon}\Big) \max\!\left\{\tfrac{1}{\eta^2}, \tfrac{1}{\epsilon}\right\} \right)
$ 
function value and subgradient evaluations. Consequently, the algorithm also achieves the best-known complexity of $\mathcal{O}(1/\delta^4)$ for finding an approximate Moreau stationary point with $\|\nabla f_{2m}(x)\|\leq \delta$. A distinctive feature of our method is its \emph{automatic adaptivity}: with no parameter tuning or algorithmic modification, it accelerates to $\mathcal{O}(1/\delta^2)$ complexity under smoothness and further achieves linear convergence under quadratic growth. 
Overall, this work bridges convex bundle methods and weakly convex optimization, while providing accelerated guarantees under structural assumptions. 
\end{abstract}

\input{1-Introduction}
\input{2-Preliminary}
\input{3-ProxDescent}

\input{4-ProxBundle}

\input{5-ImprovedRate}
\input{6-Experiments}
\input{7-Conclusion}

\bibliographystyle{unsrt}
\bibliography{reference}

\newpage

\tableofcontents

\numberwithin{equation}{section}
\numberwithin{example}{section}
\numberwithin{remark}{section}
\numberwithin{assumption}{section}
\numberwithin{theorem}{section}
\numberwithin{proposition}{section}
\numberwithin{lemma}{section}
\numberwithin{definition}{section}

\appendix

\input{Appendix}

\end{document}

%% file: Preamb.tex
\usepackage{xcolor}
\usepackage{fullpage}
\usepackage{amsmath,amsthm,amssymb,amsfonts}
\usepackage{algorithm} %2e %algorithmic
\usepackage{algpseudocode}
\usepackage[hidelinks]{hyperref}
\hypersetup{
    colorlinks=true,
    linkcolor=blue,
    filecolor=magenta,      
    urlcolor=cyan,
    pdftitle={Overleaf Example},
    pdfpagemode=FullScreen,
    breaklinks=true,   % splits links across lines
}

\usepackage[noblocks]{authblk}
\usepackage{tabularx}%table
\usepackage{booktabs}%table

\usepackage{mathtools}
\usepackage{tcolorbox}
\usepackage{bbm}
\usepackage{todonotes}
\usepackage{tensor}
\usepackage[font=small,labelfont=bf]{caption}
\usepackage{subcaption}
\usepackage{graphicx}
\usepackage{sectsty}

\usepackage{adjustbox} %table
\usepackage{enumitem} %indent

\newcommand{\tr}{\mathsf{ T}}
\newcommand{\RR}{\mathbb{R}}

\newtheorem{theorem}{Theorem}
\newtheorem{definition}{Definition}[section]

\newtheorem{remark}{Remark}[section]
\newtheorem{lemma}{Lemma}%[section]

\newtheorem{proposition}{Proposition}[section]
\newtheorem{corollary}{Corollary}[section]

\newtheorem{assumption}{Assumption}[section]

\newcommand{\bsf}[1]{\textsf{\LARGE\textbf{#1}}}
\newcommand{\bigline}{\\ \vrule height2pt width 5 in depth 0pt\newline\noindent}

\newcommand{\Dist}{\mathrm{dist}}

\DeclareMathOperator*{\argmin}{\arg\!\min}

\newcommand{\innerproduct}[2]{\left \langle #1, #2 \right\rangle }

\definecolor{moccasin}{rgb}{0.98, 0.92, 0.84}
\newtcolorbox{mybox}{colback=moccasin,
colframe=moccasin}

\usepackage[nameinlink]{cleveref}
\crefname{equation}{}{}
\crefname{theorem}{Theorem}{Theorems}
\crefname{corollary}{Corollary}{Corollaries}
\crefname{example}{Example}{Examples}
\crefname{assumption}{Assumption}{Assumptions}
\crefname{lemma}{Lemma}{Lemmas}
\crefname{proposition}{Proposition}{Propositions}
\crefname{figure}{Figure}{Figures}
\crefname{table}{Table}{Tables}
\crefname{section}{Section}{Sections}
\crefname{appendix}{Appendix}{Appendices}
\Crefname{equation}{}{}
\Crefname{theorem}{Theorem}{Theorems}
\Crefname{corollary}{Corollary}{Corollaries}
\Crefname{example}{Example}{Examples}
\Crefname{lemma}{Lemma}{Lemma}
\Crefname{proposition}{Proposition}{Propositions}
\Crefname{figure}{Figure}{Figures}
\Crefname{table}{Table}{Tables}
\Crefname{section}{Section}{Sections}
\Crefname{appendix}{Appendix}{Appendices}

\newcommand{\bigO}{\mathcal{O}}

\newcommand{\PBM}{\texttt{PBM}}
\newcommand{\PPM}{PPM}

\newcommand{\muq}{\mu_{\mathrm{q}}}
\newcommand{\mup}{\mu_{\mathrm{p}}}

\newcommand{\ee}{\mathrm{e}}

%% file: 1-Introduction.tex
\section{Introduction}
We consider the unconstrained optimization problem 
\begin{align}   
    \label{eq:main-problem}
     f^\star = \min_{x \in \RR^n} \; f(x)
\end{align}
where the objective function $f:\RR^n \to \RR $ is $m$-weakly convex, possibly nonsmooth, and the optimal value $f^\star$ is assumed finite. Recall that $f$ is called $m$-weakly convex if the mapping $ x \to f(x) + \frac{m}{2} \|x\|^2$ is convex.  
The class of weakly convex functions, first studied in \cite{nurminskii1973quasigradient}, is very broad. It naturally includes all convex functions, smooth functions with Lipschitz continuous gradient, as well as composition functions of the form 
$x \mapsto h(c(x)),$
where $h :\RR^m \to \RR$ is convex and Lipschitz continuous, and the Jacobian of the function $c :\RR^n \to \RR^m$ is Lipschitz continuous \cite{davis2019stochastic}. Note that such composite functions~$h\circ c$ need not be smooth or convex; instead, this function class naturally interpolates between the convex and smooth regimes. 
The importance of weak convexity has long been recognized in classical optimization theory \cite{poliquin1996prox,rockafellar1997convex,drusvyatskiy2017proximal}. More recently, it has gained renewed attention due to its prevalence in modern applications in statistical learning and signal processing. Representative examples include robust phase retrieval, covariance matrix estimation, blind deconvolution, sparse dictionary learning, and robust principal component analysis, among others; see \cite[Section 2.1]{davis2019stochastic} and \cite{drusvyatskiy2017proximal}. 

{
\renewcommand{\arraystretch}{1.5}
\begin{table}[t]
\centering

\begin{tabular}{c c c c c}
\hline 
Assumption & Growth & Measure & Complexity &  Result  \\
\hline
\multirow{4}{*}{$L$-Lipchitz}
 & \multirow{2}{*}{No growth}
  & $(\eta,\epsilon)$-IS
 & $\bigO \left(  \left ( \frac{1}{\eta^2} + \frac{1}{\epsilon} \right ) \max \left\{\frac{1}{\eta^2},\frac{1}{\epsilon}\right\} \right )$ 
 & 
   \cref{theorem:main-result}  \\
 & & $(\delta,\alpha)$-MS & $\bigO \left( \frac{1}{\delta^4} \right )$ 
 & \cref{corollary:main-result}
 \\
 \cline{2-5}
 & \multirow{2}{*}{$\muq$-QG and $\muq > m$ } &$(\eta,\epsilon)$-IS & $\bigO \left (\max \left \{ \frac{1}{\eta^2},\frac{1}{\epsilon} \right \} \right)$ & \cref{theorem:main-result-QG}
 \\
 &  & $(\delta,\alpha)$-MS & $\bigO \left( \frac{1}{\delta^2} \right )$ &   \cref{corollary:main-result-QG} \\
\hline
\multirow{2}{*}{$M$-Smooth}
 & No growth & 
 $\|\nabla f(x) \| \leq \delta$ & $\bigO\left (\frac{1}{\delta^2} \right)$ &  \cref{theorem:smooth-functions} 
 \\
\cline{2-5}
 & $\muq$-QG and $\muq > m$ & $\|\nabla f(x) \| \leq \delta$  & $\bigO   \left  ( \log \left( \frac{1}{\delta} \right)   \right)$ &   \cref{theorem:smooth-functions-quadratic-growth}\\
\hline
\end{tabular}

\caption{Complexity of (sub)gradient and function evaluations for the proximal descent method in \Cref{alg:Proxi-descent} when minimizing $m$-weakly convex functions. We use $\muq$-QG to denote the quadratic growth condition; see \cref{eq:QG}. $(\eta,\epsilon)$-IS denotes the $(\eta,\epsilon)$-inexact stationarity, i.e., $\mathrm{dist}(0,\partial_\epsilon f(x)) \leq \eta$,  (\Cref{def:Regularized-stationary}), and $(\delta,\alpha)$-MS denotes the $(\delta,\alpha)$-Moreau stationarity, i.e., $\|\nabla_\alpha f(x)\| \leq \delta$,  (\Cref{def:Moreau-stationary}).
}
\label{tab:summary}
\end{table}
}

The main goal of this paper is to develop an efficient first-order algorithm, 
called \textit{proximal descent method}, for computing stationary points of problem \cref{eq:main-problem}. Our approach builds on the~framework of inexact proximal point methods \cite{correa1993convergence} and classical convex bundle methods \cite{lemarechal1994condensed}. While the underlying ideas have their origins in earlier works such as \cite{hare2009computing,hare2010redistributed}, our key contribution is to establish \textit{non-asymptotic convergence guarantees} for minimizing weakly convex functions. A notable feature of the proposed method is its adaptivity: without any algorithmic modification, it automatically accelerates in the presence of smoothness or quadratic growth conditions. 

\Cref{tab:summary} lists a detailed comparison of the resulting convergence rates under different assumptions on smoothness and growth (all allowing for nonconvexity). Before elaborating on these rates, let us first clarify the notion of optimality and stationary measure, which quantifies how quickly our algorithm converges. 

\subsection{Search for optimality or stationary points} 

\textbf{(Sub)gradient methods.} For convex functions, the most common performance measure is the rate at which the function values decrease, that is, how quickly $f(x_k) - \min_x f(x)$ converges to zero. 
In particular, when $f$ is differentiable and its gradient map $x \mapsto \nabla f(x)$ is $M$-Lipschitz continuous, the negative gradient direction always provides a descent direction. In this case, the basic gradient descent method  
$
    x_{k+1} = x_k - \alpha_k \nabla f(x_k), \, k = 0,1,2,\ldots,
$
with stepsize $\alpha_k = 1/M$, achieves an iteration complexity of $\mathcal{O}(1/\epsilon)$ to find a point satisfying $f(x_k) - \min_x f(x) \leq \epsilon$ \cite[Section~1.2.3]{nesterov2018lectures}. Moreover, Nesterov’s accelerated method improves this rate to the optimal $\mathcal{O}(\sqrt{1/\epsilon})$ \cite[Section~2.2]{nesterov2018lectures}. When $f$ is convex but nonsmooth, the situation is markedly different. In this case, the standard algorithm is the subgradient method $
     x_{k+1} = x_k - \alpha_k g_k , \, k = 0,1,2,\ldots,
$
where $g_k \in \partial f(x_k)$ is a subgradient. Unlike the smooth case, the negative subgradient direction may no longer guarantee descent in function value, and constant stepsize schemes typically fail to ensure convergence. Instead, a diminishing stepsize schedule is usually required to establish convergence of the function values, which achieves an iteration complexity of $\mathcal{O}(1/\epsilon^2)$ \cite[Section~3.2.3]{nesterov2018lectures}. 

For nonconvex functions, it is generally unrealistic to expect the function values to converge to a global minimum. Instead, the standard objective is to identify a stationary point. In the case of nonconvex but smooth functions, a common measure is the gradient norm along the iterates, $\|\nabla f(x_k)\|$. In particular, the basic gradient descent method with a constant stepsize achieves an iteration complexity of $\mathcal{O}(1/\epsilon^2)$ for finding an $\epsilon$-stationary point, i.e., a point satisfying $\|\nabla f(x_k)\| \leq \epsilon$ \cite[Equation 1.2.22]{nesterov2018lectures}.  
For nonconvex and nonsmooth problems, one may instead consider the minimal norm subgradient, $\Dist(0,\partial f(x_k))$, as a natural stationary measure. However, this quantity can be uninformative for subgradient methods, since $\Dist(0,\partial f(x_k))$ may remain bounded away from zero along all iterates, for example, when $f(x) = |x|$. This difficulty arises from the discontinuity of $x \mapsto  \Dist(0,\partial f(x))$, which complicates the analysis of nonsmooth nonconvex problems.  

Weak convexity provides a useful setting for analyzing nonsmooth and nonconvex problems. As highlighted in \cite{davis2019stochastic}, a key tool in this setting is the \emph{Moreau envelope}, defined for any $\rho > m$ as  
\begin{align*}
    f_{\rho}(x) := \inf_{y} \Big\{ f(y) + \tfrac{\rho}{2}\|y-x\|^2 \Big\}.
\end{align*}
The Moreau envelope smooths the $m$-weakly convex function $f$ and admits a Lipschitz continuous gradient, even when $f$ itself is nonsmooth. Consequently, the gradient norm $\|\nabla f_{\rho}(x)\|$ serves as an informative and continuous stationary measure. We defer a more detailed review of the Moreau envelope to \Cref{subsection:weakly}. In particular, the work \cite{davis2019stochastic} established a non-asymptotic convergence rate of $\mathcal{O}(1/\delta^4)$ for the subgradient method in terms of the Moreau envelope: after $J$ iterations, the method ensures  
$
    \min_{k = 1,\ldots,J} \|\nabla f_{2m}(x_k)\| \leq \delta
$  
with a complexity of order $\mathcal{O}(1/\delta^4)$. While this result provides the first rigorous rate guarantee in the weakly convex setting, the stopping criterion requires evaluating $\|\nabla f_{2m}(x_k)\|$, which may be prohibitively expensive in practice. 

\vspace{3pt}

\noindent \textbf{Proximal point methods.} 
In contrast to the subgradient method, the natural stationary measure $\Dist(0,\partial f(x_k))$ works well for the proximal point method (\PPM{}) \cite{rockafellar1976monotone}, which follows the recursion
\begin{align}
    \label{eq:PPM}
    x_{k+1} = \argmin_{x \in \RR^n}\; f(x) + \frac{\alpha}{2}\|x - x_k\|^2, \; \forall k = 1,2,\ldots,
\end{align}
with $\alpha > m$. 
 A simple telescoping argument shows that \PPM{} attains the rate $\mathcal{O}(1/\delta^2)$ for finding a point with $\Dist(0,\partial f(x_k)) \leq \delta$ (see \cref{proposition:PPM-weakly-convex}). The main drawback of \PPM{}, however, is that each iteration requires solving problem \cref{eq:PPM} which itself can be non-trivial.  

In the convex but nonsmooth setting (i.e., $m = 0$), it is common to consider inexact variants of \PPM{}, where the subproblem \eqref{eq:PPM} is solved only approximately \cite{rockafellar1973multiplier,rockafellar1976monotone}. The proximal bundle~method (\PBM{}) is a particularly efficient implementation of this idea \cite{lemarechal1994condensed}. Roughly speaking, \PBM{} addresses the difficulty of solving \eqref{eq:PPM} by constructing a sequence of simple minorants $\{\tilde f_j\}$ of $f$ (i.e., $\tilde f_j \leq f$, which is easy to construct since $f$ is convex), and then generating candidate points $\{w_j\}$ via
\begin{align*}
    w_j = \argmin_{x \in \RR^n} \; f_j(x) + \frac{\alpha}{2}\|x - x_k\|^2, \; \forall j = 1,2,\ldots,
\end{align*}
which admits an efficient solution by the construction of $f_j$. 
Through a systematic procedure for building these minorants and a specialized descent test, \PBM{} ensures that the resulting sequence $\{w_j\}$ yields a valid inexact solution of \eqref{eq:PPM}. Compared with the subgradient method, \PBM{} does not require delicate stepsize tuning to guarantee convergence of the function values, and it further achieves the optimal convergence rates for convex problems \cite{du2017rate,kiwiel2000efficiency,diaz2023optimal}.

Classical results on \PBM{} have primarily focused on the convex setting \cite{du2017rate,kiwiel2000efficiency,diaz2023optimal,lemarechal1994condensed,liao2023overview,liao2025bundle}. More recently, several works have extended \PBM{} to weakly convex problems, including \cite{hare2009computing,hare2010redistributed,fuduli2004minimizing,davis2019proximally,liang2023proximal,apkarian2009proximity}. Among these, the work \cite{davis2019proximally} introduced a novel scheme that combines \PPM{} with the subgradient method, called the \emph{proximally guided stochastic subgradient method} (PGSG). This algorithm achieves a non-asymptotic complexity guarantee of $\bigO(1/\epsilon^2)$ finding a point $x_k$ such that  
$$
    \Dist\!\left(x_k, \left\{ x \in \RR^n \,\middle|\, \Dist(0, \partial f(x))^2 \leq \varepsilon \right\} \right)^2 \leq \varepsilon.
$$
However, PGSG requires a \emph{predefined} number of inner subgradient iterations for approximately solving the subproblem \cref{eq:PPM}, which reduces its flexibility in practice.  
Another line of work \cite{hare2009computing,hare2010redistributed} explored constructing convex bundle models to lower approximate the local convexification $f(\cdot) + \tfrac{m}{2}\|\cdot - x_k\|^2$. This idea is conceptually very appealing and indeed serves as one of the motivations for our approach; but no iteration complexity guarantees were established in \cite{hare2009computing,hare2010redistributed}. 
More recently, \cite{liang2023proximal} also exploited the idea of local convexification and proposed a variant of the classical \PBM{} for weakly convex functions, and proved that it can find a point $x_k$ satisfying $\|\nabla f_{2m}(x_k)\| \leq \delta$ at a rate of $\mathcal{O}(1/\delta^4)$. Nonetheless, the algorithmic framework of \cite{liang2023proximal} appears considerably more complicated compared to the classical \PBM{} for convex functions or the subgradient method.

\subsection{Our contributions}

In this work, we develop a \emph{simple} and \emph{efficient} first-order algorithm, called the \emph{proximal descent method}, for computing stationary points of the weakly convex optimization problem \cref{eq:main-problem}.~Our~method is based on the framework of the inexact proximal point method \cite{rockafellar1976monotone}, with each subproblem solved efficiently via the classical convex bundle method \cite{lemarechal1994condensed}.  
A central concept in our algorithm is the $\epsilon$-inexact subdifferential $\partial_\epsilon f(x)$ (see \cref{eq:inexact-subdifferential} for a precise definition). At each iterate $x_k$, the proximal descent method aims to generate a new point $x_{k+1}$ satisfying the decrease condition  
\begin{equation} \label{eq:PPM-value-drop-main}
    f(x_{k+1}) \leq f(x_k) - \frac{\tilde{\beta}}{2\alpha}\|\tilde{g}_{k+1}\|^2,
\end{equation}
where $\tilde{\beta} \in (m/\alpha,1]$ and $\tilde{g}_{k+1} \in \partial_{\epsilon_{k+1}} f(x_{k+1})$ is an inexact subgradient with accuracy $\epsilon_{k+1} \geq 0$. For the exact PPM update \cref{eq:PPM}, this inequality holds with $\tilde{\beta} = 1$ and $\epsilon_{k+1} = 0$, i.e., $\tilde{g}_{k+1} \in \partial f(x_{k+1})$. Since solving \cref{eq:PPM} exactly is generally intractable, we relax the condition to allow both a smaller decrease and inexact subgradients as in \cref{eq:PPM-value-drop-main}. If \cref{eq:PPM-value-drop-main} holds, a simple telescoping argument then guarantees that %after $T$ iterations, 
the method finds a point satisfying $\min_{k=1,\ldots,T} \|\tilde{g}_{k+1}\| \leq \eta$ within $T =\mathcal{O}(1/\eta^2)$ iterations. Additional analysis shows that the inexactness terms $\epsilon_{k+1}$ also converge to zero, ensuring true stationarity (\Cref{lemma:proximal-descent}).  

Our key observation is that a point $x_{k+1}$ satisfying \cref{eq:PPM-value-drop-main} can be computed efficiently using classical convex bundle methods. Following the idea of \cite{hare2009computing,hare2010redistributed}, we decompose the PPM parameter $\alpha$ as $\alpha = m + \rho$, leading to the subproblem  
\[
x_{k+1} = \argmin_{x \in \RR^n}\; f(x) + \frac{m}{2}\|x - x_k\|^2 + \frac{\rho}{2}\|x - x_k\|^2,
\]
where $m$ is the \emph{convexification parameter} and $\rho$ the \emph{proximal parameter}. The function $x \mapsto f(x) + \tfrac{m}{2}\|x - x_k\|^2$ is convex for any $x_k \in \RR^n$, which enables us to adapt convex bundle methods to efficiently construct lower approximations of this local convex model and obtain iterates satisfying \cref{eq:PPM-value-drop-main}. Unlike the classical convex bundle setting, however, the local convexified model $f(\cdot) + \tfrac{m}{2}\|\cdot - x_k\|^2$ evolves with the iterate $x_k$, necessitating a different analysis for the proposed algorithm.

The idea of combining inexact \PPM{} with convex bundle methods provides a natural extension of the classical \PBM{} from the convex to the weakly convex setting. Our main technical contribution is to establish explicit convergence rates for the resulting proximal descent method. We measure progress using the notion of \emph{$(\eta,\epsilon)$-inexact stationarity}, defined as $\Dist(0,\partial_\epsilon f(x)) \leq \eta$. For weakly convex functions, we show that the proximal descent method finds an $(\eta,\epsilon)$-inexact stationary point within at most
\[
    \mathcal{O}\!\left( \Big(\tfrac{1}{\eta^2} + \tfrac{1}{\epsilon}\Big) \max\!\left\{\tfrac{1}{\eta^2}, \tfrac{1}{\epsilon}\right\} \right)
\]
function and subgradient evaluations (\cref{theorem:main-result}). Moreover, since $(\eta,\epsilon)$-inexact stationarity implies $(\alpha,\delta)$-Moreau stationarity, defined as $\|\nabla f_{\alpha}(x) \| \leq \delta$, our analysis further guarantees that the method produces an iterate $x_k$ with $\|\nabla f_{2m}(x_k)\| \leq \delta$ in at most $\mathcal{O}(1/\delta^4)$ function and subgradient evaluations (\cref{corollary:main-result}), matching the best known rates in \cite{davis2019stochastic,liang2021proximal}.  

Finally, we show that the proximal descent method \emph{automatically accelerates} in the presence of additional structural assumptions, such as smoothness or quadratic growth, without requiring any algorithmic modification. These accelerated guarantees apply for nonconvex functions, thereby extending the convex-only results of \cite{diaz2023optimal} to the nonconvex setting. In particular, when $f$ is $M$-smooth, the method adapts to find a point $x_k$ with $\|\nabla f(x_k)\| \leq \delta$ in $\mathcal{O}(1/\delta^2)$ iterations (\cref{theorem:smooth-functions}). If $f$ further satisfies a quadratic growth condition, the proximal descent method achieves a linear convergence rate (\cref{theorem:smooth-functions-quadratic-growth}). Unlike existing methods for weakly convex optimization (e.g., subgradient or PGSG), which require explicit tuning or modification to exploit smoothness, our approach adapts naturally. A summary of our convergence results across different settings is provided in \cref{tab:summary}. To the best of our knowledge, this is the first analysis that unifies best-known rates in the \textit{weakly convex} setting with accelerated guarantees under additional structural assumptions. A more detailed comparison with existing algorithms is given in \Cref{subsection:comparison}. 

\subsection{Paper organization}

The rest of this paper is organized as follows. \Cref{section:prelim} introduces stationary measures and reviews the (sub)gradient method and the proximal point method. \Cref{section:PBM-inexact} introduces our \textit{proximal descent method}, and  \Cref{section:PMB-weakly} details the bundle updates and establishes the non-asymptotic convergence rates. 
\Cref{section:improved-rate} establishes improved convergence results under additional regularity conditions. \Cref{section:numerics} demonstrates the numerical behavior. \Cref{section:conclusion} concludes the paper and discusses future directions. Some detailed proofs and discussion are presented in the appendix.

%% file: 2-Preliminary.tex
\section{Basic notation and preliminaries}
\label{section:prelim}

In this section, we define some basic notation and review some background materials.  We first introduce the subdifferential for weakly convex functions and the Moreau envelope in \Cref{subsection:weakly}, and then discuss two stationarity measures in \Cref{subsection:stationary}. We finally review the subgradient method and the proximal point method to minimize weakly convex functions in \cref{subsection:subgradient-method}.  

\subsection{Subdifferential and the Moreau envelope}
\label{subsection:weakly}

We focus on algorithms to minimize weakly convex functions. Let us begin with the definition. 

\begin{definition}
    A function $f: \mathbb{R}^n \to \mathbb{R} \cup \{+\infty\}$ is called $m$-weakly convex if the function $x \mapsto f(x) + \frac{m}{2}\|x\|^2$ is a convex function. 
\end{definition}

The class of weakly convex functions is quite broad, including any smooth functions, convex functions, and the composition of a convex Lipschitz function and a smooth function \cite{davis2019stochastic}. Given an $m$-weakly convex function $f$, it is clear that $f(\cdot) + \frac{m}{2}\|\cdot - z\|^2$ is also convex for any fixed $z \in \mathbb{R}^n$. This means that the function $f$ is convex up to a quadratic perturbation with curvature $m$. This simple fact is important since it allows us to construct a global affine under-estimator for the perturbed function $f(\cdot) + \frac{m}{2}\|\cdot - z\|^2$. 

Consider a proper closed function $f:\mathbb{R}^n \to \mathbb{R} \cup \{+\infty\}$ and a point $x$ with $f(x)$ finite. The Fr\'echet subdifferential of $f$ at $x\in\mathbb{R}^n$ is defined as  \cite[Definition 1.10]{kruger2003frechet}
$$
    \partial f(x) = \left \{ v \in \RR^n \mid \liminf_{y \to x }  \frac{ f(y) - f(x) -  \innerproduct{v}{y-x}}{\|y-x\|} \geq 0 \right\}.
$$ 
Informally, the Fr\'echet subdifferential $\partial f(x)$ consists of all vectors $v$ such that the function $f$ admits an affine under-estimator $l(y) = f(x) -  \innerproduct{v}{y-x}$ locally at the point $x$. If $f$ continuously differentiable, the Fr\'echet subdifferential $\partial f(x)$ only consists of the gradient $\{\nabla f(x)\}$ \cite[Exercise 8.8]{rockafellar2009variational}. For convex functions, it becomes the same as the usual convex subdifferential \cite[Proposition 8.12]{rockafellar2009variational}
$
    \partial f (x) = \{ v \in \RR^n \mid f(y) \geq f(x) + \innerproduct{v}{y-x}, \forall y \in \RR^n\}. 
$ 

The following result characterizes the Fr\'echet subdifferential of a weakly convex function. 
\begin{lemma} \label{lemma:subdifferential-weakly-convex}
    Let $f:\mathbb{R}^n \to \mathbb{R}\cup \{+\infty\}$ be a closed $m$-weakly convex function. We have
\begin{align}
    \partial f(x) = \left \{v \in \RR^n \mid f(y) \geq  f(x) + \innerproduct{v}{y-x} - \frac{m}{2}\|y-x\|^2,\; \forall y \in \RR^n\right \}. \label{eq:Frechet-subdifferetnial-concrete}
\end{align}
\end{lemma}
This result is standard; see, for example, \cite[Lemma 2.1]{davis2019stochastic}. \Cref{lemma:subdifferential-weakly-convex} means that a weakly convex function admits global \textit{concave quadratic} under-estimators from any subgradient. This generalizes the case of convex functions, where any subgradient yields a global affine under-estimator. 

Given a function $f:\mathbb{R}^n \to \mathbb{R}\cup \{+\infty\}$ and $\rho > 0$, the Moreau envelope is defined by 
\begin{equation} \label{eq:Moreau-envelope}
    f_{\rho}(x) := \inf_{y\in \mathbb{R}^n} \left\{f(y) + \frac{\rho}{2}\|y-x\|^2\right\}.
\end{equation}
The Moreau envelope $f_{\rho}$ is a natural minorant of $f$ by definition since $\inf_{y} \{f(y) + \frac{\rho}{2}\|y-x\|^2\} \leq f(x)$. When $f$ is convex, the Moreau envelope is continuously differentiable for any $\rho > 0$; When $f$ is $m$-weakly convex, with $\rho > m$, the function $y \mapsto f(y) + \frac{\rho}{2}\|y-x\|^2$ becomes strongly convex, and the minimizer $ \hat{x}:=\argmin_{y} \{f(y) + \frac{\rho}{2}\|y-x\|^2\}$ is unique. In this case, the Moreau envelope also becomes continuously differentiable and the gradient is given by the following lemma.

\begin{lemma}
    \label{lemma:gradient-moreau}
    Let $f:\mathbb{R}^n \to \mathbb{R}\cup \{+\infty\}$ be a closed $m$-weakly convex function. Then, for any $\rho > m$, the Moreau envelope $f_{\rho}$ is continuously differentiable with gradient given by 
\begin{align}
    \label{eq:gradient-Moreau}
    \nabla f_{\rho}(x) = \rho (x - \hat{x}),
\end{align}
where $\hat{x}:=\argmin_{y} \{f(y) + \frac{\rho}{2}\|y-x\|^2\}$ is unique. 
\end{lemma}

A simple but fundamental fact is that for any $\rho > m$, we have $\nabla f_{\rho}(x) = \{0\}$ if and only if $0 \in \partial f(x)$, i.e., the set of stationary points of a weakly convex function coincides with those of its Moreau envelope (see \cref{prop:moreau-iff}). Moreover, at these stationary points, the function values also agree, i.e., $f(x) = f_\rho (x)$ whenever $ 0 \in \partial f(x)$. \Cref{fig:envelop} illustrates this fact using a $2$-weakly convex function $x \mapsto |x^2-1|$.

\begin{figure}[t]
\setlength\textfloatsep{0pt}
\centering 
{\includegraphics[width=0.4\textwidth]
{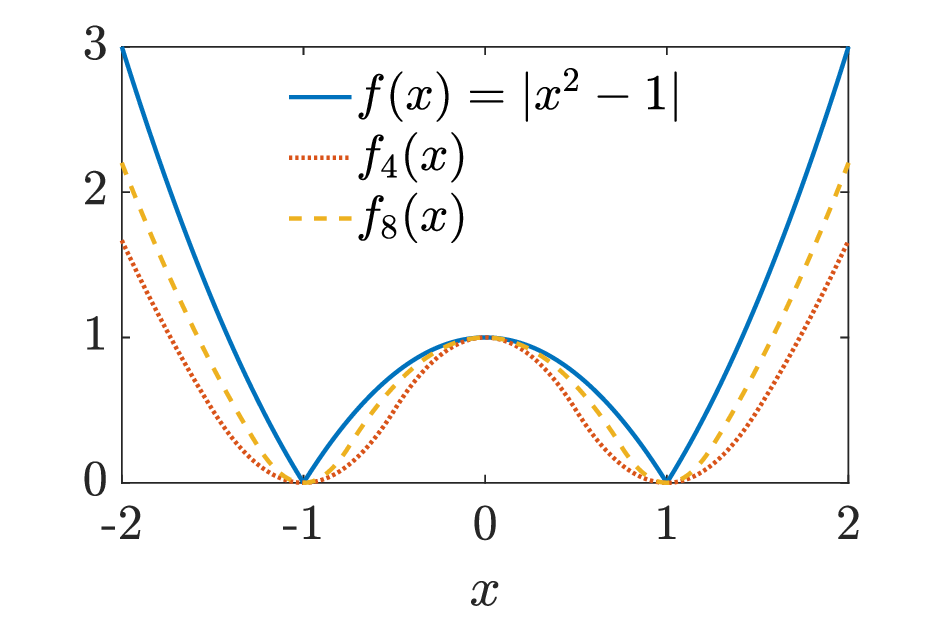}}
\caption{Moreau envelops for a $2$-weakly convex function $x \mapsto f(x)= |x^2-1|$ with $\rho = 4$ and $8$. Their stationary points are the same at $-1, 0,$ and $1$, and their function values also agree at these stationary points.}
\label{fig:envelop}
\end{figure}

\subsection{Approximate stationary points} \label{subsection:stationary}

When minimizing a weakly convex function, we cannot guarantee the discovery of its global minimum due to the potential nonconvexity. The practical goal is to find a first-order stationary point, i.e., a point $x$ satisfying $0 \in \partial f(x)$. This motivates the need for a notion of approximate stationarity.  

For smooth functions, the gradient norm $\|\nabla f(x)\|$ is a natural measure of stationarity since $\nabla f$ is continuous: small gradient norms indicate proximity to a stationary point.
For nonsmooth functions, one might consider $\mathrm{dist}(0,\partial f(x))$, which is the minimum distance from the origin to the subdifferential. However, this quantity may be discontinuous: even when $x$ is close to a stationary point, $\mathrm{dist}(0,\partial f(x))$ might be bounded away from zero. For example, in the scalar function $f(x) = |x|$, whose subdifferential jumps discontinuously at $x = 0$. 

For an $m$-weakly convex function, we can consider an inexact Fr\'echet subdifferential, defined~as\footnote{We note that $\partial_{\epsilon} f$ in \cref{eq:inexact-subdifferential} is different from the $\epsilon$-subdifferential defined in \cite[Definition 1.10]{kruger2003frechet}. Our $\partial_{\epsilon} f$ is more aligned with the $\epsilon$-subdifferntial for a convex function.}
\begin{align} \label{eq:inexact-subdifferential}
    \partial_{\epsilon} f(x)  & = \left \{v \in \RR^n \mid f(y) \geq  f(x) + \innerproduct{v}{y-x} - \frac{m}{2}\|y-x\|^2 - \epsilon, \forall y \in \RR^n \right \},
\end{align}
where $\epsilon \geq 0$ is the inexactness.~When $\epsilon = 0$, $\partial_\epsilon f(x)$ recovers the original subdifferential $\partial f(x)$ (see \Cref{lemma:subdifferential-weakly-convex}).  We then consider the following $(\eta,\epsilon)$-inexact stationary point.

\begin{definition}[$(\eta,\epsilon)$-inexact stationarity]
    \label{def:Regularized-stationary}
    Let $f:\RR^n \to \RR\cup \{+\infty\}$ be an $m$-weakly convex function, and $\eta \geq 0, \epsilon \geq 0$. A point $x \in \mathrm{dom}\, f$ is called an $(\eta,\epsilon)$-inexact stationary point if 
    \begin{align*}
        \Dist(0,\partial f_{\epsilon}(x)) \leq \eta.
    \end{align*}
\end{definition}

This $(\eta,\epsilon)$-inexact stationarity is also called the \textit{regularized stationarity} in \cite[Definition 2.5]{liang2023proximal}. 
It is clear that the above notion recovers the exact stationary if $\eta = 0$ and $\epsilon = 0$. For a fixed $\epsilon > 0$,  $\Dist(0,\partial_{\epsilon} f(x))$ is indeed a continuous function in $x$ when $f$ is convex; see \cref{lemma:continuous-subdifferential} for details.

As mentioned above, the \textit{Moreau envelope} $f_\rho$ of an $m$-weakly convex function is \textit{continuously differentiable} when $\rho > m$. The magnitude $\|\nabla f_{\rho}(x)\|$ naturally provides a quantification in terms of near-stationarity. We define another type of approximate stationary point.  
\begin{definition}[$(\delta,\alpha)$-Moreau stationarity]
    \label{def:Moreau-stationary}
    Let $f:\RR^n \to \RR\cup \{+\infty\}$ be an $m$-weakly convex function, and $\delta>0, \alpha >m$. A point $x \in \mathrm{dom}\, f$ is called an $(\delta,\alpha)$-Moreau stationary point if 
\begin{align*}
    \|\nabla f_{\alpha}(x)\| \leq \delta.
\end{align*}
\end{definition}

The optimality condition of the minimization problem \cref{eq:Moreau-envelope} directly implies that for any $x \in \mathbb{R}^n$ and $\alpha > m$, the point $\hat{x}:=\argmin_{y} \{f(y) + \frac{\alpha}{2}\|y-x\|^2\}$ satisfies 
\begin{subequations}
    \label{eq:Moreau-property}
    \begin{align}
    \|\hat{x} - x\| &=  \|\nabla f_\alpha(x)\|/\alpha, \label{eq:Moreau-property-1} \\
f(\hat{x}) &\leq f(x), \label{eq:Moreau-property-2}\\
\Dist(0, \partial f(\hat{x})) & \leq \|\nabla f_\alpha(x)\|. \label{eq:Moreau-property-3}
\end{align}
\end{subequations}
We now see from \cref{eq:Moreau-property-1} that that the distance between $x$ and $\hat{x}$ is  $\|\nabla f_\alpha(x)\|/\alpha$. Also, \cref{eq:Moreau-property-2} means the point $\hat{x}$ achieves a lower function value; and \Cref{eq:Moreau-property-3} says that the magnitude of the minimal norm element in the subdifferential at $\hat{x}$  is at most $\|\nabla f_\alpha(x)\|$. Thus, a small $\|\nabla f_{\alpha}(x)\|$ indicates that there is a nearby point $\hat{x}$ that is close to stationary, measured by $\Dist(0, \partial f(\hat{x}))$. If $\|\nabla f_{\alpha}(x)\| = 0$, we must have  $\Dist(0, \partial f(x)) = 0$, i.e., $x$ is an exact stationary point. 

These two notions, $(\eta,\epsilon)$-inexact stationarity and $(\delta,\alpha)$-Moreau stationarity, are closely related, as shown below.  

\begin{proposition}[$(\eta,\epsilon)$-inexact stationarity to $(\delta,\alpha)$-Moreau stationarity]
\label{lemma:stationary-conversion-R2M-2-main}
    Let $f:\RR^n \to \RR\cup \{+\infty\}$ be an $m$-weakly convex function. Suppose $x$ is a $(\eta,\epsilon)$-inexact stationary point. For any $\lambda > 0$, the point $x$ satisfies
    \begin{align*}
        \|\nabla f_{m+\lambda}(x)\|  \leq  (m + \lambda)\left ( \frac{2}{\lambda}\eta + \sqrt{\frac{2}{\lambda}\epsilon} \right).
    \end{align*}
    If $\lambda = m$, it simplifies to $\|\nabla f_{2m}(x)\| \leq  4 \eta + 2\sqrt{2m\epsilon}$.
\end{proposition}
A similar result appears in \cite[Proposition 2.10]{liang2023proximal} where $\lambda = m$. Our result holds for~any~$\lambda >0$. {\Cref{lemma:stationary-conversion-R2M-2-main} implies that an $(\eta,\epsilon)$-inexact stationary point must also be an $(\delta,\alpha)$-Moreau~stationary point, where $\delta$ depends on $\epsilon$ and $\eta$.} A natural question is then whether $(\delta,\alpha)$-Moreau~stationarity implies $(\eta,\epsilon)$-inexact stationarity. This is the case when the function is Lipschitz continuous.

\begin{proposition}[$(\delta,\alpha)$-Moreau stationarity to $(\eta,\epsilon)$-inexact stationarity]
     \label{prop:stationary-conversion-M2R-main}
        Let $f:\RR^n \to \RR\cup \{+\infty\}$ be an $m$-weakly convex function and $L$-Lipschitz continuous. If $x \in \RR^n$ is an $(\delta,\alpha)$-Moreau stationary point, then $x$ is an $(L\delta/\alpha,\sqrt{2(\alpha - m) L{\delta}/{\alpha}})$-inexact stationary point.
\end{proposition}
We provide the proofs of \cref{lemma:stationary-conversion-R2M-2-main,prop:stationary-conversion-M2R-main}, as well as further discussions on different stationarities, in \Cref{appendix:stationary-point}. 

\subsection{Subgradient methods and proximal point methods} \label{subsection:subgradient-method}

We next review two basic algorithms to minimize a weakly convex function $f$. The simplest~algorithm is {arguably} the \textit{subgradient method}, which iteratively updates 
\begin{equation} \label{eq:subgradient-method}
    x_{k+1} = x_k - \alpha_k g_k, \qquad k = 0, 1, 2, \ldots, 
\end{equation}
where $g_k \in \partial f(x_k)$, and $\alpha_k > 0$ is a step size. If $f$ is $M$-smooth, then \cref{eq:subgradient-method} becomes the gradient descent algorithm. In this case, a proper choice of step sizes can ensure the best gradient norm $\min_{k \in \{1,2, \ldots, T\}}\|\nabla f(x_k)\|$ converges to zero with rate $\bigO(1/\sqrt{T})$ \cite[Section 1.2.3]{nesterov2018lectures}.  If $f$ is convex and $M$-smooth, then we can ensure that the function value gap $f(x_k) - \min_x f(x)$ converges to zero with rate $\bigO(1/T)$ \cite[Corollary 2.1.2]{nesterov2018lectures}. If the function $f$ is only convex but nonsmooth, the best function value gap $\min_{k \in \{1,2, \ldots, T\}} f(x_k) - \min_x f(x)$ converges to zero with rate $\bigO(1/\sqrt{T})$ \cite[Section 3.2.3]{nesterov2018lectures}. However, it was an open question to determine the rate of convergence for \cref{eq:subgradient-method} when applied to the class of $m$-weakly convex functions. One difficulty was that neither the function value gap $f(x_k) - \min_x f(x)$ nor the stationary measure $\Dist(0,\partial f(x_k))$ goes to zero.  

\begin{algorithm}[t]
\caption{Subgradient method}\label{alg:subgradient-main}
\begin{algorithmic}
\Require $x_1 \in \RR^n$, a sequence of step sizes $\alpha_k > 0$, and maximum iteration $T$
\For{$k = 1, 2, \ldots, T$}
    \State set $ x_{k+1} = x_k - \alpha_k g_k,$ where $g_k \in \partial f(x_k)$;
\EndFor
\end{algorithmic}
\end{algorithm}

The seminal work \cite{davis2019stochastic} established the first convergence rate using the notion of Moreau stationarity. We here state the deterministic version of \cite[Theorem 3.1]{davis2019stochastic}.

\begin{lemma}
    \label{theorem:weakly-subgradient}
    Consider the $m$-weakly convex problem \cref{eq:main-problem}. Let $\alpha = \rho + m$ with $\rho > 0$. The subgradient method in \Cref{alg:subgradient-main} generates a sequence of iterates, satisfying
    \begin{align*}
        \min_{k \in \{ 1,\ldots,T\} } \|\nabla f_{\alpha}(x_k)\|^2 \leq \frac{\alpha}{\rho} \frac{ (f_{\alpha}(x_1) - f^\star) + \frac{\alpha L^2}{2} \sum_{k=1}^T \alpha_k^2 }{\sum_{k=1}^T \alpha_k},
    \end{align*}
    where $L\geq \max_{k \in \{ 1,\ldots,T\}}\|g_k\|$. If we choose a constant step size  $\alpha_k = \sqrt{\frac{\Delta}{m L^2 (T+1)}}$ with $\Delta \geq f_{2m}(x_1) - \min_{x}\, f(x)$, then it holds that 
    \begin{align*}
         \min_{k \in \{ 1,\ldots,T\} } \|\nabla f_{2m}(x_k)\|^2 \leq \sqrt{\frac{2m \Delta L^2}{T+1}}.
    \end{align*}
\end{lemma}

\Cref{theorem:weakly-subgradient} shows that the subgradient method ensures the Moreau stationary measure~goes to zero with a rate $\mathcal{O}(T^{-1/4})$. In other words, the subgradient method finds an iterate $x_k$ with $\|\nabla f_{2m}(x_k)\| \leq \delta$ within at most $\bigO(1/\delta^4)$ iterations. However, we should carefully interpret this complexity result. First, as in the standard subgradient method, the stepsize $\alpha_k$ depends on the iteration horizon $T$. Moreover, the measure $\|\nabla f_{2m}(x_k)\|$ is implicit and, in general, not computable, and thus it is difficult to implement a stopping criterion for the subgradient method.

Another conceptually simple algorithm is the PPM \cite{rockafellar1976monotone}, which is closely linked with the Moreau envelope. For an $m$-weakly convex function $f$, \PPM{} generates a sequence of iterates by 
\begin{align}
    \label{eq:PPM-weakly}
    x_{k+1} = \argmin_{x \in \RR^n}\; f(x) + \frac{\alpha}{2}\|x - x_k\|^2, \; \forall k = 1,2,\ldots
\end{align}
where $\alpha > m$. Since we choose $\alpha > m$, the minimization problem in \cref{eq:PPM-weakly} becomes strongly convex, and thus the minimizer is unique. Accordingly, the iterates in \cref{eq:PPM-weakly} are well-defined. The convergence of PPM is very well-studied for convex functions \cite{rockafellar1976monotone,luque1984asymptotic}. Even for $m$-weakly convex functions, it is relatively straightforward to derive the following convergence guarantee.

\begin{algorithm}[t]
\caption{Proximal point method (PPM)}\label{alg:PPM}
\begin{algorithmic}
\Require $x_1 \in \RR^n$, $\alpha > m$, the maximum number of iterations $T$.
\For{$k=1,2, \ldots, T$}
    \State $x_{k+1} =  \argmin_{y} f(y) + \frac{\alpha}{2}\|y-x_k\|^2$. 
\EndFor
\end{algorithmic}
\end{algorithm}

\begin{proposition} \label{proposition:PPM-weakly-convex}
    Consider the $m$-weakly convex problem \cref{eq:main-problem}. Let $\alpha > m$. The proximal point method in \Cref{alg:PPM} generates a sequence of iterates, satisfying
    \begin{equation} 
        \min_{k \in \{1, 2,\ldots,T+1\} } \Dist^2(0,\partial f(x_k)) \leq \frac{2\alpha(f(x_1) -f^\star )}{T}. 
    \end{equation}
\end{proposition}
\begin{proof}
    Since $x_{k+1}$ is the minimizer in \cref{eq:PPM-weakly}, we have 
    \begin{equation} \label{eq:cost-value-drop}
    f(x_{k+1}) + \frac{\alpha}{2}\|x_{k+1} - x_{k}\|^2 \leq f(x_k). 
    \end{equation}
    By the optimality condition, we have $\nabla f_{\alpha}(x_{k}) = \alpha (x_{k+1} - x_{k})$ in \cref{eq:Moreau-property-1}, we have
    $$
    \|\nabla f_{\alpha}(x_{k})\|^2 \leq 2\alpha(f(x_k) - f(x_{k+1})). 
    $$
    By a telescope sum, we obtain 
    $
    \sum_{k = 1}^{T} \|\nabla f_{\alpha}(x_k)\|^2 \leq 2\alpha(f(x_1) - f(x_{T+1})) \leq2\alpha (f(x_1)-f^\star). 
    $ 
    Considering the property in \cref{eq:Moreau-property-3}, we have $\Dist(0,\partial f(x_{k+1})) \leq \|\nabla f_{\alpha}(x_k)\|$, leading to 
    \begin{align*}
          \sum_{k = 1}^{T} \Dist^2(0, \partial f(x_{k+1})) \leq2\alpha (f(x_1)-f^\star) \quad
          & \Rightarrow \quad  T\min_{k \in \{2,\ldots,T+1\} } \Dist^2(0,\partial f(x_k)) \leq 2\alpha (f(x_1) -f^\star) \\
          & \Rightarrow \quad  T\min_{k \in \{1,\ldots,T+1\} } \Dist^2(0,\partial f(x_k)) \leq 2\alpha (f(x_1) -f^\star).
    \end{align*}
    This completes the proof. 
\end{proof}

Note that the convergence guarantee of the PPM is much stronger than that of the subgradient method. In particular, \Cref{proposition:PPM-weakly-convex} provides a direct convergence rate for the stationary measure $\Dist(0,\partial f(x_k))$, which decays at a rate $\mathcal{O}(1/\sqrt{T})$ for any parameter $\alpha > m$. 
This convergence rate holds uniformly for both convex and $m$-weakly convex functions. The PPM not only ensures convergence of $\Dist(0,\partial f(x_k))$, but also achieves a faster rate of $\mathcal{O}({T}^{-1/2})$ than the subgradient method, which converges at  $\mathcal{O}({T}^{-1/4})$ in terms of the Moreau stationary measure.  

One significant drawback in PPM is its computational complexity in each iteration \cref{eq:PPM-weakly}. Updating the iterates via \cref{eq:PPM-weakly} is typically much more expensive than the simple subgradient update in \cref{eq:subgradient-method}. As a result, PPM is often employed as a conceptual algorithm that guides the design and analysis of more practical methods \cite{drusvyatskiy2017proximal,rockafellar1976augmented}.

%% file: 3-ProxDescent.tex
\section{Implicit subgradient and proximal descent methods}
\label{section:PBM-inexact}

In this section, we first introduce the idea of implicit subgradient updates. The discussion applies to both convex and $m$-weakly convex functions, with the convex case corresponding to $m = 0$. We then introduce a conceptual yet implementable proximal descent method. 

\subsection{Motivation: inexact subgradient method}

It is known that the PPM can be viewed as an implicit subgradient method \cite{correa1993convergence}, and this perspective remains valid when minimizing a weakly convex function. In particular, the optimality condition in PPM update \cref{eq:PPM-weakly} ensures that $
0 \in \partial f(x_{k+1}) + \alpha (x_{k+1} - x_{k}), 
$ and thus, we can rewrite it as 
\begin{equation} \label{eq:implicit-subgradient}
    x_{k+1} = x_{k} - \frac{1}{\alpha} g_{k+1},
\end{equation}
where $g_{k+1} \in \partial f(x_{k+1})$ and $\alpha > m$. Unlike the standard (explicit) subgradient method \cref{eq:subgradient-method}, this iteration is called an implicit subgradient update, as the subgradient $g_{k+1}$ is evaluated as the next point $x_{k+1}$ rather than the current one. 

Meanwhile, the cost drop for PPM   \cref{eq:cost-value-drop}  can be expressed in terms of the implicit subgradient~as 
\begin{equation} \label{eq:PPM-value-drop-subgradient}
    f(x_{k+1}) \leq f(x_k) - \frac{1}{2\alpha}\|g_{k+1}\|^2.
\end{equation}
A telescope sum from \cref{eq:PPM-value-drop-subgradient} directly leads to the convergence rate in \Cref{proposition:PPM-weakly-convex}. The implicit subgradient method \cref{eq:implicit-subgradient} is conceptual, but it provides a guiding principle for implementable algorithms. In particular, many implementable methods can be interpreted as approximately achieving the cost decrease in \cref{eq:PPM-value-drop-subgradient}.  We discuss here a general idea of inexact subgradient methods.   

In particular, we aim to find a point $x_{k+1}$ satisfying 
\begin{equation} \label{eq:PPM-value-drop-subgradient-inexact}
    f(x_{k+1}) \leq f(x_k) - \frac{\beta}{2\alpha}\|\tilde{g}_{k+1}\|^2,
\end{equation}
where $\beta \in (0,1)$ relaxes the cost value drop and $\tilde{g}_{k+1} \in \partial f_{\epsilon}(x_{k+1})$ is an inexact subgradient at the next point with inexactness $\epsilon > 0$ (recall the inexact subdifferential in \cref{eq:inexact-subdifferential}).
It is now clear that the inexact subgradient update \cref{eq:PPM-value-drop-subgradient-inexact} guarantees that 
$$
\min_{k \in \{ 1,\ldots,T+1\} } \Dist^2(0,\partial_{\epsilon} f(x_k)) \leq \frac{2\alpha(f(x_1) - \min_x f(x) )}{\beta T}.
$$
In other words, the update \cref{eq:PPM-value-drop-subgradient-inexact} finds an $(\eta,\epsilon)$-inexact stationary point for weakly convex functions with a rate $\mathcal{O}(T^{-1/2})$. We emphasize that \cref{eq:PPM-value-drop-subgradient-inexact} is a relaxed version of \cref{eq:PPM-value-drop-subgradient} from the PPM. Still, we need to establish an implementable method to achieve \Cref{eq:PPM-value-drop-subgradient-inexact} efficiently.  

\subsection{An inexact subgradient method from proximal approximations}
\label{subsection:PBM-convex}
The inexact update \cref{eq:PPM-value-drop-subgradient-inexact} is a good guiding principle, but it does not guarantee the convergence of stationarity as the inexactness $\epsilon$ is fixed across iterations. Based on \Cref{eq:PPM-value-drop-subgradient-inexact}, we here consider the following modified requirement
\begin{equation} \label{eq:bundle-update}
     f(x_{k+1}) \leq f(x_k) - \frac{m + \beta \rho }{\alpha}\times \frac{1}{2\alpha}\|\tilde{g}_{k+1}\|^2,
\end{equation}
where $\rho > 0$, $\alpha =  m + \rho$ , $\beta \in (0,1)$ and $\tilde{g}_{k+1} \in \partial f_{\epsilon_{k+1}}(x_{k+1})$. Compared with \cref{eq:PPM-value-drop-subgradient-inexact}, the cost reduction factor is relaxed to be 
$$
\frac{m + \beta\rho }{\alpha} \in \left(\frac{m}{\alpha},1\right),
$$
which allows the inexactness $\epsilon_{k}$ to be time varying. We will properly control $\epsilon_{k}$ such that they are summable, thus approaching zero.

A general idea is that if a trial point $z_{k+1}$ does not satisfy \cref{eq:bundle-update}, we will use it to find a better point. With some iterations, it guarantees that we can find a point satisfying \cref{eq:bundle-update} efficiently.~We will detail this strategy as \textit{proximal bundle updates} later, which builds on the idea of \cite{hare2009computing,hare2010redistributed}. Throughout this section, we let $\rho = \alpha - m > 0$. Then, the exact proximal update \cref{eq:PPM-weakly} can be written~as 
\begin{align}
    \label{eq:PPM-weakly-split}
    x_{k+1} = \argmin_{x \in \RR^n}\; f(x) + \frac{m}{2}\|x - x_k\|^2 + \frac{\rho}{2} \|x - x_k\|^2.
\end{align}
Since $f$ is $m$-weakly convex, the function $f(\cdot) + \frac{m}{2}\|\cdot - x_k\|^2$ is naturally convex. Thus, \cref{eq:PPM-weakly-split} can also be viewed as a proximal mapping on the convex function $f(\cdot) + \frac{m}{2}\|\cdot - x_k\|^2$ with the parameter $\rho > 0$, centered at $x_k$. Here, we call $m$ the  \textit{convexification} parameter and $\rho$ the \textit{proximal} parameter. 

One key idea is to use a \textit{simple} convex function $\tilde{f}$ to lower approximate the convex part $f(\cdot) + \frac{m}{2}\|\cdot - x_k\|^2$, such that the following proximal update
\begin{align} \label{eq:trial-point-1}
    z_{k+1}:= \argmin_{x}\; \tilde{f}(x) + \frac{\rho}{2}\|x - x_k\|^2
\end{align}
can be solved efficiently. Thanks to the convexity, a globally lower and convex approximation 
\begin{equation} \label{eq:lower-approximation-1}
    \tilde{f}(x) \leq f(x) + \frac{m}{2}\|x - x_k\|^2, \qquad \forall x \in \mathbb{R}^n
    \tag{Minr.}
\end{equation}
is easy to construct, e.g., using a first-order approximation. Then, the optimality condition for \cref{eq:trial-point-1} naturally gives 
$
{w}_{k+1}:=\rho (x_k -z_{k+1} ) \in \partial \tilde{f}(z_{k+1}).
$
Furthermore, with some arguments, we see that {$\tilde{g}_{k+1}: = {w}_{k+1} + m (x_k -z_{k+1} ) = \alpha (x_k -z_{k+1})$} is an inexact subgradient of the original function $f$ at the trial point $z_{k+1}$. 
\begin{lemma} \label{lemma:inexactness-update}
    Consider the $m$-weakly convex problem \cref{eq:main-problem}. Let $x_k \in \mathbb{R}^n$ and $\tilde{f}:\RR^n \to \RR$ be a convex function that satisfies the lower approximation \cref{eq:lower-approximation-1}. Denote $z_{k+1}$ as the proximal update in \cref{eq:trial-point-1}. Then, we have  
    \begin{subequations} \label{eq:inexactness-trial-point}
        \begin{equation} \label{eq:inexactness-trial-point-a}
           \tilde{g}_{k+1} := \alpha (x_k -z_{k+1} )\in \partial f_{\epsilon_{k+1}} (z_{k+1}),
        \end{equation}
    where $\alpha = m + \rho$ and the inexactness is given by 
    \begin{equation} \label{eq:inexactness-trial-point-b}
        \epsilon_{k+1} = f(z_{k+1}) + \frac{m}{2}\|z_{k+1} - x_k\|^2 - \tilde{f}(z_{k+1}) \geq 0. 
    \end{equation}
    \end{subequations}
\end{lemma}
\begin{proof}
We observe the following property: for all $y \in \RR^n$, it holds that
\begin{align}
    \label{eq:inexactness-update-step-1}
    f(y) + \frac{m}{2}\|y - x_k\|^2 \geq \tilde{f}(y)  \geq  \tilde{f}(z_{k+1}) + \innerproduct{w_{k+1}}{y-z_{k+1}},
\end{align}
where the first inequality is by the lower approximation \cref{eq:lower-approximation-1} and the second inequality is due to the subgradient ${w}_{k+1}:=\rho (x_k -z_{k+1} ) \in \partial \tilde{f}(z_{k+1})$ for the convex function $\tilde{f}$. Then, by adding $\frac{m}{2}\|y - z_{k+1}\|^2$ on both sides of \cref{eq:inexactness-update-step-1} and a simple rearrangement of terms, we obtain
\begin{equation} \label{eq:subgradient-lower-bound}
    f(y) + \frac{m}{2}\|y - z_{k+1}\|^2 \geq  \tilde{f}(z_{k+1}) + \innerproduct{{w}_{k+1}}{y-z_{k+1}} + \frac{m}{2}(\|y-z_{k+1}\|^2- \|y - x_k \|^2).  
\end{equation}
Also, we have the identity 
\begin{equation}\label{eq:parallele-identity}
\begin{aligned}
   \|y - x_k \|^2  &=  \|y-z_{k+1}\|^2 -2\langle x_k - z_{k+1}, y - z_{k+1} \rangle + \| z_{k+1} - x_k \|^2. 
\end{aligned}
\end{equation}
Substituting \cref{eq:parallele-identity} into \cref{eq:subgradient-lower-bound}, we obtain 
\begin{equation*}
\begin{aligned}
    f(y) + \frac{m}{2}\|y - z_{k+1}\|^2 &\geq \tilde{f}(z_{k+1}) + \innerproduct{{w}_{k+1}}{y-z_{k+1}}  + m\langle x_k - z_{k+1}, y - z_{k+1}\rangle - \frac{m}{2}\|z_{k+1} - x_k \|^2.\\
    &=\tilde{f}(z_{k+1}) + \innerproduct{\tilde{g}_{k+1}}{y-z_{k+1}}  - \frac{m}{2}\| z_{k+1} - x_k\|^2\\
    &=f(z_{k+1})+ \innerproduct{\tilde{g}_{k+1}}{y-z_{k+1}}  - \left (f(z_{k+1}) + \frac{m}{2}\|z_{k+1} - x_k  \|^2 - \tilde{f}(z_{k+1})\right), 
\end{aligned}
\end{equation*}
for any $ y \in \mathbb{R}^n$. This confirms that $\tilde{g}_{k+1} = (\rho + m)(x_k - z_{k+1}) = \alpha (x_k - z_{k+1})\in \partial f_{\epsilon_{k+1}} (z_{k+1})$ with $\epsilon_{k+1}$ given in \cref{eq:inexactness-trial-point-b}.
\end{proof}

As shown in \cref{eq:inexactness-trial-point-b}, the inexactness $\epsilon_{k+1}$ is the same as the approximation gap at the next trial point $z_{k+1}$. Furthermore, since 
\cref{eq:trial-point-1} is an exact proximal update for the approximated function $\tilde{f}$, so we must have 
 \begin{equation} \label{eq:PPM-value-drop-inexact}
    \tilde{f}(z_{k+1}) \leq \tilde{f}(x_k) - \frac{1}{2\rho}\|w_{k+1}\|^2 \leq f(x_k) - \frac{1}{2\rho}\|w_{k+1}\|^2 % \frac{\rho}{2\alpha^2}\|\tilde{g}_{k+1}\|^2,
\end{equation}
where the first inequality comes from \cref{eq:PPM-value-drop-subgradient}, and the second inequality is due to \cref{eq:lower-approximation-1}. We also note the following relationship 
$ \|w_{k+1}\| = \frac{\rho}{\alpha}\|\tilde{g}_{k+1}\|.$ 
We thus observe that this cost value drop \cref{eq:PPM-value-drop-inexact} is bounded by a fraction of $\|\tilde{g}_{k+1}\|^2$, similar to the requirement \cref{eq:bundle-update}, with the exception that \cref{eq:PPM-value-drop-inexact} measures the approximated function value $\tilde{f}(z_{k+1})$. 

The value $f(x_k) - \tilde{f}(z_{k+1})$ can be viewed as the approximated cost value drop that is predicted by $\tilde{f}$. If $\tilde{f}(z_{k+1})$ approximates the true value $f(z_{k+1}) + \frac{m}{2}\|z_{k+1} - x_k\|^2 $ good enough in terms of 
\begin{equation} \label{eq:descent-condition-main}
    f(x_k) - \left(f(z_{k+1}) + \frac{m}{2}\|z_{k+1} - x_k\|^2\right)  \geq \beta \left(f(x_k) - \tilde{f}(z_{k+1})\right), 
    \tag{Approx.}
\end{equation}
where $\beta \in (0,1)$ is a relaxation parameter, we will obtain the desired descent criterion \cref{eq:bundle-update}. In particular, we have the following result.  

\begin{lemma}
    \label{lemma:descent-step}
     Consider the $m$-weakly problem \cref{eq:main-problem}. Let $x_k \in \mathbb{R}^n$ and $\tilde{f}:\RR^n \to \RR$ be a convex function that satisfies the lower approximation \cref{eq:lower-approximation-1}. Denote $z_{k+1}$ as the proximal update in \cref{eq:trial-point-1}. If this point $z_{k+1}$ satisfies \cref{eq:descent-condition-main}, then we have  
      \begin{subequations} \label{eq:desecent-progress}
        \begin{align}
            f(z_{k+1}) &\leq f(x_k) - \frac{m + \beta \rho }{\alpha}\times \frac{1}{2\alpha}\|\tilde{g}_{k+1}\|^2, \label{eq:desecent-progress-a} \\
            \epsilon_{k+1}&\leq \frac{1-\beta}{\beta}\left (f(x_k) - f(z_{k+1}) - \frac{m}{2\alpha^2}\|\tilde{g}_{k+1}\|^2\right ) , \label{eq:desecent-progress-b}
        \end{align}
        with $\tilde{g}_{k+1} \in \partial f_{\epsilon_{k+1}}(z_{k+1})$. 
\end{subequations} 
\end{lemma}
\begin{proof}
We observe the following property. 
\begin{equation*}
\begin{aligned}
      f(z_{k+1}) &\leq  f(x_k) - \frac{m}{2}\|z_{k+1} - x_k\|^2  - \beta \left(f(x_k) - \tilde{f}(z_{k+1})\right) \\
      &\leq f(x_k) - \frac{m}{2}\|z_{k+1} - x_k\|^2  - \frac{\beta}{2\rho}\|w_{k+1}\|^2  \\
      &= f(x_k) - \frac{ m + \beta \rho }{\alpha} \times \frac{1}{2\alpha}\|\tilde{g}_{k+1}\|^2, 
\end{aligned}
\end{equation*}
where the first inequality is a rearrangement of \cref{eq:descent-condition-main}, the second inequality applies \cref{eq:PPM-value-drop-inexact}, and the last inequality uses $w_{k+1} = \rho (x_k - z_{k+1})$ and the relationship $\|w_{k+1}\| = \frac{\rho}{\alpha}\|\tilde{g}_{k+1}\|$. %Note that the above cost value drop is the same as \cref{eq:bundle-update}.
When \cref{eq:descent-condition-main} holds, the inexactness parameter $\epsilon_{k+1}$ also satisfies  
\begin{equation} \label{eq:inexactness-from-descent}
\begin{aligned}
    \epsilon_{k+1} &\leq (1-\beta)(f(x_k) - \tilde{f}(z_{k+1})) \\
    &\leq \frac{1-\beta}{\beta}\left(f(x_k) - \left(f(z_{k+1}) + \frac{m}{2}\|z_{k+1} - x_k\|^2\right)\right) \\
    &= \frac{1-\beta}{\beta}\left (f(x_k) - f(z_{k+1}) - \frac{m}{2\alpha^2}\|\tilde{g}_{k+1}\|^2 \right).  
\end{aligned}
\end{equation}
where $\tilde{g}_{k+1} \in \partial f_{\epsilon_{k+1}}(z_{k+1})$ is defined in \cref{eq:inexactness-trial-point-a}. 
\end{proof}

If a trial point $z_{k+1}$ satisfies \cref{eq:descent-condition-main}, we denote this step as 
\begin{equation} \label{eq:proximal-descent}
z_{k+1} = \texttt{ProxDescent}(x_k,\beta,\rho),
\end{equation}
which means that $z_{k+1}$ comes from the proximal step \cref{eq:trial-point-1} while satisfying \cref{eq:descent-condition-main}. We thus can list a meta-algorithm to minimize an $m$-weakly convex function in \Cref{alg:Proxi-descent}, which we call the \textit{proximal descent method}.  This proximal descent method has the following convergence guarantees.

\begin{algorithm}[t]
\caption{Proximal descent method}\label{alg:Proxi-descent}
\begin{algorithmic}
\Require $x_1 \in \RR^n$, $T$, $\beta$, $\rho$.
\For{$k=1,2, \ldots, T$}
    \State $x_{k+1} =  \texttt{ProxDescent}(x_k,\beta,\rho)$ from \cref{eq:proximal-descent};  
\EndFor
\end{algorithmic}
\end{algorithm}

\begin{theorem} \label{lemma:proximal-descent}
    Consider the $m$-weakly convex problem \cref{eq:main-problem}. The proximal descent method in \Cref{alg:Proxi-descent} with parameters $\beta \in (0,1)$, $\rho > 0$, and $\alpha = m + \rho$,  takes at most 
    \begin{equation} \label{eq:descent-bound}
    \frac{2\alpha^2 (f(x_1) - f^\star)}{m + \beta \rho } \frac{1}{\eta^2} + \frac{(1-\beta)\left(f(x_1) - f^\star\right)}{\beta}\frac{1}{\epsilon} + 1. 
    \end{equation}
    iterations to find an $(\eta,\epsilon)$-inexact stationary point. 
\end{theorem}
\begin{proof}
    By \cref{lemma:descent-step}, each iteration of \Cref{alg:Proxi-descent} satisfies \cref{eq:desecent-progress}. 
        \begin{subequations}
        A telescope sum for \cref{eq:desecent-progress-a} leads~to 
        \begin{equation} \label{eq:inexact-subgradient-bound}
        \sum_{k = 1}^T \|\tilde{g}_{k+1}\|^2 \leq (f(x_1) - f(x_{T+1})) \frac{2\alpha^2}{m + \beta \rho } \leq  \frac{2\alpha^2 (f(x_1) - f^\star)}{m + \beta \rho }. 
        \end{equation}
        On the other hand, a telescope sum for \cref{eq:desecent-progress-b} leads to 
        \begin{equation} \label{eq:inexactness-bound}
        \begin{aligned}
        \sum_{k=1}^T \epsilon_{k+1} &\leq \frac{1-\beta}{\beta}\left(f(x_1) - f(x_{T+1}) - \frac{m}{2\alpha^2}\sum_{k=1}^T(\|\tilde{g}_{k+1}\|^2)\right), \\
        &\leq \frac{1-\beta}{\beta}\left(f(x_1) - f^\star\right). 
        \end{aligned}
         \end{equation}
The right-hand size bounds are fixed in both \cref{eq:inexact-subgradient-bound,eq:inexactness-bound}. Therefore, both $\tilde{g}_{k+1}$ and $\epsilon_{k+1}$ approach to zero asymptotically. Note that they may not decrease monotonically; thus, we can only guarantee a sublinear rate for their best values, i.e., $\min_{k \in \{1,\ldots,T +1\}} \|\tilde{g}_{k}\| $ and  $\min_{k \in \{1,\ldots,T +1\}}\epsilon_{k} $, respectively. However, the best values of $\tilde{g}_{k+1}$ and $\epsilon_{k+1}$ may not be achieved at the same iterate.     
        \end{subequations}
        
We need another trick to establish the sublinear rate \cref{eq:descent-bound}. Fix $\eta > 0, \epsilon > 0$, and suppose we run $T$ iterations. We let $G\subseteq \{1, \ldots, T\}$ denote the indices such that $\|\tilde{g}_{i}\| > \eta, \forall i \in G$, and let $E \subseteq \{1, \ldots, T\}$ such that $\|\epsilon_{i}\| > \epsilon, \forall i \in E$. From \cref{eq:inexact-subgradient-bound,eq:inexactness-bound}, we see the cardinality of these two index sets satisfies 
\begin{equation*}
    \begin{aligned}
        |G|\leq \frac{2\alpha^2 (f(x_1) - f^\star)}{m + \beta\rho } \frac{1}{\eta^2}, \quad 
        |E|\leq \frac{(1-\beta)\left(f(x_1) - f^\star\right)}{\beta}\frac{1}{\epsilon}.
     \end{aligned}
\end{equation*}
If we let the total number of iterations $T \geq |G| + |E| + 1$, then there must be at least one iterate $k$ such that both $\|\tilde{g}_{k}\| \leq \eta, \epsilon_k \leq \epsilon$. Thus, the bound in \cref{eq:descent-bound} is sufficient. 
\end{proof}

We note that \Cref{lemma:proximal-descent} can be viewed as an inexact version of \cref{proposition:PPM-weakly-convex}, and the sublinear rate in \Cref{lemma:proximal-descent} is similar to that in \cref{proposition:PPM-weakly-convex}. Now, the important step is to generate a convex function $\tilde{f}$ to lower approximate $f(\cdot) + \frac{m}{2}\|\cdot - x_k\|^2$ such that the trial point $z_{k+1}$ in \cref{eq:trial-point-1} satisfies the gap \cref{eq:descent-condition-main}. Since the function $f(\cdot) + \frac{m}{2}\|\cdot - x_k\|^2 $ is convex, $\texttt{ProxDescent}(x_k,\beta,\rho)$ in \cref{eq:proximal-descent} can be realized using a classical proximal bundle idea, which will be detailed in \Cref{subsection:bundle-step}.

\begin{remark}[Approximation criteria]
The approximation criterion \cref{eq:descent-condition-main} is very commonly used in bundle-type methods \cite{lemarechal1981bundle,lemarechal1994condensed,diaz2023optimal,du2017rate,kiwiel2000efficiency}. 
Another potential approximation criterion might be 
\begin{equation} \label{eq:alterantive-criteria}
        f(z_{k+1}) + \frac{m}{2}\|z_{k+1} - x_k\|^2 - \tilde{f}(z_{k+1})  \leq (1-\beta)\frac{1}{2\rho}\|w_{k+1}\|^2,
\end{equation}
where $\beta \in (0,1)$. Combing \cref{eq:alterantive-criteria} with \cref{eq:PPM-value-drop-inexact} also leads to the desired inequality  
\begin{equation*}
\begin{aligned}
      f(z_{k+1}) \leq f(x_k) - \frac{m}{2}\|z_{k+1} - x_k\|^2  - \frac{\beta}{2\rho}\|w_{k+1}\|^2  
      &= f(x_k) - \frac{\beta \rho + m}{\alpha} \times \frac{1}{2\alpha}\|\tilde{g}_{k+1}\|^2. 
\end{aligned}
\end{equation*} 
However, the requirement of \cref{eq:alterantive-criteria} further implies that the inexactness parameter $\epsilon_{k+1}$ satisfies  
\begin{equation*} 
\epsilon_{k+1}=f(z_{k+1}) + \frac{m}{2}\|z_{k+1} - x_k\|^2 - \tilde{f}(z_{k+1}) \leq (1-\beta)\frac{\rho}{2\alpha^2} \|\tilde{g}_{k+1}\|^2. 
\end{equation*}
It means that the inexactness is also upper bounded by $\|\tilde{g}_{k+1}\|$. In contrast, \cref{eq:inexactness-from-descent} informs us that $\epsilon_{k+1}$ could be larger if $\tilde{g}_{k+1}$ gets smaller. Therefore, ensuring \cref{eq:alterantive-criteria} is harder than \cref{eq:descent-condition-main}. {Indeed, we will have finite-time guarantees to ensure  \cref{eq:descent-condition-main} using a bundle strategy in the next subsection, while the authors are currently unclear whether \cref{eq:alterantive-criteria} can be efficiently realized.} \hfill $\square$
\end{remark}

%% file: 4-ProxBundle.tex
\section{A proximal bundle algorithm for weakly convex functions}
\label{section:PMB-weakly}

In this section, we introduce a proximal bundle update to realize \texttt{ProxDescent($x_k,\beta,\rho$)}, and derive the iteration complexity in terms of subgradient and function evaluations for \Cref{alg:Proxi-descent}. We also interpret \Cref{alg:Proxi-descent} as a standard proximal bundle method to minimize weakly convex functions. 

\subsection{Proximal bundle updates for \texttt{ProxDescent}} \label{subsection:bundle-step}
We here detail a classical bundle idea to achieve the proximal descent step \cref{eq:proximal-descent}. This step has two criteria: 1) a global convex lower approximation in \cref{eq:lower-approximation-1}; and 2) an appropriate approximation quality for the next proximal point in \cref{eq:descent-condition-main}. In particular, any convex lower approximation $\tilde{f}$ in \cref{eq:lower-approximation-1} with a proximal mapping guarantees an inexact subgradient for the original function, as shown in \cref{eq:inexactness-trial-point-a}; and the approximation quality \cref{eq:descent-condition-main} guarantees the inexact update \cref{eq:desecent-progress}.  

We now consider a fixed proximal center $x_{k} \in \RR^n$. For any convex lower approximation $\tilde{f}$ satisfying \cref{eq:lower-approximation-1}, we have 
\begin{align} \label{eq:proximal-gap}
    \min_{x} \left \{ \tilde{f}(x) + \frac{\rho}{2}\|x - x_k\|^2  \right \}   \leq f_{\alpha}(x_k) := \min_{x} \left \{f(x) + \frac{\alpha}{2}\|x - x_k\|^2 \right \},
\end{align}
where we recall $\alpha = m + \rho$. 
If \cref{eq:descent-condition-main} is not satisfied, we will update $\tilde{f}$ such that the value $\min_{x} \{ \tilde{f}(x) + \frac{\rho}{2}\|x - x_k\|^2 \}$ gets closer to the true Moreau envelope value $f_{\alpha}(x_k)$. With these iterations, the criterion \cref{eq:descent-condition-main} will be satisfied eventually. 

For notational convenience, let an index $j\geq 1$ and $\tilde{f}_{j}$ be a convex lower approximation satisfying \cref{eq:lower-approximation-1}, and we~denote 
\begin{align} \label{eq:lower-approximation-j}
    z_{j+1} = \argmin_{x} \left \{ \tilde{f}_j(x) + \frac{\rho}{2}\|x - x_{k}\|^2  \right \}, \quad  \eta_j := \min_{x} \left \{ \tilde{f}_j(x) + \frac{\rho}{2}\|x - x_{k}\|^2 \right \}.
\end{align}
From the optimality condition $s_{j+1} : = \rho(x_k - z_{j+1}) \in \partial \tilde{f}_j(z_{j+1})$, we observe that 
\begin{align} \label{eq:lower-approximation-j-reformulation}
   \eta_{j} = \min_{x} \left  \{ \tilde{f}_j(z_{j+1}) + \innerproduct{s_{j+1}}{x - z_{j+1}}  + \frac{\rho}{2}\|x - x_k\|^2 \right \}. 
\end{align}
This means that the minimization problems in \cref{eq:lower-approximation-j} and \cref{eq:lower-approximation-j-reformulation} have the same optimal value and the same minimizer. 
Thus, if the next lower approximation $\tilde{f}_{j+1}$ satisfies \cref{eq:lower-approximation-1}, as well as 
   \begin{align}
     \tilde{f}_{j+1}(\cdot) & \geq  \tilde{f}_j(z_{j+1}) + \innerproduct{s_{j+1}}{\cdot - z_{j+1}},   \label{eq:requirement-2} 
    \tag{Aggr.}
   \end{align} 
then the next optimal value $\eta_{j+1} := \min_x \{ \tilde{f}_{j+1}(x) + \frac{\rho}{2}\|x - x_{k}\|^2 \}$ is guaranteed to satisfy 
\begin{align*} 
    \eta_j \leq \eta_{j+1} \leq f_{\alpha}(x_k)
\end{align*}
where the first inequality is due to \cref{eq:requirement-2} and the second inequality is due to \cref{eq:lower-approximation-1}.

To ensure a strict improvement, we further require $\tilde{f}_{j+1}$ incorporates a subgradient of $f(\cdot) + \frac{m}{2}\|\cdot - x_{k}\|^2$ at the trial point $z_{j+1}$, i.e., there is a $g_{j+1} \in \partial \left (f(\cdot) + \frac{m}{2}\|\cdot- x_{k}\|^2 \right)(z_{j+1})$ such that
\begin{align}
    \label{eq:requirement-3}
    \tilde{f}_{j+1} (\cdot) \geq f(z_{j+1}) + \frac{m}{2}\|z_{j+1} - x_{k}\|^2 + \innerproduct{ g_{j+1}}{\cdot - z_{j+1}}.
    \tag{Subg.}
\end{align}
We have a key technical result that quantifies the improvement in terms of the proximal value $\eta_j$. 

\begin{lemma}
    \label{lemma:null-step-improvement}
    Consider an $m$-weakly convex function  $f:\RR^n \to \RR$, a proximal center $x_{k} \in \mathbb{R}^n$, and a proximal parameter $\rho > 0$. Let $\tilde{f}_{j}$ be any convex function satisfying \cref{eq:lower-approximation-1} and $\tilde{f}_{j+1}$ be any convex function  satisfying \cref{eq:lower-approximation-1}, \cref{eq:requirement-2}, and \cref{eq:requirement-3}. Then, we~have  
\begin{align} \label{eq:strict-improvement}
    \eta_{j+1} \geq \eta_j + \frac{1}{2} \min\left \{ \tilde{\epsilon}_{j+1} , \frac{\rho \tilde{\epsilon}_{j+1}^2}{\|g_{j+1} -s_{j+1} \|^2} \right\},
\end{align}
where $\tilde{\epsilon}_{j+1} = f(z_{j+1}) + \frac{m}{2}\|z_{j+1}-x_k\|^2 - \tilde{f}_j(z_{j+1})$ is the approximated gap at $z_{j+1}$, $\eta_j$ and $z_{j+1}$ are defined in \cref{eq:lower-approximation-j}, $s_{j+1} : = \rho(x_k - z_{j+1}) \in \tilde{f}_j(z_{j+1})$, and $g_{j+1} - m( z_{j+1} - x_k  ) \in \partial f(z_{j+1})$. 
If the criterion \cref{eq:descent-condition-main} is not satisfied at $z_{j+1}$, the improvement in \cref{eq:strict-improvement} can be further quantified as
\begin{align} \label{eq:strict-improvement-2}
    \eta_{j+1} \geq \eta_j + \frac{1}{2} \min \left \{(1-\beta) \tilde{\Delta}_j, \frac{(1 - \beta)^2 \rho \tilde{\Delta}_j^2}{  \|g_{j+1} -s_{j+1} \|^2} \right\},
\end{align}
where $\tilde{\Delta}_j := f(x_{k}) - \eta_j$ is called the approximated proximal gap.
\end{lemma} 
This result follows standard analysis in convex bundle methods (see e.g., \cite{du2017rate,kiwiel2000efficiency} and  \cite[Section 5.3]{diaz2023optimal}). In particular, when $m = 0$, \Cref{lemma:null-step-improvement} becomes the same as \cite[Section 5.3]{diaz2023optimal}. We adapt and generalize these classical analyses to the $m$-weakly convex case; see \cref{subsection:null-step-improvement} for proof details. 

\begin{algorithm}[t]
\caption{\texttt{ProxDescent($x_k,\beta,\rho$)}}\label{alg:Proxi-descent-subproblem}
\begin{algorithmic}[1]
\Require $x_{k} \in \RR^n$, $\beta \in (0,1)$ and $\rho > 0$.
\State Let $j = 1$, $\tilde{f}_j= f(x_{k}) + \innerproduct{g_{j}}{\cdot -x_{k}}$, where $g_j \in \partial f(x_k)$, and $z_j = x_k$; 
\State Compute $z_{j+1} = \argmin_{x}\; \tilde{f}_{j}(x) + \frac{\rho}{2}\|x - x_k\|^2$; 
\While{$f(x_k) - \left(f(z_{j+1}) + \frac{m}{2}\|z_{j+1} - x_k\|^2\right)  < \beta \left(f(x_k) - \tilde{f}_j(z_{j+1})\right) $}
    
    \State Construct $\tilde{f}_{j+1}$ satisfying \cref{eq:lower-approximation-1}, \cref{eq:requirement-2} and \cref{eq:requirement-3}; 
    \State Compute $z_{j+2} = \argmin_{x}\; \tilde{f}_{j+1}(x) + \frac{\rho}{2}\|x - x_k\|^2$;  
    \State $j = j + 1$;
\EndWhile
 \State \Return $z_{j+1}$;  
\end{algorithmic}
\end{algorithm}

We can thus provide an implementable procedure for \texttt{ProxDescent($x_k,\beta,\rho$)} in \Cref{alg:Proxi-descent-subproblem}. From \Cref{lemma:null-step-improvement}, we can bound the number of iterations required before \Cref{alg:Proxi-descent-subproblem} terminates.  

\begin{proposition}
    \label{lemma:null-step}
    With the same setup in \Cref{lemma:null-step-improvement}, let $\Delta_{k}= f(x_k) - f_{\alpha}(x_k)$ denote the true proximal gap.
    Then \Cref{alg:Proxi-descent-subproblem} takes at most 
     \begin{equation} \label{eq:null-steps}
        T_k \leq  \frac{8G^2_{k}}{(1 - \beta)^2 \rho \Delta_{k}}
     \end{equation}
    iterations to terminate, where $G_k = \sup_{1\leq j \leq T_k} \|g_{j}\|$ is the largest subgradient norm in \Cref{alg:Proxi-descent-subproblem}. In other words, \Cref{alg:Proxi-descent-subproblem} finds an iterate $z_{j+1}$ satisfying  \cref{eq:descent-condition-main} within at most $T_k$ iterations. 
\end{proposition}
\begin{proof}    
    We here outline a proof sketch, and the proof details are postponed to \Cref{subsection:standard-bundle-proof}.  Note that \cref{eq:strict-improvement-2} estimates the improvement of $\eta_j$ in terms of the approximated proximal gap $\tilde{\Delta}_j$. We have the following  key inequality (see \Cref{lemma:key-improvement})
    \begin{align*}
        \frac{1}{2\rho }\|s_{j+1}\|^2 \leq \tilde{\Delta}_j \leq \tilde{\Delta}_1  \leq \frac{1}{2\rho}\|g_1\|^2,
    \end{align*}
    which is a consequence of the assumption \cref{eq:requirement-3}, implying $\frac{2 \rho \tilde{\Delta}_j }{\|g_1\|^2} \leq 1$ and $ \|s_{j+1}\|^2 \leq  \|g_1\|^2.$ By the definition of $G_k$ and some simple algebras, the improvement in \cref{eq:strict-improvement-2} can be lower bounded~as
    \begin{align*}
        \frac{1}{2} \min \left \{(1-\beta) \tilde{\Delta}_j, \frac{(1 - \beta)^2 \rho \tilde{\Delta}_j^2}{  \|g_{j+1} -s_{j+1} \|^2} \right\} \geq   \frac{(1 - \beta)^2 \rho \tilde{\Delta}_j^2}{ 8  G_k^2}.
    \end{align*}
        Hence, multiplying $(-1)$ on both sides of \cref{eq:strict-improvement-2} and adding $f(x_k)$ on the both sides lead to 
\begin{align*}
    \tilde{\Delta}_{j+1} \leq \tilde{\Delta}_j - \frac{(1 - \beta)^2 \rho \tilde{\Delta}_j^2}{8 G_{k}^2}.
\end{align*}
    This is a simple scalar sequence. Recursively, the above inequality yields
    $$
        \tilde{\Delta}_{j} \leq  \frac{ 8 G_{k}^2 }{(1 - \beta)^2 \rho} \times \frac{1}{j}, \qquad \; \forall j \geq 1.  
    $$
    Since $\Delta_k \leq \tilde{\Delta}_{j}$ for all $j \geq 1$, we know that the number of iterations that \cref{eq:descent-condition-main} fails is at most $\frac{8G^2_{k}}{(1 - \beta)^2 \rho \Delta_{k}}.$ This completes the proof.
\end{proof}
We here remark that the quantity $G_k$ in \Cref{lemma:null-step} encapsulates both the weakly convex parameter $m$ and the proximal center $x_k$, since it is related to the element in the subdifferential $\partial \left (f(\cdot) + \frac{m}{2}\|\cdot- x_{k}\|^2 \right)$. As we will detail later, $G_k$ can be further upper-bounded for certain classes of functions. In particular, when $f$ is Lipschitz continuous, $G_k$ can be bounded as a constant (see \cref{lemma:bounds-on-G-k}); when $f$ is smooth, $G_k$ can be bounded as $\bigO(\sqrt{\Delta_k})$ (see \cref{proposition:Upper-bound_G}). 

For a bounded  $G_k$, \Cref{lemma:null-step} guarantees that the number of subgradient and function evaluations (iterations) of \texttt{ProxDescent($x_k,\beta,\rho$)} in \Cref{alg:Proxi-descent-subproblem} is inversely related to the proximal gap $\Delta_{k}$. 
If this proximal gap $\Delta_{k}$ is larger, \texttt{ProxDescent($x_k,\beta,\rho$)} takes fewer iterations to terminate. We thus expect \texttt{ProxDescent($x_k,\beta,\rho$)} only takes very few iterations to ensure \cref{eq:descent-condition-main} when the proximal center $x_k$ is far away from the stationary.

\subsection{Iteration complexity in terms of subgradient and function evaluations}

The proximal descent method in \Cref{alg:Proxi-descent} can be viewed as a double-loop algorithm: the outer loop follows an inexact subgradient update that mimics the proximal iterations, and the inner loop uses bundle approximations (i.e., \Cref{alg:Proxi-descent-subproblem}) to achieve the desired inexact update. 

\begin{subequations}\label{eq:descent-steps-function-evaluations}
To reach an $(\eta,\epsilon)$-inexact stationary point for minimizing an $m$-weakly convex  functions, we let 
\begin{equation} \label{eq:maximum-descent-steps}
K_{\max}(\eta,\epsilon) := \left\lceil \frac{2\alpha^2 (f(x_1) - f^\star)}{m + \beta \rho } \frac{1}{\eta^2} + \frac{(1-\beta)\left(f(x_1) -  f^\star\right)}{\beta}\frac{1}{\epsilon} + 1\right \rceil
\end{equation}
from \Cref{lemma:proximal-descent}. Then, together with \Cref{lemma:null-step}, \Cref{alg:Proxi-descent} takes at most 
\begin{equation} \label{eq:total-subgradient-evaluations}
\sum_{k=1}^{K_{\max}(\eta,\epsilon)} T_k \leq  \frac{8K_{\max}(\eta,\epsilon)}{(1 - \beta)^2 \rho} \times \max_{k=\{1,\ldots, K_{\max}(\eta,\epsilon) \}} \frac{G^2_{k}}{\Delta_{k}}. 
\end{equation}
subgradient and function evaluations to find an $(\eta,\epsilon)$-stationary point. Our next task is to bound two quantities in \cref{eq:total-subgradient-evaluations}: 1) the largest subgradient norm $G_k$ and 2) the proximal gap $\Delta_k$.  
\end{subequations}

To bound $\Delta_{k}$, we observe that the magnitude of the proximal gap $\Delta_{k}$ naturally quantifies the stationarity of the center point $x_k$. The proof of the result below is given in \cref{lemma:prox-gap-consq}.
\begin{proposition}
    \label{proposition:implication-proximal-gap} 
    Consider a $m$-weakly convex function  $f:\RR^n \to \RR$, a proximal center $x_{k} \in \mathbb{R}^n$, and a proximal parameter $\rho > 0$. Let $\alpha = m + \rho$ and denote the proximal gap $\Delta_{k} = f(x_k) - f_{\alpha}(x_k)$.
    % , where $\eta_k^\ast$ denotes the proximal value in \cref{eq:proximal-gap}, 
    Then the point $x_{k}$ is an $(\eta,\Delta_k)$-inexact stationary point with $\eta = \sqrt{2\rho \Delta_k}$, and it is also a $(\delta,\alpha)$-Moreau stationary point with $\delta = \sqrt{{2\alpha^2\Delta_k}/{\rho}}$.
\end{proposition}

\cref{proposition:implication-proximal-gap} implies that if we have $\Delta_{k} \leq \min\{\frac{ \eta^2}{2\rho } ,\epsilon\} $, then \Cref{alg:Proxi-descent} has found an $(\eta,\epsilon)$-inexact stationary point. Thus, the following corollary is immediate.

\begin{corollary} \label{corollary:proximal-gap}
    Consider \Cref{alg:Proxi-descent} that minimizes a $m$-weakly convex function  $f$. For any iteration $x_k$ that is not a $(\eta,\epsilon)$-inexact stationary point, we must have 
    $$ 
    \Delta_{k} > \min\left\{\frac{ \eta^2}{2\rho} ,\epsilon\right\}.  
    $$ 
\end{corollary}

We can further bound the largest subgradient norm $G_k$.

\begin{lemma} \label{lemma:bounds-on-G-k}
Let $ f : \mathbb{R}^n \to \mathbb{R}$ be an $ L $-Lipschitz and $m$-weakly convex function. Consider the proximal descent method in \Cref{alg:Proxi-descent}, using the subroutine in \Cref{alg:Proxi-descent-subproblem} with parameters $ \beta \in (0,1) $ and $ \rho > 0 $. Then the maximum subgradient norm $G_k = \sup_{1\leq j \leq T_k} \|g_{j}\|$ satisfies
\[
G_k \leq \left(1 + m/\rho  \right)L, \quad \forall k \geq 1.
\]
\end{lemma}

The proof is presented in \cref{subsection:proof:bounds-on-G-k}.
Combining \Cref{lemma:bounds-on-G-k}, \cref{corollary:proximal-gap} with \Cref{eq:descent-steps-function-evaluations} directly leads to the following complexity guarantee of \Cref{alg:Proxi-descent}, which is one main theoretical convergence in this paper.  

\begin{theorem}[Main result]
    \label{theorem:main-result}
    % Given an $L$-lipschitz and $m$-weakly convex function $f:\RR^n \to \RR $, 
    Consider the $m$-weakly convex problem \cref{eq:main-problem}.~Suppose the function $f$
    is further $L$-Lipschitz. Then the proximal descent method in \Cref{alg:Proxi-descent} with subroutine in \Cref{alg:Proxi-descent-subproblem} and parameters $\beta \in (0,1)$, $\rho > 0$, and $\alpha = m + \rho$, takes at most 
    \begin{equation} \label{eq:descent-bound}
       \frac{8K_{\max}(\eta,\epsilon)}{(1 - \beta)^2 \rho} \times \left (1 + \frac{m}{\rho} \right)^2L^2 \times \max\left\{ \frac{2\rho}{ \eta^2},\frac{1}{\epsilon}  \right\}
    \end{equation}
    subgradient and function evaluations\footnote{Our count of subgradient and function evaluation is under the assumption that the subroutine \Cref{alg:Proxi-descent-subproblem} uses the essential model $\tilde{f}_{k+1}(\cdot) = \max\{f(z_{j+1}) + \frac{m}{2}\|z_{j+1} - x_{k}\|^2 + \innerproduct{g_{j+1}}{\cdot - z_{j+1}} ,\tilde{f}_j(z_{j+1}) + \innerproduct{s_{j+1}}{\cdot - z_{j+1}}\}$ that satisfies \cref{eq:lower-approximation-1},\cref{eq:requirement-2}, and \cref{eq:requirement-3} at every iteration. Thus, every iteration in \Cref{alg:Proxi-descent-subproblem} requires only one subgradient and one function evaluation.} to find an $(\eta,\epsilon)$-inexact stationary point, where $K_{\max}(\eta,\epsilon)$ is given~in~\cref{eq:maximum-descent-steps}.  
\end{theorem}

In summary, the complexity of finding a $(\eta,\epsilon)$-inexact stationary point for \Cref{alg:Proxi-descent} takes at most 
$$
\bigO \left(  \left ( \frac{1}{\eta^2} + \frac{1}{\epsilon} \right ) \max \left\{\frac{1}{\eta^2},\frac{1}{\epsilon}\right\} \right )
$$
subgradient and function evaluations. Combining this with \Cref{lemma:stationary-conversion-R2M-2-main}, we have the following complexity of finding an approximate Moreau stationary point. 
\begin{corollary}[Moreau stationary]
    \label{corollary:main-result}
    With the same setup in \Cref{theorem:main-result}, the proximal descent method in \Cref{alg:Proxi-descent} takes at most $\mathcal{O}(1/\delta^4)$ subgradient and function evaluations to find a point $x$ satisfying $\|\nabla_{\alpha} f(x)\| \leq \delta$ with $\alpha > m$, i.e., an  $(\delta,\alpha)$-Moreau stationary point.
\end{corollary}

The proof can be found in \cref{subsection:corollary:main-result}. The complexity guarantees in \cref{theorem:main-result,corollary:main-result} are comparable to the best-known rates in the literature, while providing additional advantages. A detailed comparison with closely related works is presented in \cref{subsection:comparison}.

\subsection{Interpreting \Cref{alg:Proxi-descent} as a proximal bundle algorithm}
\label{subsection:re-interpretation}
In this section, we re-interpret \Cref{alg:Proxi-descent} as a standard proximal bundle method for minimizing $m$-weakly convex functions. The key idea of the proximal bundle method is to mimic the conceptual proximal update \cref{eq:PPM-weakly}, which is also the same as \Cref{eq:PPM-weakly-split}. We approximate the convex part $f(\cdot) + \frac{m}{2}\|\cdot - x_k\|^2$ in \cref{eq:PPM-weakly-split} by a lower bound $f_k$ (i.e., $f_k \leq f + \frac{m}{2}\|\cdot - x_k\|^2$) and solve the following subproblem 
\begin{align}
    \label{eq:PBM-subproblem-weakly}
    z_{k+1} = \argmin_{x \in \RR^n}\; f_k(x) + \frac{ \rho}{2}\|x - x_k\|^2. 
\end{align}
To decide if $z_{k+1}$ is a good solution to the original problem \Cref{eq:PPM-weakly-split}, we check the following criterion
\begin{equation}
    \begin{aligned}
    \label{eq:descent-weakly}
        \beta (f(x_k) - f_k(z_{k+1})) \leq f(x_k) -  \left ( f(z_{k+1}) + \frac{m}{2}\|z_{k+1}- x_k\|^2 \right ),
    \end{aligned}
\end{equation}
where $\beta \in (0,1)$, which is the same as \cref{eq:descent-condition-main}. If \cref{eq:descent-weakly} holds, we set $x_{k+1} = z_{k+1}$ (\textit{descent} step). Otherwise, we set $x_{k+1} = x_k$ (\textit{null} step).  
Regardless of a descent or null step, we construct a new function $f_{k+1}$ satisfying the conditions below.
\begin{assumption} \label{assumption:approximation-conditions}
The convex function $f_{k+1}$ satisfies three conditions:
\begin{enumerate}
    \item \textbf{Lower approximation:} Global convex lower approximation, 
    \begin{equation*} 
        f_{k+1}(x) \leq f(x) +\frac{m}{2}\|x - x_{k+1}\|^2, \quad \forall x \in \mathbb{R}^n.
    \end{equation*}
    \item  \textbf{Subgradient lower bound:} We have  
    \begin{equation*} 
        f_{k+1}(x) \geq  f(z_{k+1})  +  \frac{m}{2}\|z_{k+1} - x_{k+1}\|^2 + \innerproduct{g_{k+1}}{x-z_{k+1}},\quad \forall x \in \mathbb{R}^n,
    \end{equation*}
    where $g_{k+1}$ satisfies $g_{k+1}  - m(z_{k+1} - x_{k+1}) \in  \partial f(z_{k+1})$.  
        \item \textbf{Aggregation from the past approximation:} If \cref{eq:descent-weakly} fails, then we require
        \begin{equation*} 
            f_{k+1}(x) \geq f_k(z_{k+1}) + \innerproduct{s_{k+1}}{x-z_{k+1}}, \quad \forall x \in \mathbb{R}^n,
        \end{equation*}
        where $s_{k+1} = \rho(x_k - z_{k+1}) \in \partial f_{k}(z_{k+1})$.
\end{enumerate} 
\end{assumption}

When $m = 0$, \Cref{assumption:approximation-conditions} is the same as the classical proximal bundle method  \cite{kiwiel2000efficiency,du2017rate,diaz2023optimal,lemarechal1981bundle,lemarechal1994condensed,kiwiel1990proximity}. We here have adapted the classical bundle idea to the case of $m$-weakly convex function by adding the convexification term $\frac{m}{2}\|\cdot - x_k\|^2$. According to \Cref{assumption:approximation-conditions}, in principle, we can use only a maximum of two linear functions to construct the lower approximation, i.e., 
$$
f_{k+1}(\cdot) = \max\left\{f(z_{k+1}) \! +\!  \frac{m}{2}\|z_{k+1} - x_k\|^2 + \innerproduct{g_{k+1}}{\cdot-z_{k+1}}, f_k(z_{k+1})\! +\! \innerproduct{s_{k+1}}{\cdot-z_{k+1}}\right\}.
$$
In this case, the proximal update \cref{eq:trial-point-1} with $f_{k+1}$ is extremely simple and admits an analytical formula (see \cref{subsection:analytical-sol}).

We list the proximal bundle method for minimizing a weakly convex function in \Cref{alg:PBM-weakly}. Note that \Cref{alg:PBM-weakly} updates the proximal center $x_k$ to a new point $z_{k+1}$ only when the condition \cref{eq:descent-weakly} holds, which is the same as \cref{eq:descent-condition-main}. The three model update assumptions in \cref{assumption:approximation-conditions} are identical to  \cref{eq:lower-approximation-1}, \Cref{eq:requirement-2}, and \cref{eq:requirement-3}, respectively. Therefore,  the proximal bundle method in \Cref{alg:PBM-weakly} is just a re-interpretation of the proximal descent method in \Cref{alg:Proxi-descent}. In particular, the indices in \cref{subsection:bundle-step} drop the dependence on $k$ of the center point $x_k$, whereas the \Cref{alg:PBM-weakly} makes the dependence explicit. Hence,  the seemingly single-loop \Cref{alg:PBM-weakly} also has a double-loop interpretation as \Cref{alg:Proxi-descent}. Consequently, \Cref{alg:PBM-weakly} enjoys the same convergence guarantees in \cref{theorem:main-result}. The count for the subgradient and function evaluations in \cref{theorem:main-result,corollary:main-result} also becomes the iteration complexity for \Cref{alg:PBM-weakly}.

\begin{algorithm}[t]
\caption{Proximal bundle method (\PBM{}) for $m$-weakly convex functions}
\label{alg:PBM-weakly}
\begin{algorithmic}
\Require $x_1 \in \RR^n, T > 0, \rho > 0,$ weakly convex parameter $m \geq 0$, $\beta \in (0,1)$
\For{$k=1,2, \ldots, T$}
    \State Compute 
    $
         z_{k+1}= \argmin_{y}  f_k(y) + \frac{\rho}{2}\|y - x_k\|^2;
    $
    \If {$\beta (f(x_k) - f_k(z_{k+1})) \leq f(x_k) -  \left ( f(z_{k+1}) + \frac{m}{2}\|z_{k+1}- x_k\|^2 \right )$ }
        \State Set $x_{k+1} = z_{k+1}$; \Comment{\textit{Descent step}}
    \Else
        \State Set $x_{k+1} = x_k$;  \Comment{\textit{Null step}}
    \EndIf
    \State Construct $f_{k+1}$ that approximates $f(\cdot) + \frac{m }{2}\|\cdot - x_{k+1}\|^2$ satisfying \cref{assumption:approximation-conditions};
\EndFor
\end{algorithmic}
\end{algorithm}

\subsection{Connections and complexity comparison} \label{subsection:comparison}

The use of proximal bundle updates for solving weakly convex optimization problems is not new. This idea can be traced back to  \cite{hare2010redistributed,hare2009computing}, and a recent work \cite{liang2023proximal} also explored a similar idea. In particular, the works \cite{hare2010redistributed,hare2009computing} focus on minimizing a broader class of functions, namely lower-$C^2$ functions, and therefore do not provide iteration complexity guarantees. In this work, we establish iteration complexity results in \Cref{theorem:main-result,corollary:main-result} finding an inexact stationary point. The recent work \cite{liang2023proximal} modifies the classical proximal bundle update to obtain an inexact solution to the proximal subproblem \cref{eq:PPM-weakly}. In contrast, as established in \Cref{subsection:re-interpretation}, our \Cref{alg:Proxi-descent} retains the classical proximal bundle update and can be viewed as a natural generalization of the proximal bundle method \cite{lemarechal1994condensed,kiwiel2000efficiency,diaz2023optimal} from convex optimization to nonconvex optimization.

We here compare our main complexity results in \Cref{theorem:main-result,corollary:main-result} with several closely related works \cite{davis2019stochastic,davis2019proximally,zhang2020complexity,davis2022gradient,burke2020gradient,liang2023proximal}. First, in terms of finding an $(\delta,\alpha)$-Moreau stationary point, our result in \cref{corollary:main-result} is comparable to the subgradient method in \cite[Theorem 3.1]{davis2019stochastic} (see \cref{theorem:weakly-subgradient}), both of which require $\bigO(1/\delta^4)$  iteration complexity. However, the subgradient method is less practical as 1) its stepsize depends on the total number of iterations and 2) the stationarity measure $\min_{k = 1,\ldots,T} \|\nabla f_{\alpha}(x_k)\|$ is implicit and not easily computable. 
In contrast, \Cref{alg:Proxi-descent} is guaranteed to converge for any parameter $\alpha > m$ (or $\rho>0$), and relies on two explicit and easily computable quantities $\tilde{g}_{k+1}$ and $\epsilon_{k+1}$ in \cref{eq:inexactness-trial-point} to certify whether an iterate is $(\eta,\epsilon)$-stationary. 

Second, in terms of finding an $(\eta,\epsilon)$-inexaxt stationary point, our \cref{theorem:main-result}  matches the state-of-the-art complexity result in \cite[Theorem 3.3]{liang2023proximal}, both of which require $\bigO \left(  \left ( \frac{1}{\eta^2} + \frac{1}{\epsilon} \right ) \max \left\{\frac{1}{\eta^2},\frac{1}{\epsilon}\right\} \right )$ function value and subgradient evaluations. However, as we have clarified above,  our \Cref{alg:Proxi-descent} is more aligned with the classical proximal bundle method \cite{kiwiel2000efficiency,diaz2023optimal} that uses sufficient objective descent in \Cref{eq:descent-condition-main}. In contrast, the algorithm in \cite{liang2023proximal} uses a different and potentially more complicated algorithm structure.

Third, our \Cref{alg:Proxi-descent} is developed within the framework of the inexact proximal point method. A related algorithm in this framework is the proximally guided stochastic subgradient method (PGSG) in \cite{davis2019proximally}. Although PGSG was originally designed for stochastic optimization, its deterministic variant can be interpreted as applying the subgradient method as a subroutine to inexactly solve the proximal subproblem  \cref{eq:PPM-weakly}. In contrast, our \Cref{alg:Proxi-descent} adapts the proximal bundle idea to inexactly solve the proximal subproblem. 
While both algorithms share the same algorithmic framework, the criteria used by their respective subroutines to certify an inexact solution to \cref{eq:PPM-weakly} differ. Specifically, PGSG uses a \textit{predetermined} number of subgradient iterations, determined by the target accuracy, whereas our \Cref{alg:Proxi-descent} applies the dynamic test in \cref{eq:descent-condition-main} to adaptively determine termination. This makes \Cref{alg:Proxi-descent} more flexible and potentially more efficient in practice. Our numerical experiments in \Cref{section:numerics} further confirm this advantage.

Finally, two other classes of algorithms that can minimize $m$-weakly convex functions are the Interpolated Normalized Gradient Descent (INGD) method in \cite{davis2022gradient,zhang2020complexity} and the gradient sampling method in \cite{burke2020gradient}.  
One key difference is that we use the notion of inexact subdifferential in \cref{eq:inexact-subdifferential}, whereas the INGD and the gradient sampling approximate with probability the notion of \textit{Goldstein subdifferential}
$\hat{\partial}_{\delta} f(x) :=\mathrm{conv} \left ( \bigcup_{y \in \mathbb{B}_{\delta}(x)} \partial f(x) \right )$
where $\delta > 0$ and $\mathbb{B}_{\delta}(x):=\{y \in \RR^n \mid \|y-x\|\leq \delta\}$. As a result, our \Cref{alg:Proxi-descent} is deterministic, while the INGD and the gradient sampling are probabilistic. Moreover, while the INGD and the gradient sampling algorithms work for a more general class of nonconvex and nonsmooth functions, their convergence guarantees are much weaker than \Cref{theorem:main-result,corollary:main-result}.

%% file: 5-ImprovedRate.tex
\section{Improved rates with smoothness and quadratic growth}
\label{section:improved-rate}

The proximal descent method in \Cref{alg:Proxi-descent} (and similarly, the proximal bundle method in \Cref{alg:PBM-weakly}) applies to general weakly convex functions. In this section, we focus on two widely studied subclasses: $M$-smooth functions and functions exhibiting quadratic growth. Both can be nonconvex; still, we prove that the convergence guarantees of \Cref{alg:Proxi-descent} naturally improve for these two special cases, without requiring any modification to the algorithm. \Cref{tab:summary} lists our established convergence rates under different nonconvex settings.

\subsection{Nonconvex $M$-smooth functions}

We here establish an improved convergence guarantee for \Cref{alg:Proxi-descent} when the objective function is $M$-smooth, meaning that its gradient map $ \nabla f $ is $M$-Lipschitz continuous. Under this condition, the number of steps required for \Cref{alg:Proxi-descent-subproblem} to generate a trial point satisfying \cref{eq:descent-condition-main} can be bounded by a constant. 
In particular, we have the following proposition. 

\begin{proposition}
    \label{proposition:Upper-bound_G}
   With the same setup in \Cref{lemma:null-step-improvement}, suppose the function $f$ is further $M$-smooth. Then, the quantity $G_{k}$ can be bounded as 
    $$
        G_k \leq \left (1+ \frac{M}{\rho} \right ) \sqrt{2 (M+ \alpha ) \Delta_k}.
    $$
    Consequently, the ratio $\frac{G_k^2}{\Delta_k}$ in \cref{eq:total-subgradient-evaluations} is bounded as $\frac{G^2_{k}}{\Delta_{k}} \leq  \left (1+ \frac{M}{\rho} \right )^22 (M+ \alpha ).$ 
\end{proposition}
The proof is given in \cref{subsection:Upper-bound_G}, which is adapted from \cite[Section 5.2]{diaz2023optimal}.
\cref{proposition:Upper-bound_G} means that for each iteration of \Cref{alg:Proxi-descent}, the subroutine $\texttt{ProxDescent}(x_k,\beta,\rho)$ can be bounded by the same constant, independent of $\eta$ and $\epsilon$. Consequently, from \cref{eq:total-subgradient-evaluations}, we see that our proximal descent method in \Cref{alg:Proxi-descent} takes at most 
     \begin{align}
        \label{eq:iteration-smooth}
         \frac{8K_{\max}(\eta,\epsilon)}{(1 - \beta)^2 \rho}  \left (1+ \frac{M}{\rho} \right )^22 (M+\alpha)
     \end{align}
subgradient and function evaluations to find an $(\eta,\epsilon)$-inexact stationary point. 

In addition, we can further convert $(\eta,\epsilon)$-inexact stationarity to the magnitude of the gradient $\|\nabla f(x)\|$ for smooth functions.
\begin{lemma}
\label{lemma:stationary-conversion-R2G-main-text}
    Suppose $f:\RR^n \to \RR \cup \{+\infty\}$ is $m$-weakly convex and $M$-smooth. Let $\alpha > m$, $x \in \RR^n$, and $\epsilon \geq 0$. If $x$ is an $(\eta,\epsilon)$-inexact stationary point, then $$\|\nabla f(x)\| \leq \left ( 1+ \frac{M}{\alpha} \right )\alpha \left ( \frac{2}{\lambda}\eta + \sqrt{\frac{2}{\lambda}\epsilon} \right),$$
    where $\lambda = \alpha - m > 0$. 
\end{lemma}
The proof of \cref{lemma:stationary-conversion-R2G-main-text} is given in \cref{lemma:stationary-conversion-R2G-appendix}. It is now clear that choosing $\epsilon \leq  \frac{2}{\lambda}\eta^2$ and $\eta \leq  \frac{\delta \lambda}{4 (1+M/\alpha) \alpha}$ with $\delta \geq 0$ implies $\|\nabla f(x)\| \leq \delta$.
Hence, by choosing $\rho = \lambda$ in \cref{lemma:stationary-conversion-R2G-main-text}, 
we arrive at the following theorem.
\begin{theorem} \label{theorem:smooth-functions}
    Consider the $m$-weakly convex problem \cref{eq:main-problem}. Suppose the function $f$ is further $M$-smooth. Then the proximal descent method in \Cref{alg:Proxi-descent} with subroutine in \Cref{alg:Proxi-descent-subproblem} and parameters $\beta \in (0,1)$, $\rho > 0$, and $\alpha = m + \rho$, takes at most 
    \begin{equation*} 
   K_{\max}\left(\frac{\rho \delta }{4 (1+M/\alpha) \alpha}, \frac{2 \rho \delta^2}{(4(1+M/\alpha)\alpha)^2} \right) \frac{8}{(1 - \beta)^2 \rho}  \left (1+ \frac{M}{\rho} \right )^22 (M+\alpha)
    \end{equation*}
    subgradient and function evaluations to find an iterate $x_k$ satisfying $\|\nabla f(x_k)\| \leq \delta$.
\end{theorem}
\begin{proof}
    Let $\lambda = \rho$, $\epsilon \leq  \frac{2}{\rho}\eta^2$, and $\eta \leq \frac{\delta \rho}{4 (1+M/\alpha) \alpha}$ in \cref{lemma:stationary-conversion-R2G-main-text}. Plugging in the choice of $\epsilon$ and $\eta$ in \Cref{eq:iteration-smooth} leads to the desired complexity in \cref{theorem:smooth-functions}.
\end{proof}

In summary, when the function $f$ becomes $M$-smooth, the complexity of subgradient and function evaluations to find an iterate $x_k$ satisfying $\|\nabla f(x_k)\| \leq \delta$ improves to $\bigO\left ( 1/ \delta^2 \right)$, which is the same as the complexity result for gradient descent with the constant step size $1/M$ \cite[Section 1.2.3]{nesterov2018lectures}. It is worth pointing out that the complexity of $\bigO\left (1/ \delta^2\right)$ is achieved by \Cref{alg:Proxi-descent} for any step size parameter $\rho > 0$. In contrast, the gradient descent needs to know the smoothness parameter $M$. Nevertheless, we note that  \Cref{alg:Proxi-descent} requires the weak convexity parameter $m$, whereas the gradient descent does not.

\subsection{Nonconvex functions with quadratic growth}
\label{subsection:QG}
In convex nonsmooth optimization, algorithms generally have an improved convergence rate under the quadratic growth property, i.e., there exists $\muq > 0$ such that
    \begin{align}
        \label{eq:QG}
        f(x) - \min_x f(x) \geq \frac{\muq}{2}\cdot \Dist^2(x,S), \; \forall x \in \RR^n,
    \end{align}
where $S  = \argmin_x  f(x)$ is the set of optimal solutions (assumed nonempty). For example, the convergence of the cost value gap of the subgradient method with Polyak stepsizes \cite{polyak1969minimization} and the proximal bundle method \cite[Theorem 2.3]{diaz2023optimal} becomes $\bigO(1/\epsilon)$; the convergence rate of the proximal point method improves to $\bigO(\log(1/\epsilon))$ \cite[Theorem 4.2]{liao2024error}. In contrast, for nonconvex problems, it remains unclear whether the quadratic growth condition leads to faster convergence for the proximal bundle method. Existing results focused on either convex optimization \cite{kiwiel2000efficiency,du2017rate,diaz2023optimal,liang2021proximal,lemarechal1994condensed,lemarechal1981bundle} or nonconvex optimization without any growth assumptions \cite{liang2023proximal,hare2009computing}. 

We here analyze \Cref{alg:Proxi-descent}  under the assumption that the objective function satisfies quadratic growth \cref{eq:QG}. For technical reasons (see \cite{liao2024error}), we further assume that $\muq > m$, where $m$ is the weakly convex parameter. Under these assumptions, it is known that the function also satisfies the Polyak-{\L}ojasiewicz (PL) inequality \cite[Theorem 3.1]{liao2024error}, i.e., there exists a $\mup > 0$ such that 
\begin{align*}
    2 \mup (f(x) - \min_x f(x)) \leq \Dist^2(0,\partial f(x)), \; \forall x \in \RR^n.
\end{align*}
An immediate implication of the PL inequality is that every stationary point is a global minimizer. It is also known that, under these assumptions (i.e., weakly convex functions satisfying the PL inequality), the proximal point method in \Cref{alg:PPM} converges linearly in both the cost-value gap and the distance to the solution set \cite[Theorem 4.2]{liao2024error}. This fact gives the potential for improved convergence of \Cref{alg:Proxi-descent} in the same setting. 

We establish a linear convergence result for nonconvex and nonsmooth functions satisfying the quadratic growth condition \cref{eq:QG} with $\muq > m$, which significantly improves over the sublinear convergence in \cref{lemma:proximal-descent} for general weakly convex functions.

\begin{lemma}
    \label{lemma:Prox-descent-linear}
     Consider the $m$-weakly convex problem \cref{eq:main-problem}. Suppose the function $f$ satisfies quadratic growth \cref{eq:QG} with $\muq > m$. Then the proximal descent method in  \Cref{alg:Proxi-descent} generates a sequence of iterates $\{x_k\}$, satisfying
     $ f(x_{k+1}) - f^\star \leq \gamma (f(x_{k}) - f^\star ), \; \forall k \geq 1$, where $\gamma \in (0,1)$.
\end{lemma}
\cref{lemma:Prox-descent-linear} shows that the cost value gap converges to zero at a linear rate. Our proof requires some new analysis, which is provided in \cref{subsection:proof-Prox-descent-linear}. We next convert the cost value gap to the $(\eta,\epsilon)$-inexact stationarity. 

\begin{theorem} \label{lemma:proximal-descent-linear} 
     Consider the $m$-weakly convex problem \cref{eq:main-problem}. Suppose the function $f$ satisfies quadratic growth \cref{eq:QG} with $\muq > m$. Then the proximal descent method in \Cref{alg:Proxi-descent} with parameters $\beta \in (0,1)$, $\rho > 0$, and $\alpha = m + \rho$,  takes at most 
    \begin{equation}\label{eq:descent-bound-linear}
        T_{q,\max}(\eta,\epsilon):= (\log(\gamma^{-1}))^{-1} \log\left (  (f(x_1) - f^\star ) \max\left\{ \frac{2 \alpha^2}{(m+\beta \rho) \eta^2} ,\frac{(1-\beta)}{\beta\epsilon}\right\}\right) +1
    \end{equation} 
   iterations to find an $(\eta,\epsilon)$-stationary point, where $\gamma$ is the constant in \cref{lemma:Prox-descent-linear}.
\end{theorem}
\begin{proof}
    From \cref{lemma:Prox-descent-linear,lemma:descent-step}, we see that for all $k \geq 1$,
\begin{subequations}
 \label{eq:inexact-subgradient-bound-linear}
    \begin{align}
         \frac{m + \beta \rho }{\alpha}\times \frac{1}{2\alpha}\|\tilde{g}_{k+1}\|^2 &  \leq   f(x_k)  - f(x_{k+1})\leq f(x_k) - f^\star \leq \gamma^{k-1} (f(x_1) - f^\star), \\
            \left(\frac{\beta }{1 - \beta}\right)\epsilon_{k+1}&\leq  f(x_k) - f(x_{k+1}) \leq   f(x_k) -f^\star  \leq  \gamma^{k-1} (f(x_1) - f^\star).
    \end{align}
\end{subequations}
Fix two constants $\eta,\epsilon \geq 0$. From \Cref{eq:inexact-subgradient-bound-linear}, we can easily deduce that the conditions $\|\tilde{g}_{k+1}\| \leq \eta$ and $\epsilon_{k+1} \leq \epsilon$ will be satisfied after at most 
\begin{align*}
    (\log(\gamma^{-1}))^{-1} \log\left ( \frac{2 \alpha^2 (f(x_1) - f^\star)}{ (m+\beta \rho) \eta^2}\right) +1 \quad \text{and} \quad (\log(\gamma^{-1}))^{-1} \log\left ( \frac{(1-\beta) (f(x_1) - f^\star)}{ \beta \epsilon}\right) +1
\end{align*}
iterations, respectively. The proof is finished by considering these two iteration counts jointly.
\end{proof}

In summary, compared to the sublinear rate in \Cref{lemma:proximal-descent}, the iteration complexity of \Cref{alg:Proxi-descent} improves to 
\[
\bigO \left( \log \left( \max \left\{  \frac{1}{\eta^2} ,\frac{1}{\epsilon} \right\} \right) \right)
\]
when the objective function satisfies the quadratic growth condition with $\muq > m$. Note that \Cref{lemma:proximal-descent-linear} only counts the number of outer iterations in \Cref{alg:Proxi-descent}. We can further establish the complexity in terms of function and subgradient evaluations. 

\begin{theorem}
    \label{theorem:main-result-QG} 
    Consider the $m$-weakly convex problem \cref{eq:main-problem}. Suppose the function $f$ is $L$-Lipschitz and satisfies quadratic growth \cref{eq:QG} with $\muq > m$. Then the proximal descent method in \Cref{alg:Proxi-descent} with subroutine \Cref{alg:Proxi-descent-subproblem} and parameters $\beta \in (0,1)$, $\rho > 0$, $\alpha = m + \rho$, takes at most 
    \begin{equation*} 
    \frac{8 (1 + m/\rho)^2L^2}{(1 - \beta)^2 \rho } \frac{1}{ \min \left \{ \frac{1}{2}, \frac{ \muq - m}{4\rho} \right\} (1-m/\muq)(1-\gamma) \min \left \{ \frac{m + \beta     \rho}{2\alpha^2} , \frac{\beta}{1- \beta}  \right\} } \max \left \{\frac{1}{\eta^2},\frac{1}{\epsilon}\right \}
    \end{equation*}
    subgradient and function evaluations to find an $(\eta,\epsilon)$-inexact stationary point, where $\gamma$ is the constant in \cref{lemma:Prox-descent-linear}. 
\end{theorem}

The proof of \Cref{theorem:main-result-QG} is provided in \Cref{subsection:proof-theorem:main-result-QG}. Compared to the general weakly convex case in \Cref{theorem:main-result}, the total number of subgradient and function evaluations required to find an $(\eta, \epsilon)$-stationary point improves to 
\[
\mathcal{O} \left( \max \left\{ \frac{1}{\eta^2}, \frac{1}{\epsilon} \right\} \right)
\]
when the function satisfies the quadratic growth condition with $\muq > m$. Note that the bound in \cref{theorem:main-result-QG} has no dependency of $\log(1/\epsilon)$ or $\log(1/\eta)$. This is made possible by a more careful treatment than using \cref{eq:total-subgradient-evaluations} directly; see \Cref{subsection:proof-theorem:main-result-QG} for details. 

Similar to \cref{corollary:main-result}, \Cref{theorem:main-result-QG} also leads to a convergence result for the Moreau stationarity. The proof is provided in \cref{subsection:corollary:main-result-QG}.

\begin{corollary}
    \label{corollary:main-result-QG}
    With the same setup in \Cref{theorem:main-result-QG}, the proximal descent method in \Cref{alg:Proxi-descent} takes at most $\mathcal{O}(1/\delta^2)$ subgradient and function evaluations to find a point $x$ satisfying $\|\nabla_{\alpha} f(x)\| \leq \epsilon$ with $\alpha > m$, i.e., an  $(\delta,\alpha)$-Moreau stationary point.
\end{corollary}

If the function is further smooth, the function value and subgradient evaluations in the subroutine $\texttt{ProxDescent}(x_k,\beta,\rho)$ become a constant (see \cref{proposition:Upper-bound_G}). Hence, if the function is smooth and satisfies the conditions in \cref{lemma:proximal-descent-linear}, \Cref{alg:Proxi-descent} takes at most $\bigO(\log(1/\delta))$ function value and gradient evaluations to find an iterate $x_k$ satisfying $\|\nabla f(x_k)\|\leq \delta$. We summarize this result in the following theorem.
\begin{theorem} \label{theorem:smooth-functions-quadratic-growth}
   Consider the $m$-weakly convex problem \cref{eq:main-problem}. Suppose the function $f$ is $M$-smooth and satisfies quadratic growth \cref{eq:QG} with $\muq > m$. 
   Then, the proximal descent method in \Cref{alg:Proxi-descent} with parameters $\beta \in (0,1)$, $\rho > 0$, and $\alpha = m + \rho$, takes at most 
    \begin{align*}
         T_{q,\max}\left(\frac{\rho \delta }{4 (1+M/\alpha) \alpha}, \frac{2 \rho \delta^2}{(4(1+M/\alpha)\alpha)^2}\right)\frac{8}{(1 - \beta)^2 \rho } \left (1+ \frac{M}{\rho} \right )^22 (M+\alpha)
    \end{align*}    
    gradient and function evaluations to find an $x_k$ satisfying $\|\nabla f(x_k)\| \leq \delta$.
\end{theorem}
\begin{proof}
    From \cref{lemma:null-step,proposition:Upper-bound_G}, we know the number of iterations in each $\texttt{ProxDescent}(x_k,\beta,\rho)$ is at most 
    $$
    T:=\frac{8}{(1 - \beta)^2 \rho }  \left (1+ \frac{M}{\rho} \right )^22 (M+\alpha).
    $$ 
    Moreover, \cref{lemma:proximal-descent-linear} bounds the iteration number of \Cref{alg:Proxi-descent} to find an $(\eta,\epsilon)$-stationary point. Lastly, \cref{lemma:stationary-conversion-R2G-main-text} guarantees that a $(\eta,\epsilon)$-stationary point $x$ with $\epsilon = \frac{2}{\rho}\eta^2$ and $\eta = \frac{\delta \rho}{4 (1+M/\alpha) \alpha}$ imply $\|\nabla f(x)\| \leq \delta$. Hence, multiplying the numbers $T$ and \cref{eq:descent-bound-linear} yields the desired result.
\end{proof}

\begin{remark}[Approximate stationarity versus cost value gap]
    Throughout this paper, we have focused on finding an (approximate) stationary point (i.e., $ \Dist(0,\partial f_{\epsilon}(x)) \leq \eta$ and $ \|\nabla f_{\alpha}(x)\| \leq \delta$) for an $m$-weakly convex function (\cref{theorem:main-result,theorem:smooth-functions,theorem:main-result-QG,theorem:smooth-functions-quadratic-growth}). Our results hold for both the convex case (in which $m = 0$) and the nonconvex case (in which $m > 0$). In the convex case, our complexity bounds in \cref{theorem:main-result,theorem:main-result-QG} may appear worse than those of the classical proximal bundle method in \cite{diaz2023optimal}, which are $\bigO(1/\epsilon^2)$ for convex Lipschitz functions and $\bigO(1/\epsilon)$ under additional quadratic growth. However, we emphasize that these discrepancies are due to the change of measure from finding an approximate stationary point to an $\epsilon$-optimal solution (i.e., $f(x) - \min_x f(x) \leq \epsilon$). Our results are indeed comparable to the results in \cite{diaz2023optimal} when switching to the measure $f(x) - \min_x f(x) \leq \epsilon$. With $m$-weak convexity, quadratic growth \cref{eq:QG}, and $\muq > m$, it is possible to establish an iteration complexity of $\bigO(1/\epsilon)$ for the cost value gap $f(x) - \min_x f(x)$, even for the nonconvex case. \hfill $\square$ 
\end{remark}

%% file: 6-Experiments.tex
\section{Numerical Experiments}
\label{section:numerics}
In this section, we demonstrate the numerical behaviors of our proposed \Cref{alg:Proxi-descent}. 

\subsection{Experiment setup}

We run \Cref{alg:Proxi-descent} for two applications: Phase retrieval and blind deconvolution. 
In each application, we consider two settings:
\begin{enumerate}
    \item Varying different algorithm parameters $\beta \in (0,1)$ and $\rho > 0$ for \Cref{alg:Proxi-descent};
    \item Comparing with a deterministic version of the proximally guided stochastic subgradient method (PGSG)  in \cite{davis2019proximally}.
\end{enumerate}
The detail of the PGSG is listed in \Cref{alg:PGSG}. Both our \Cref{alg:Proxi-descent} and PGSG are developed under the framework of the proximal point method \Cref{alg:PPM}. The main difference is that our \Cref{alg:Proxi-descent} uses the proximal bundle idea in \cref{subsection:bundle-step} to solve the proximal map, whereas the PGSG uses the subgradient method to solve the proximal step.

\textbf{Practical stationarity measure.}  To measure the stationarity, we report the quantity $ \alpha \| x_{k+1} - x_k\| = \| \tilde{g}_{k+1}\| $ in \cref{eq:inexactness-trial-point-a}, which is a natural proxy of the stationarity, since the Moreau envelope is not easily computable in these applications.~Indeed, by \cref{eq:Moreau-property-3} and triangle inequality, we have
\begin{align*}
    \Dist(0,\partial f(\hat{x}_k)) \leq \alpha \|x_k - \hat{x}_k\| \leq \alpha  \|x_k - x_{k+1}\| + \alpha \|x_{k+1} -\hat{x}_k \|,
\end{align*}
where $\hat{x}_k = \argmin_{y} f(y) + \frac{\alpha}{2}\|y - x_k\|^2$. This means that $\alpha  \|x_k - x_{k+1}\|$ serves as a proxy of the near-stationary up to the error $\alpha \|x_{k+1} -\hat{x}_k \|$. When $x_{k+1}$ and $\hat{x}_k$ are close, the measure $\alpha \|x_k - x_{k+1}\|$ approximately indicates the stationarity (see \cite[Remark 1]{davis2019proximally} for related discussion). 

\textbf{Implementation.} For the implementation of \Cref{alg:Proxi-descent}, the function update in the step 4 of \Cref{alg:Proxi-descent-subproblem} is constructed using the essential model, i.e., 
\begin{align*}
    \tilde{f}_{j+1}(y) = \max \left \{ \tilde{f}_j(z_{j+1}) + \innerproduct{s_{j+1}}{ y- z_{j+1}}, f(z_{j+1}) + \frac{m}{2}\|z_{j+1} - x_{k}\|^2 + \innerproduct{ g_{j+1}}{y  - z_{j+1}} \right \}, \; \forall j \geq 1.
\end{align*}
Then, the subproblem $\argmin_{x}  \{ \tilde{f}_{j}(x) + \frac{\rho}{2}\|x - x_k\|^2\}$ admits an analytical solution (see \cref{subsection:closed-form-sol}).

\begin{algorithm}[t]
\caption{Deterministic proximally guided stochastic subgradient method (PGSG) \cite{davis2019proximally}}\label{alg:PGSG}
\begin{algorithmic}
\Require $x_1 \in \RR^n$, proximal parameter $\rho > 0$, maximum iteration $T$, step size sequence $\{\alpha_t\}$, maximum inner loop iteration $\{J_k\}$.
\For{$k=1,2, \ldots, T$}
    \State $x_{k+1} =  \texttt{SM}(x_k,\rho, \{ \alpha_t \}, J_k )$ from \Cref{alg:subgradient}
\EndFor
\end{algorithmic}
\end{algorithm}

\begin{algorithm}[t]
\caption{Subgradinet Method (SM)}\label{alg:subgradient}
\begin{algorithmic}
\Require $x_1 \in \RR^n$, proximal parameter $\rho$, maximum iteration $J$, step size sequence $\{\alpha_t\}$.
\For{$k=1,2, \ldots, J$}
    \State $v_k = g_k + \rho(x_k-x_1)$, \text{ where } $g_k \in \partial f(x_k)$
    \State $x_{k+1} = x_k + \alpha_k v_k$
\EndFor
\end{algorithmic}
\end{algorithm}

\subsection{Phase Retrieval}
\label{subsection:phase}
Our first numerical experiment considers a phase retrieval problem of the form \cite[Example 2.1]{davis2019stochastic}
\begin{align}
    \label{eq:Phase-retrieval}
    \min_{x \in \RR^d }\; \frac{1}{n} \sum_{i = 1}^n |\innerproduct{a_i}{x}^2 - b_i|,
\end{align}
where $a_i \in \RR^d$ and $b_i \in \RR $ for each $i =1,\ldots,n$ are problem data. The goal of phase retrieval is to find a point $x$ satisfying
 $ \innerproduct{a_i}{x}^2 \approx b_i, \; \forall i =1,\ldots,n.$ 

Let $f_i (\cdot) =|\innerproduct{a_i}{x}^2 - b_i|$ for $i = 1,\ldots,n$. Then $f_i$ can be written as $f_i(\cdot) = h_i(c_i(\cdot))$ 
where $h_i(\cdot) = |\cdot|$ is a convex function and $1$-Lipschitz continuous, and $c_i(\cdot) = \innerproduct{a_i}{\cdot}^2 - b_i$ is a $2 \|a_i\|^2$-smooth function. From \cite[Section 2.1]{davis2019stochastic}, it is known that the function $f_i$ is $2 \|a_i\|^2 $-weakly convex. Let $f = \frac{1}{n} \sum_{i = 1}^n f_i$. Then $f$ becomes $\frac{2}{n}\sum_{i=1}^n  \|a_i\|^2$-weakly convex. A detailed discussion can be found in \cref{Apx-section:numerics}. The subdifferential of $f$ can be computed as 
\begin{align*}
	\partial f(x) = \frac{1}{n}\sum_{i=1}^n \partial f_i(x), \quad \partial f_i(x) = 2 \innerproduct{a_i}{x} a_i \times \begin{cases}
	\mathrm{sign}(\innerproduct{a_i}{x}^2 - b_i), & \text{if } \innerproduct{a_i}{x}^2 \neq b_i\\
	[-1,1], & \text{otherwise}.
	\end{cases}
\end{align*}
We randomly draw the vectors $a_i$ for $i = 1,\ldots,n$ from a standard Gaussian distribution with mean $0$ and variance $I$, generate a ground truth $\bar{x} \in \RR^d$ from a unit sphere, and set $b_i  = \innerproduct{a_i}{\bar{x}}, i = 1,\ldots, n$.

\textbf{Experiment 1: Sensitivity to algorithm parameters $\beta \in (0,1)$ and $\rho >0$.}
We consider the phase retrieval problem \cref{eq:Phase-retrieval} with dimension $d=50$ and run \Cref{alg:Proxi-descent} under different parameter choices. Specifically, we test $\beta = 0.25, 0.5, 0.75$, which controls the inexactness of the proximal subproblem as defined in \cref{eq:bundle-update}, and $\rho = 0.1, 1, 10$ as the proximal parameter. For each configuration, the algorithm is executed for up to $10^{6}$ iterations, counting the inner loops of \Cref{alg:Proxi-descent-subproblem}.

The results are reported in \cref{fig:numerical-Phase-sensitivity}. Each subplot corresponds to a fixed $\beta$ with three trajectories for different values of $\rho$. The first row shows the residual $|\tilde{g}_{k+1}|^2 = \alpha^2|x_{k+1} - x_k|^2$ from \cref{eq:inexactness-trial-point-a}, where $\alpha = m + \rho$, while the second row reports the evolution of $\epsilon_{k+1}$ from \cref{eq:inexactness-trial-point-b}. As explained in \cref{eq:bundle-update}, $\beta$ controls the relative inexactness in \cref{eq:PPM-value-drop-subgradient-inexact}. In this experiment, when $\rho$ is fixed and $\beta$ varies, the convergence curves remain nearly unchanged, indicating that the algorithm is robust to the choice of $\beta$. In contrast, fixing $\beta$ and varying $\rho$ leads to noticeable differences. For instance, in the left column of \cref{fig:numerical-Phase-sensitivity}, the residuals $|\tilde{g}_{k+1}|^2$ and $\epsilon_{k+1}$ reach high accuracy levels of $10^{-8}$ and $10^{-6}$, respectively, for $\rho=10$, but only $10^{-1}$ and $10^{-2}$ for $\rho=0.1$.

\begin{figure}[t]
\centering
\begin{subfigure}[b]{0.32\textwidth}
\centering
{\includegraphics[width=1\textwidth]
{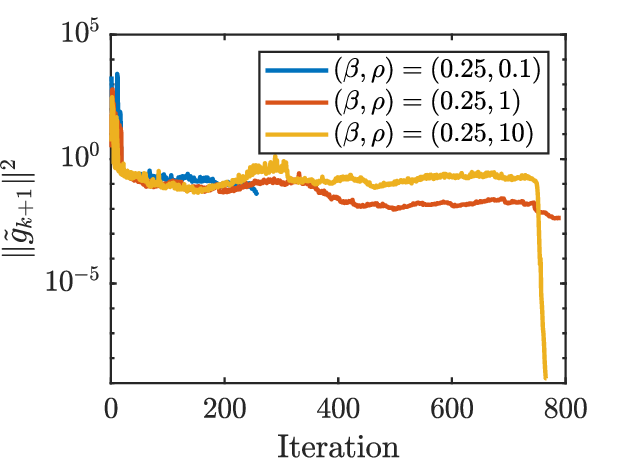}}
\end{subfigure}
\begin{subfigure}[b]{0.32\textwidth}
\centering
{\includegraphics[width=1\textwidth]
{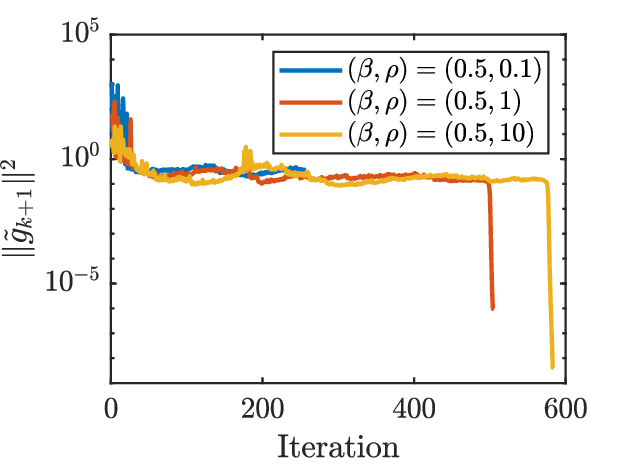}}
\end{subfigure}
\begin{subfigure}[b]{0.32\textwidth}
\centering
{\includegraphics[width=1\textwidth]
{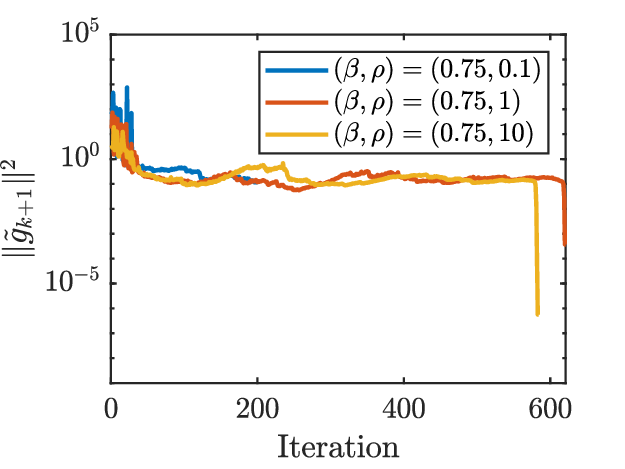}}
\end{subfigure}
\begin{subfigure}[b]{0.32\textwidth}
\centering  
{\includegraphics[width=1\textwidth]
{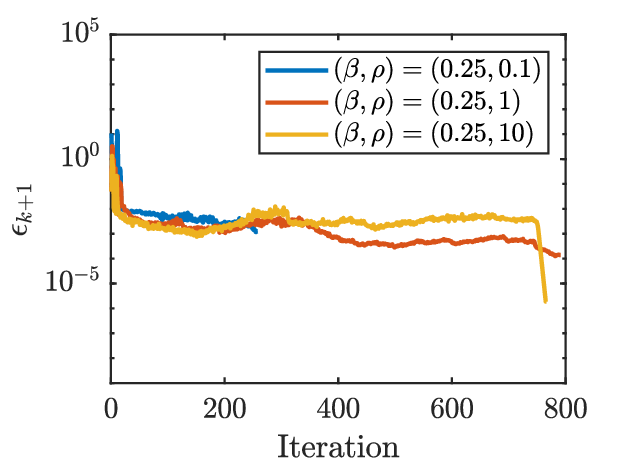}}
\end{subfigure}
\begin{subfigure}[b]{0.32\textwidth}
\centering
{\includegraphics[width=1\textwidth]
{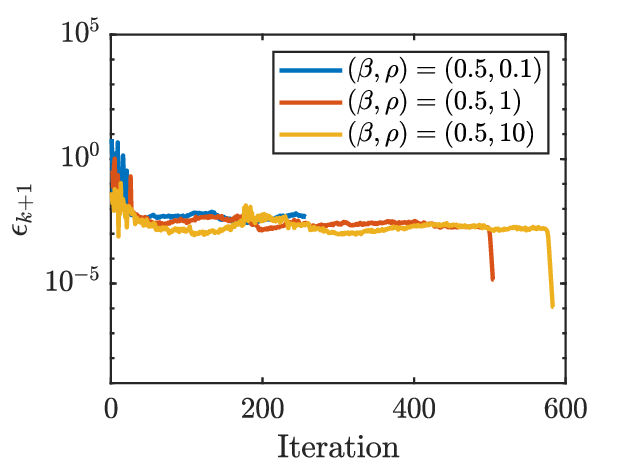}}
\end{subfigure}
\begin{subfigure}[b]{0.32\textwidth}
\centering
{\includegraphics[width=1\textwidth]
{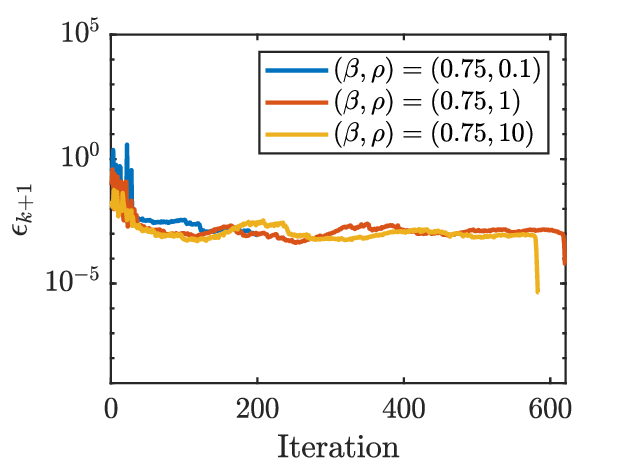}}
\end{subfigure}
\caption{Performance of \Cref{alg:Proxi-descent} on the phase retrieval problem with varying algorithm parameters: $\beta \in \{0.25, 0.5, 0.75\}$ and $\rho \in \{0.1, 1, 10\}$. The first row plots the evolution of the quantity $\|\tilde{g}_{k+1}\|^2$ in \cref{eq:inexactness-trial-point-a}. The second row plots the evolution of the quantity $\epsilon_{k+1}$ in \cref{eq:inexactness-trial-point-b}.  }
\label{fig:numerical-Phase-sensitivity}
\end{figure}
\textbf{Experiment 2: Comparison with PGSG in \cite{davis2019proximally}.}
In this experiment, we evaluate performance on three phase retrieval instances with dimensions $(d,n) = (100,300), (200,600),$ and $(300,900)$. We run both \Cref{alg:Proxi-descent} and PGSG with a total budget of $10^{6}$ iterations for each problem. For PGSG, we further explore different allocations between the outer and inner loops by setting $(T,J_k) = (250,4000), (500,2000),$ and $(1000,1000)$. The parameters are chosen following the guideline in \cite[Equation~9]{davis2019proximally},
with 
$\rho = 10$ 
and $\alpha_j =  \frac{2}{\mu\left(j+2+ \frac{36}{\gamma^4 \mu^4 (j+1)}\right)}$ where  $\mu = \rho$, $\gamma = 1/(\rho + m)$, and $m$ is the weakly convex parameter. For \Cref{alg:Proxi-descent}, the parameters $\beta$ and $\rho$ are selected as $0.75$ and $10$, respectively.

The numerical results are reported in \cref{tab:experiment-phase-comparison}. Across all problem instances, we observe that \Cref{alg:Proxi-descent} consistently attains the highest solution accuracy compared with PGSG. In particular, \Cref{alg:Proxi-descent} achieves residual accuracies of $10^{-8}$, $10^{-5}$, and $10^{-7}$ for three different sizes of problem instances, respectively, whereas the accuracy of PGSG is limited to about $10^{-2}$.

{
\renewcommand{\arraystretch}{1.1}
\begin{table}[t]
\centering

\begin{tabular}{c c c c c c}
\toprule 
$(d,n)$ &  Parameter &   \multicolumn{3}{c}{PGSG \cite{davis2019proximally}}  &  \Cref{alg:Proxi-descent}  \\
\hline
\multirow{4}{*}{$(100,300)$} 
& Outer Iter. & $250$  & $500$  & $1000$   & $1050$  \\
& Inner Iter. & $4000$ & $2000$ & $1000$  & Dynamic \\
& Total Iter. & $10^6$  & $10^6$ & $10^6$ & $10^6$ \\ 
& Stationarity & $1.98\ee{-1}$ & $9.53\ee{-2}$ &  $8.58\ee{-2}$ &$\mathbf{6.66\ee{-8}}$  \\
\cline{2-6}
\multirow{4}{*}{$(150,450)$}
& Outer Iter. & $250$  & $500$  & $1000$   & $2033$  \\
& Inner Iter. & $4000$ & $2000$ & $1000$  & Dynamic \\
& Total Iter. & $10^6$  & $10^6$ & $10^6$ & $10^6$ \\ 
& Stationarity  & $1.9\ee{-1}$ & $5.99\ee{-2}$& $2.05\ee{-2}$ & $\mathbf{6.76\ee{-5}}$ \\
\cline{2-6}
\multirow{4}{*}{$(200,600)$}
& Outer Iter. & $250$  & $500$  & $1000$ & $1927$  \\
& Inner Iter. & $4000$ & $2000$ & $1000$  & Dynamic \\
& Total Iter. & $10^6$  & $10^6$ & $10^6$ & $10^6$ \\ 
& Stationarity  & $4.4\ee{-1}$ & $1.49\ee{-1}$&$3.04\ee{-2}$ & $\mathbf{8.57\ee{-7}}$ \\
\bottomrule
\end{tabular}

\caption{The numerical performance of \Cref{alg:Proxi-descent} and PGSG for solving the phase retrieval problem \cref{eq:Phase-retrieval}. The stationarity measure uses the quantity $\min_{k} \{(\rho+m)^2\|x_{k+1} - x_k\|^2\}$.}
\label{tab:experiment-phase-comparison}
\end{table}
}

\subsection{Blind deconvolution}
\label{subsection:blind}
We next consider a blind deconvolution problem of the form \cite[Example 2.3]{davis2019stochastic}  
\begin{align}
    \label{eq:Blind-deconvolution}
    \min_{x,y \in \RR^d }\; \frac{1}{n} \sum_{i = 1}^n |\innerproduct{u_i}{x}\innerproduct{v_i}{y} - b_i|,
\end{align}
where $u_i,v_i \in \RR^d$ and $b_i \in \RR $ for $i =1,\ldots,n$ are problem data. The goal is to recover a pair of vectors from their pairwise convolution. 

Let $f_i (x,y) =|\innerproduct{u_i}{x}\innerproduct{v_i}{y} - b_i|$ for $i = 1,\ldots,n$. Then $f_i$ can be written as $f_i(x,y) = h_i(c_i(x,y))$ where $h_i(\cdot) = |\cdot|$ is a convex function and $1$-Lipschitz continuous, and $c_i$ is a $ |v_i^\tr u_i|$-smooth function defined as $c_i(x,y) = \innerproduct{u_i}{x}\innerproduct{v_i}{y} - b_i$. Similarly, from \cite[Section 2.1]{davis2019stochastic}, the function $f_i$ is $ |v_i^\tr u_i|$-weakly convex. Let $f = \frac{1}{n} \sum_{i = 1}^n f_i$. 
Then $f$ becomes $\frac{1}{n}\sum_{i=1}^n |v_i^\tr u_i|$-weakly convex. We provide detailed discussion in \cref{Apx-section:numerics}. The subdifferential of $f$ can be computed as 
\begin{align*}
	\partial f(x,y) = \frac{1}{n}\sum_{i=1}^n \partial f_i(x,y), \; \partial f_i(x,y) =  \begin{bmatrix}
	    \innerproduct{v_i}{y} u_i \\
            \innerproduct{u_i}{x} v_i
	\end{bmatrix} \times \begin{cases}
	\mathrm{sign}(\innerproduct{u_i}{x}\innerproduct{v_i}{y} - b_i), & \text{if } \innerproduct{u_i}{x}\innerproduct{v_i}{y} \neq b_i\\
	[-1,1], & \text{otherwise}.
	\end{cases}
\end{align*}

To set up the problem data, we draw the vectors $u_i$ and $v_i$ for $i = 1,\ldots,n$ from a standard Gaussian distribution with mean $0$ and variance $I$ randomly, generate a pair of ground truth signals $\bar{x}, \bar{y} \in \RR^d$ from a unit sphere, and set $b_i  = \innerproduct{u_i}{\bar{x}}\innerproduct{v_i}{\bar{y}}$ for each $i = 1,\ldots, n$.

\textbf{Experiment 1: Sensitivity to algorithm parameters $\beta \in (0,1)$ and $\rho >0$.} Similar to the experiment one in \cref{subsection:phase}, this experiment considers problem \cref{eq:Blind-deconvolution} with the dimension $d = 50$ and run \Cref{alg:Proxi-descent} with different parameters. Our choice of $\beta$ and $\rho$ and the maximum number of iterations are the same as the experiment one in \cref{subsection:phase}. The numerical results are presented in \cref{fig:numerical-Blind-sensitivity}. Across different selection of parameters $\beta$ and $\rho$, we see that the value of $\beta$ does not have a big impact on the convergence. The convergence is largely affected by the magnitude of $\rho$. For all different $\beta$, choosing $\rho = 1$ enables \Cref{alg:Proxi-descent} to achieve the accuracy of $10^{-3}$ and $10^{-4}$ for $\|\tilde{g}_{k+1}\|^2$ and $\epsilon_{k+1}$, respectively, whereas  \Cref{alg:Proxi-descent} with $\rho = 0.1$ can only achieve $10^{-2}$ and $10^{-3}$  for those residuals respectively.

\begin{figure}[t]
\centering
\begin{subfigure}[b]{0.32\textwidth}
\centering
{\includegraphics[width=1\textwidth]
{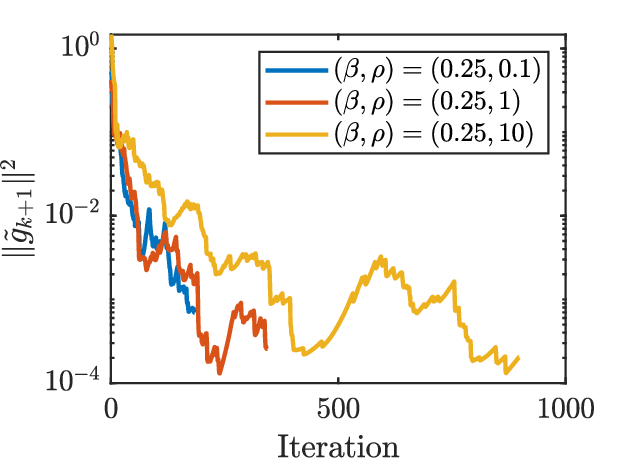}}
\end{subfigure}
\begin{subfigure}[b]{0.32\textwidth}
\centering
{\includegraphics[width=1\textwidth]
{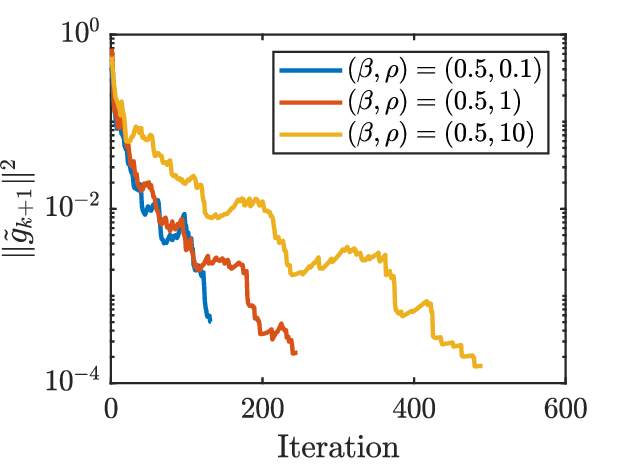}}
\end{subfigure}
\begin{subfigure}[b]{0.32\textwidth}
\centering
{\includegraphics[width=1\textwidth]
{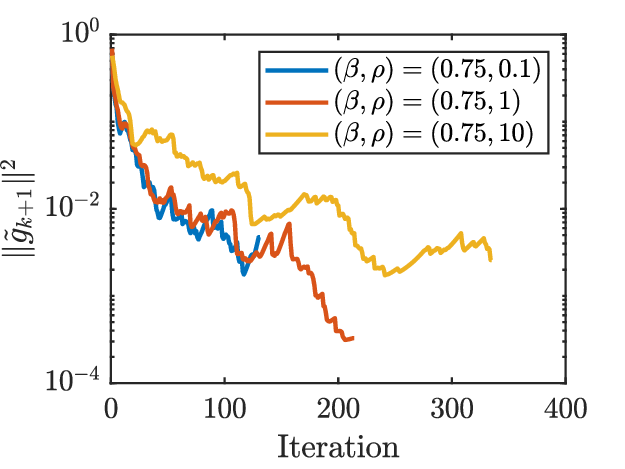}}
\end{subfigure}
\begin{subfigure}[b]{0.32\textwidth}
\centering  
{\includegraphics[width=1\textwidth]
{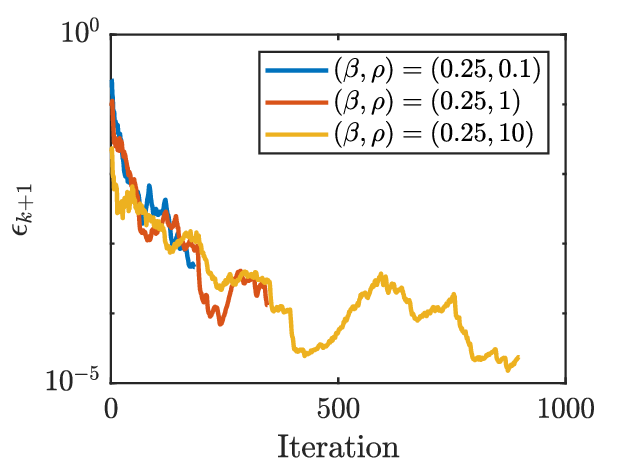}}
\end{subfigure}
\begin{subfigure}[b]{0.32\textwidth}
\centering
{\includegraphics[width=1\textwidth]
{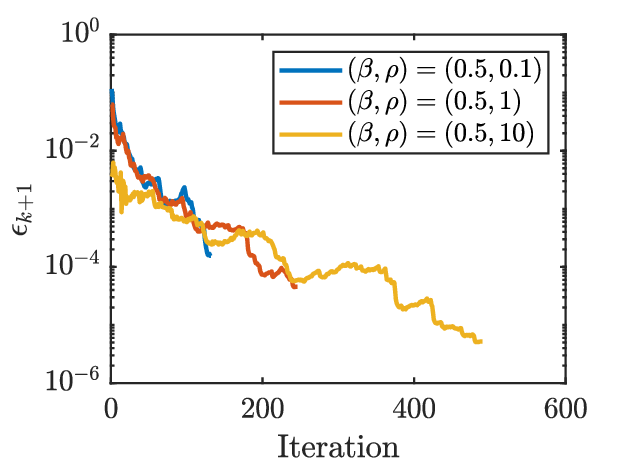}}
\end{subfigure}
\begin{subfigure}[b]{0.32\textwidth}
\centering
{\includegraphics[width=1\textwidth]
{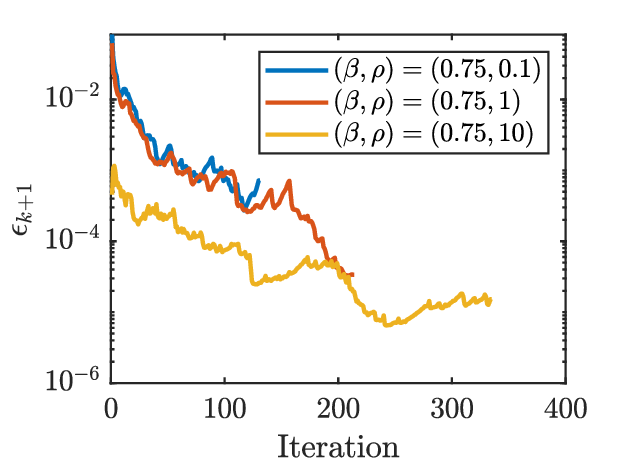}}
\end{subfigure}
\caption{Performance of \Cref{alg:Proxi-descent} on the blind deconvolution problem with varying algorithm parameters: $\beta \in \{0.25, 0.5, 0.75\}$ and $\rho \in \{0.1, 1, 10\}$. The first row plots the evolution of the quantity $\|\tilde{g}_{k+1}\|^2$ in \cref{eq:inexactness-trial-point-a}. The second row plots the evolution of the quantity $\epsilon_{k+1}$ in \cref{eq:inexactness-trial-point-b}.  }
\label{fig:numerical-Blind-sensitivity}
\end{figure}

\textbf{Experiment 2: Comparison with PGSG in \cite{davis2019proximally}.}  We consider three problem instances with the dimension $(d,n)  = (100,300), (150,450),$ and $ (200,600)$. Similar to the second experiment in \cref{subsection:phase}, we run \Cref{alg:Proxi-descent} and PGSG with the same total iteration budget $10^6$, and consider three different combinations of inner and outer budget of PGSG. Our implementation of PGSG is the same as the description in \cref{subsection:phase}. For the implementation of \Cref{alg:Proxi-descent}, we choose the parameters $\beta = 0.75$ and $\rho = 10$. The numerical results are presented in \cref{tab:experiment-Blind-comparison}. We see that \Cref{alg:Proxi-descent} consistently outperforms PGSG, across different dimensional problems. For example, when the dimension of the problem equals $d = 150$ and $n = 450$, \Cref{alg:Proxi-descent} solves the problem to the accuracy of $3.01\times 10^{-4}$, whereas PGSG only achieves the accuracy of $6.05\times 10^{-4}$. 

{
\renewcommand{\arraystretch}{1.1}
\begin{table}[t]
\centering

\begin{tabular}{c c c c c c}
\toprule 
$(d,n)$ &  Parameter &   \multicolumn{3}{c}{PGSG \cite{davis2019proximally}}  &  \Cref{alg:Proxi-descent}  \\
\hline
\multirow{4}{*}{$(100,300)$} 
& Outer Iter. & $250$  & $500$  & $1000$   & $676$  \\
& Inner Iter. & $4000$ & $2000$ & $1000$  & Dynamic \\
& Total Iter. & $10^6$  & $10^6$ & $10^6$ & $10^6$ \\ 
& Stationarity & $5.96\ee{-3}$ & $8.94\ee{-4}$ &  $3.28\ee{-4}$ &$\mathbf{1.66\ee{-4}}$  \\
\cline{2-6}
\multirow{4}{*}{$(150,450)$}
& Outer Iter. & $250$  & $500$  & $1000$   & $828$  \\
& Inner Iter. & $4000$ & $2000$ & $1000$  & Dynamic \\
& Total Iter. & $10^6$  & $10^6$ & $10^6$ & $10^6$ \\ 
& Stationarity  & $5.09\ee{-3}$ & $1.68\ee{-3}$ &  $6.05\ee{-4}$ &$\mathbf{3.01\ee{-4}}$ \\
\cline{2-6}
\multirow{4}{*}{$(200,600)$}
& Outer Iter. & $250$  & $500$  & $1000$ & $891$  \\
& Inner Iter. & $4000$ & $2000$ & $1000$  & Dynamic \\
& Total Iter. & $10^6$  & $10^6$ & $10^6$ & $10^6$ \\ 
& Stationarity  &$6.91\ee{-3}$ & $1.75\ee{-3}$ &  $9.89\ee{-4}$ &$\mathbf{4.27\ee{-4}}$  \\
\bottomrule
\end{tabular}

\caption{The numerical performance of \Cref{alg:Proxi-descent} and PGSG for solving the blind deconvolution problem \cref{eq:Phase-retrieval}. The stationarity measure uses the quantity $\min_{k} \{(\rho+m)^2\|x_{k+1} - x_k\|^2\}$.}
\label{tab:experiment-Blind-comparison}
\end{table}
}

%% file: 7-Conclusion.tex
\section{Conclusion}
\label{section:conclusion}
In this paper, we have introduced a proximal descent algorithm for solving weakly convex optimization problems. 
The algorithm is built upon the framework of the inexact proximal point method, with each subproblem efficiently addressed via a proximal bundle update. From a methodological perspective, our work extends the well-studied proximal bundle method from the convex setting to the broader class of weakly convex functions. From a theoretical perspective, we established explicit non-asymptotic convergence guarantees: the algorithm finds near-stationary points at the best-known rates for weakly convex optimization, and it automatically accelerates under additional structural conditions such as smoothness or quadratic growth. Finally, our numerical experiments validate the theoretical findings, illustrating the effectiveness and robustness of the proposed method for solving practical non-convex problems. 

We conclude by highlighting a few directions for future research:
\begin{itemize}[wide = 0pt]
\setlength{\itemsep}{0pt}
     \item \textbf{Weakly convex parameter and growth condition:} \Cref{alg:Proxi-descent} requires the knowledge of the weakly convex parameter $m$. In practice, however, $m$ may be unknown even though the objective is weakly convex. An interesting direction is to design adaptive schemes that estimate or learn the weak convexity parameter on the fly, thereby broadening the applicability of the method in real-world settings. In addition, 
    \Cref{subsection:QG} focused on the quadratic growth condition \cref{eq:QG}. A natural extension is to consider the more general growth condition 
    \begin{align*}
        f(x) - \min_{x} f(x) \geq \frac{\mu}{2}  \cdot \Dist^\alpha(x,S), \; \forall x \in \RR^n,
    \end{align*}
    where $\mu > 0$, $\alpha \geq 1$, and $S = \argmin_{x} f(x)$. Investigating this condition could yield sharper convergence guarantees, paralleling the recent advances for convex proximal bundle methods \cite{diaz2023optimal}. 
    \item \textbf{Convergence to local minima:} Examining the proof for \cref{lemma:proximal-descent-linear}, we find that the assumption can be relaxed while still guaranteeing the same improved iteration complexity. 
    The linear convergence result in \cref{lemma:proximal-descent-linear} relies on the linear decay of the function value gap $f(x) - f^\star$, which follows from the quadratic growth condition \cref{eq:QG} together with the assumption $\muq > m$ (in which case no local minima exist). When the function admits local minima, if it satisfies a quadratic growth-type condition in a neighborhood of a local minimum, and assuming the algorithm converges to that local minimum, then the same complexity bound can still be established. This raises a natural open question: Does the proximal descent method always converge to a local minimum?

    \item \textbf{Constrained weakly convex optimization and applications to control-theoretic problems.} Robust control was one of the earliest motivations and applications for nonsmooth optimization \cite{lewis2007nonsmooth}. However, control problems naturally involve nonconvex domains, since the set of stabilizing policies is nonconvex \cite{talebi2024policy}. Recent advances on benign nonconvexity \cite{zheng2024benign,zheng2023benign,zheng2025extended} suggest that many control problems may in fact be weakly convex over compact regions. This makes it particularly interesting to extend the proximal descent method to constrained weakly convex optimization, thereby opening new opportunities for applications in control-theoretic settings.
\end{itemize}

%% file: Appendix.tex
\newpage

\noindent {\LARGE \bfseries Appendix}

\vspace{5mm}
In this appendix, we provide supplementary materials for the main text. The appendix is divided into four parts
\begin{itemize}
    \item \cref{Apx-section:diff-stationary} establishes the relationships between different stationary measures;
    \item \cref{Apx-section:PMB-weakly} completes the missing proofs in \cref{section:PMB-weakly};
    \item \cref{Apx-section:improved-rate} completes the missing proofs in \cref{section:improved-rate};
    \item \cref{Apx-section:numerics} provides more details on the numerical experiment in \cref{section:numerics}.
\end{itemize}

\section{Different notions of stationarity and their relationship} \label{appendix:stationary-point}

In this appendix, we review different notions of stationarity, including gradient norm, $(\eta,\epsilon)$-inexact stationarity, and $(\delta,\alpha)$-Moreau stationarity. We also summarize their relationship. 

\label{Apx-section:diff-stationary}
\subsection{Different notions of stationarity}
For an $M$-smooth function $f:\RR^n \to \RR \cup \{+\infty\}$, a point $x$ is stationary if $\nabla f(x) = 0$, and a natural measure of near-stationarity is the norm of its gradient $\|\nabla f(x)\|$. If $\|\nabla f(x)\|$ is small, we say $x$ is close to a stationary point. For a closed function $f$, its Fr\'echet subdifferential at $x$ with $f(x)$ finite is defined as
$$
    \partial f(x)  = \left \{ v \in \RR^n \mid \liminf_{y \to x }  \frac{ f(y) - f(x) -  \innerproduct{v}{y-x}}{\|y-x\|} \geq 0 \right\}.
$$ 
If $f$ is further $m$-weakly convex, the subdifferential $\partial f$ can be equivalently characterized as
\begin{align*}
    \partial f(x)  = \left \{v \in \RR^n \mid f(y) \geq  f(x) + \innerproduct{v}{y-x} - \frac{m}{2}\|y-x\|^2,\; \forall y \in \RR^n\right \}.
\end{align*}
We say $x$ is a stationary point if $0 \in \partial f(x)$. However, the quantity $\mathrm{dist}(0,\partial f(x))$ may not be a good measure of near stationarity, since $\mathrm{dist}(0,\partial f(x))$ may be a discontinuous function. For example, consider $x \mapsto |x|$,  we have $\mathrm{dist}(0,\partial f(x)) = 1, \forall x \neq 0$, and $\mathrm{dist}(0,\partial f(x)) = 1, \text{if}\, x = 0$.  

To quantify approximated stationarity for $m$-weakly convex functions, we have considered two notions in the main text: $(\eta,\epsilon)$-inexact stationarity, and $(\delta,\alpha)$-Moreau stationarity. 
An inexact version of the subdiffernetial $\partial f$ introduces an inexactness $\epsilon >0$ and defines 
\begin{align*}
    \partial_{\epsilon} f(x)  =\left \{v \in \RR^n \mid f(y) \geq  f(x) + \innerproduct{v}{y-x} - \frac{m}{2}\|y-x\|^2 - \epsilon,\; \forall y \in \RR^n\right \}.
\end{align*}
Ideally, we aim to find a point $x\in \RR^n$ such that $\Dist(0,\partial_{\epsilon} f(x)) \leq \eta $ with both $\epsilon$ and $\eta$ being small. Thus, we consider the notion of  $(\eta,\epsilon)$-inexact stationarity.

\begin{definition}[$(\eta,\epsilon)$-inexact stationarity]
    Let $f:\RR^n \to \RR\cup \{+\infty\}$ be an $m$-weakly convex function, and $\eta > 0, \epsilon > 0$. A point $x \in \mathrm{dom}\, f$ is called an $(\eta,\epsilon)$-stationary point if 
    \begin{align*}
        \Dist(0,\partial f_{\epsilon}(x)) \leq \eta.
    \end{align*}
\end{definition}
When $\epsilon =0 $ and $\eta = 0$, an $(\eta,\epsilon)$-inexact stationary point becomes a true stationary point, i.e., $0 \in \partial f(x)$. Another stationary measure uses the notion of the Moreau envelope, defined as 
\begin{equation*} 
    f_{\rho}(x) := \inf_{y\in \mathbb{R}^n} \left\{f(y) + \frac{\rho}{2}\|y-x\|^2\right\},
\end{equation*}
where $\rho >0$ is fixed. For a $m$-weakly convex function and any $\rho > m $, the function $f_{\rho}$ becomes continuous differentiable (see \cref{lemma:gradient-moreau}), and the gradient can be computed as 
$
    \nabla f_{\rho}(x) = \rho (x - \hat{x}),
$
where $\hat{x}:=\argmin_{y} \{f(y) + \frac{\rho}{2}\|y-x\|^2\}$ is unique. Moreover, we have the following result. 
\begin{proposition}
    \label{prop:moreau-iff}
    Let $f:\RR^n \to \RR \cup \{+\infty\}$ be an $m$-weakly convex function, $\rho > m$ and $x \in \RR^n$. Then $ \nabla f_{\rho}(x) = \{0\}$ if and only if  $0 \in \partial f(x)$.
\end{proposition}
\begin{proof}
    The proof follows easily from \cref{lemma:gradient-moreau}. Suppose $\nabla f_{\rho}(x) = \{0\}$. From \cref{eq:gradient-Moreau}, we know that $x = \argmin_{y \in \RR^n } \{  f(y) + \frac{\rho}{2}\|y -x \|^2 \}$ and therefore $0 \in \partial f(x) + \rho (x - x) = \partial f(x)$.
        
         Conversely, suppose $0 \in \partial f(x)$. Then $0 \in \partial f(x) + \rho (x-x)$, showing that $x \in \argmin_{y \in \RR^n } \{  f(y) + \frac{\rho}{2}\|y -x \|^2 \}$. By \cref{eq:gradient-Moreau}, we know $x$ is the unique minimizer and $\nabla f_{\rho}(x) = \rho (x - x) = 0$. 
\end{proof}

This motivates the following approximate stationary notion.
\begin{definition}[$(\delta,\alpha)$-Moreau stationarity]
    Let $f:\RR^n \to \RR\cup \{+\infty\}$ be an $m$-weakly convex function, $\delta>0,$ and $ \alpha >m$. A point $x \in \mathrm{dom}\, f$ is called an $(\delta,\alpha)$-Moreau stationary point if 
\begin{align*}
    \|\nabla f_{\alpha}(x)\| \leq \delta.
\end{align*}
\end{definition}
Clearly, if $\delta = 0$, $(\delta,\alpha)$-Moreau stationarity recovers $\|\nabla f_{\alpha}(x)\| = 0$, i.e., $x$ is a stationary point. 

\subsection{Properties and their relationship}
In this subsection, we establish the relationship between different stationary measures.
\begin{lemma}[$(\eta,\epsilon)$-inexact stationarity to $(\delta,\alpha)$-Moreau stationarity]
\label{lemma:stationary-conversion-R2M-2}
    Let $f:\RR^n \to \RR \cup \{+\infty\}$ be an $m$-weakly convex function. Suppose $x$ is a $(\eta,\epsilon)$-inexact stationary point. Then  we have
    \begin{align*}
        \|\nabla f_{m+\lambda}(x)\|  \leq  (m + \lambda)\left ( \frac{2}{\lambda}\eta + \sqrt{\frac{2}{\lambda}\epsilon} \right), \quad \forall \lambda > 0.
    \end{align*}
    If $\lambda = m$, it simplifies to
     $   \|\nabla f_{2m}(x)\| \leq  4 \eta + 2\sqrt{2m\epsilon}.$ 
\end{lemma}
\begin{proof}
    Since $x$ is a $(\eta,\epsilon)$-inexact stationary point and $\partial_{\epsilon}f(x)$ is a closed set, there exists a $w \in \RR^n$ such that 
    \begin{align}
        \label{eq:R2M-step-1}
        \|w\| \leq \eta, \quad f(y) + \frac{m}{2}\|y-x\|^2 \geq f(x) + \innerproduct{w}{y-x} - \epsilon, \; \forall y \in \RR^n.
    \end{align}
    Let $\lambda> 0$ and $\hat{x} = \argmin_{y} \{f(y) + \frac{ m + \lambda}{2} \|y - x\|^2\}$. Since $f(\cdot) + \frac{m+\lambda}{2}\|\cdot-x\|^2$ is $\lambda$-strongly convex, we have 
    \begin{align*}
        \frac{\lambda}{2}\|\hat{x} - x\|^2 &\leq f(x) + \frac{ m + \lambda}{2} \|x - x\|^2 - \left(f(\hat{x}) + \frac{ m + \lambda}{2} \|\hat{x}^\lambda - x\|^2\right) \\
        & \overset{(a)}{\leq} f(x ) -  \left (f(\hat{x}) + \frac{ m }{2} \|\hat{x} - x\|^2 \right) \\
        & \overset{(b)}{\leq} \innerproduct{w}{x - \hat{x}} + \epsilon \\
        & \overset{(c)}{\leq} \eta\|\hat{x} - x\| + \epsilon
    \end{align*}
    where $(a)$ drops a nonpositive term $\frac{\lambda}{2}\|\hat{x} - x\|^2$, $(b)$ uses \cref{eq:R2M-step-1} with $y = \hat{x}$, and $(c)$ uses Cauchy-inequality with the bound $\|w\| \leq \eta$ in  \cref{eq:R2M-step-1}.
    Therefore, we obtain
    \begin{align*}
        \|\hat{x} - x\|^2  - \frac{2}{\lambda}\eta\|\hat{x} - x\| - \frac{2}{\lambda}\epsilon\leq 0 .
    \end{align*}
    Since the above is a quadratic function in terms of $\|\hat{x} - x\|$, the quantity $\|\hat{x} - x\|$ can not be larger than its right root. We thus have  
    \begin{align*}
        \|\hat{x} - x\| &\leq \frac{1}{2}\left (\frac{2}{\lambda}\eta + \sqrt{\frac{4}{\lambda^2}\eta^2 +\frac{8}{\lambda} \epsilon}\right) \leq \frac{1}{2}\left (\frac{4}{\lambda}\eta + \sqrt{ \frac{8}{\lambda} \epsilon}\right) = \frac{2}{\lambda}\eta + \sqrt{\frac{2}{\lambda}\epsilon}.
    \end{align*}
    Lastly, using the relationship $\nabla f_{ m + \lambda}(x) = (m+\lambda)(\hat{x} - x)$ from \cref{eq:gradient-Moreau}, we obtain
    \begin{align*}
        \|\nabla f_{m + \lambda}(x)\| \leq (m + \lambda)\left ( \frac{2}{\lambda}\eta + \sqrt{\frac{2}{\lambda}\epsilon} \right).
    \end{align*}
    This completes the proof. 
\end{proof}

If an $m$-weakly convex function is also $M$-smooth, $(\delta,\alpha)$-Moreau stationarity also implies a small gradient, which is a more natural stationary measure for smooth functions.

\begin{lemma}[$(\delta,\alpha)$-Moreau stationarity to gradient stationarity]
\label{lemma:stationary-conversion-R2G}
    Suppose $f:\RR^n \to \RR \cup \{+\infty\}$ is $m$-weakly convex and $M$-smooth. Let $\alpha > m$, $x \in \RR^n$, and $\epsilon \geq 0$. If $x$ satisfies $\|\nabla_{\alpha} f(x)\| \leq \delta$, then $\|\nabla f(x)\| \leq \left ( 1+ \frac{M}{\alpha} \right )\delta$.
\end{lemma}
\begin{proof}
    Suppose $\|\nabla_{\alpha} f(x)\| \leq \epsilon$. From \cref{eq:Moreau-property}, we know that there exists $\hat{x} \in \RR^n $ such that $\|\nabla f(\hat{x})\| \leq \epsilon$ and $\|\hat{x} - x\| = \|\nabla f_{\alpha}(x)\|/\alpha$. It then holds that
    \begin{align*}
        \|\nabla f(x)\| \leq \|\nabla f(\hat{x})\| +  M \| x - \hat{x}\|  \leq \epsilon + \frac{M}{\alpha}\|\nabla f_{\alpha}(x)\| \leq \left (1+ \frac{M}{\alpha} \right)\epsilon,
    \end{align*}
    where the first inequality uses the $M$-smoothness.
\end{proof}

Combining \cref{lemma:stationary-conversion-R2M-2,lemma:stationary-conversion-R2G}, we also see that $(\eta,\epsilon)$-inexact stationarity implies gradient stationarity.

\begin{lemma}[$(\eta,\epsilon)$-inexact stationarity to gradient stationarity]
\label{lemma:stationary-conversion-R2G-appendix}
    Suppose $f:\RR^n \to \RR \cup \{+\infty\}$ is $m$-weakly convex and $M$-smooth. Let $\lambda > 0$, $\alpha = m +\lambda$, $x \in \RR^n$, and $\epsilon \geq 0$. If $x$ is a $(\eta,\epsilon)$-inexact stationary point, then $$\|\nabla f(x)\| \leq \left ( 1+ \frac{M}{\alpha} \right )\alpha \left ( \frac{2}{\lambda}\eta + \sqrt{\frac{2}{\lambda}\epsilon} \right).$$
    Consequently, choosing $\epsilon \leq  \frac{2}{\lambda}\eta^2$ and $\eta \leq  \frac{\delta \lambda}{4 (1+M/\alpha) \alpha}$ with $\delta \geq 0$ implies $\|\nabla f(x)\| \leq \delta$.
\end{lemma}
\begin{proof}
    From \cref{lemma:stationary-conversion-R2M-2}, we know that 
    \begin{align*}
        \|\nabla f_{\alpha}(x)\|  \leq  \alpha\left ( \frac{2}{\lambda}\eta + \sqrt{\frac{2}{\lambda}\epsilon} \right).
    \end{align*}
    Combining this with \cref{lemma:stationary-conversion-R2G} yields
    \begin{align*}
        \|\nabla f(x)\| \leq \left ( 1+ \frac{M}{\alpha} \right ) \alpha\left ( \frac{2}{\lambda}\eta + \sqrt{\frac{2}{\lambda}\epsilon} \right).
    \end{align*}
     If we choose $\epsilon \leq  \frac{2}{\delta}\eta^2$ and $\eta \leq  \frac{\delta \lambda}{4 (1+M/\alpha) \alpha}$ with $\delta \geq 0$, then the above becomes
    \begin{align*}
        \|\nabla f(x)\| & \leq \left ( 1+ \frac{M}{\alpha} \right ) \alpha   \frac{4}{\lambda}\eta \leq \delta.
    \end{align*}
\end{proof}

In \cref{proposition:implication-proximal-gap,corollary:proximal-gap}, we quantify the stationarity of the iterates in \Cref{alg:Proxi-descent} by looking at the proximal gap. We next show that the proximal gap indeed implies $(\eta,\epsilon)$-inexact stationarity and $(\delta,\alpha)$-Moreau stationarity with $\eta,\epsilon,$ and $\delta$ related to the proximal gap.

\begin{lemma}[Consequences of the proximal gap]
    \label{lemma:prox-gap-consq}
    Assume $f:\RR^n \to \RR \cup\{+\infty\}$ is $m$-weakly convex. Fix $x \in \RR^n$, $\rho > 0$, and $\alpha = m + \rho$. Let $\Delta_x = f(x) - \inf_{y} \{f(y) + \frac{\alpha}{2}\|y-x\|^2\}$, $\hat{x}^\rho = \argmin_{y \in \RR^n} \{ f(y) + \frac{m + \rho }{2}\|y - x\|^2 \}$, and $w = \rho (x -\hat{x}^\rho)$. Then we have the following
    \begin{itemize}
        \item The vector $w$ is an $\Delta_x$-inexact subgradient of $f$ at $x$, i.e., $w \in \partial_{\Delta_x} f(x)$;
        \item The vector $w$ can be bounded as $\|w\| \leq  \sqrt{ 2\rho \Delta_x}$. Consequently, the point $x$ is a $(\Delta_x,\sqrt{ 2\rho \Delta_x})$-inexact stationary point and a $(\sqrt{\frac{2 \Delta_x \alpha^2 }{\rho}},\alpha)$-Moreau stationary point.
\end{itemize}

\end{lemma}
\begin{proof}
    Let $\hat{x}^\rho = \argmin_{y} \{f(y) + \frac{m+\rho}{2}\|y-x\|^2 \}$. From the optimality condition, we know 
\begin{align*}
     (m + \rho) (x -\hat{x}^\rho) \in \partial f (\hat{x}^\rho).
\end{align*}
Let $v = (m + \rho) (x -\hat{x}^\rho)$ and $w = \rho (x -\hat{x}^\rho)$. As $f$ is $m$-weakly convex, for all $y \in \RR^n$, \cref{eq:Frechet-subdifferetnial-concrete} shows that 
\begin{align*}
     f(y) + \frac{m}{2}\|y - \hat{x}^\rho\|^2  \geq  f(\hat{x}^\rho) + \innerproduct{v}{y-\hat{x}^\rho},
\end{align*}
which is equivalent to 
\begin{align*}
    f(y) + \frac{m}{2}\|y - \hat{x}^\rho\|^2 + \frac{m}{2}\|y - x \|^2 \geq \frac{m}{2}\|y-x\|^2+  f(\hat{x}^\rho) + \innerproduct{v}{y-\hat{x}^\rho}.
\end{align*}
Hence, we have
\begin{align*}
     f(y) + \frac{m}{2}\|y - x\|^2 &   \geq \frac{m}{2}\|y-x\|^2  -  \frac{m}{2}\|y - \hat{x}^\rho \|^2  + f(x) - f(x)+  f(\hat{x}^\rho) + \innerproduct{v}{y-\hat{x}^\rho}\\
    & \overset{(a)}{=} f(x)+ \frac{m}{2}\|x-\hat{x}^\rho\|^2 + m\innerproduct{\hat{x}^\rho - x}{y-\hat{x}^\rho}   - f(x)+  f(\hat{x}^\rho) + \innerproduct{v}{y-\hat{x}^\rho} \\
    & \overset{(b)}{=} f(x)+ \frac{m}{2}\|x-\hat{x}^\rho\|^2 + \rho \innerproduct{ x - \hat{x}^\rho }{y-\hat{x}^\rho}   - f(x)+  f(\hat{x}^\rho)  \\
    &  \overset{(c)}{=} f(x)+ \frac{m + \rho }{2}\|x-\hat{x}^\rho\|^2 +  \innerproduct{ w}{y-x}   - f(x)+  f(\hat{x}^\rho) \\
     & = f(x)+ \innerproduct{ w }{y-x}  - \left( f(x) -   f(\hat{x}^\rho) - \frac{m + \rho }{2}\|x-\hat{x}^\rho\|^2 \right)\\
     & = f(x)+ \innerproduct{ w }{y-x}  -\Delta_x,
\end{align*}
where $(a)$ applies the identity $2m \innerproduct{a}{b} = m\|a + b\|^2 - m \|a\|^2 - m\|b\|^2$ with $a +b = y-x, a = \hat{x}^\rho -x,$ and $b = y - \hat{x}^\rho$, $(b)$ uses $\rho ( x - \hat{x}^\rho ) =  v - m(x -\hat{x}^\rho ) $, and $(c)$ adds and subtracts the term $\rho \innerproduct{x-\hat{x}^\rho}{x}$. Hence, we have $w \in \partial_{\Delta_x} f(x)$.

On the other hand, due to the $\rho$-strong convexity of $y \mapsto f(y) + \frac{m+\rho}{2}\|y-x\|^2 $, we have 
\begin{align*}
    \frac{1}{2\rho}\|w\|^2= \frac{\rho}{2}\|x - \hat{x}^\rho\|^2 \leq \Delta_x,
\end{align*}
implying $\|w\| \leq \sqrt{2\rho\Delta_x}$. Thus, by \Cref{def:Regularized-stationary}, the point $x$ is a $(\Delta_x,\sqrt{2\rho\Delta_x})$-inexact stationary point.

Lastly, using the gradient of the Moreau envelope \cref{eq:gradient-Moreau}, we can further deduce that $\| \nabla f_{\alpha} (x) \| \leq \sqrt{\frac{2 \Delta_x \alpha^2 }{\rho}},$ completing the proof.
\end{proof}

\Cref{lemma:stationary-conversion-R2M-2} establishes the implication from $(\eta,\epsilon)$-inexact stationarity to $(\delta,\alpha)$-Moreau stationarity. We next show that $(\delta,\alpha)$-Moreau stationarity implies $(\eta,\epsilon)$-inexact stationarity, under the Lipschitz continuity assumption.

\begin{lemma}[$(\delta,\alpha)$-Moreau stationarity to $(\eta,\epsilon)$-inexact stationarity]
\label{lemma:MS-2-IS}
     Let $f:\RR^n \to \RR \cup\{+\infty\}$ be a $m$-weakly convex function, $x \in \RR^n$, and $\alpha = m + \rho$ with $\rho > 0$. If $f$ is $L$-Lipschitz continuous and $x$ is a $(\delta,\alpha)$-Moreau stationary point, then $x$ is a $(L \frac{\epsilon}{\alpha},\sqrt{2\rho L \frac{\epsilon}{\alpha}})$-inexact stationary point.
\end{lemma}
\begin{proof}
    Since $x$ is a $(\delta,\alpha)$-Moreau stationary point, from \cref{eq:gradient-Moreau}, we have $\alpha\|x - \hat{x}\| \leq  \delta$, where $\hat{x} = \argmin_{y}\{f(y) + \frac{\alpha}{2}\|y - x\|^2\}$. It follows that
\begin{align*}
    \Delta_x := f(x) -  f(\hat{x}) - \frac{\alpha }{2}\|x-\hat{x}\|^2 \leq L\| x - \hat{x}\| \leq L \frac{\delta}{\alpha},
\end{align*}
where the first inequality uses the $L$-Lipschitz continuous assumption.

On the other hand, \cref{lemma:prox-gap-consq} tells us that $x$ is a $(\Delta_x,\sqrt{2\rho\Delta_x})$-inexact stationary point. Consequently, we deduce that $x$ is a $(L \frac{\delta}{\alpha},\sqrt{2\rho L \frac{\delta}{\alpha}})$-inexact stationary point.
\end{proof}

To conclude this section, we present a sketch of the proof showing that the function $x \mapsto \min_{v \in \partial_{\epsilon}f(x)}\|v\|$ is continuous for any given $\epsilon > 0$, provided $f$ is convex. The proof is adapted from the recent paper \cite[Proposition 4]{li2025subgradient}, which utilizes results from \cite{asplund1969gradients,rockafellar2009variational}.
\begin{lemma} \label{lemma:continuous-subdifferential}
    Let $f:\RR^n \to \RR$ be a convex function and $\epsilon > 0$. The function $x \mapsto \min_{v \in \partial_{\epsilon}f(x)}\|v\|$ is continuous for any given $\epsilon > 0$.
\end{lemma}
\begin{proof}
    The proof can be established through the following steps:
    \begin{itemize}
        \item Since $f$ is locally bound, for any $\epsilon > 0$, the set value map $\partial_{\epsilon} f$ is also locally bounded \cite[Corollary 1]{asplund1969gradients};
        \item Fix any $\bar x \in \RR^n$. Since $f$ is convex, it holds that $\lim_{x \to \bar x} \mathbb{H}(\partial_{\epsilon} f(x), \partial_{\epsilon} f(\bar x)) = 0$ \cite[Proposition 5]{asplund1969gradients}, where $\mathbb{H}$ is the Hausdorff distance (i.e., Pompeiu-Hausdorff distance) defined as $$\displaystyle \mathbb{H}(A,B) := \max \left  \{  \sup_{x\in A}\; \Dist(x,B) , \sup_{y \in B} \;\Dist(y,A)\right \}$$ for two closed sets $A, B \subseteq \RR^n$.
        \item Since $\partial_{\epsilon} f$ is locally bounded, it holds that $\lim_{x \to \bar x} \mathbb{H}(\partial_{\epsilon} f(x), \partial_{\epsilon} f(\bar x)) = 0$ if and only if the set value map $ \partial_{\epsilon} f$ is continuous at $\bar x$ \cite[Corollary 5.21]{rockafellar2009variational};
        \item Since $\partial_{\epsilon} f$ is continuous at $\bar x$, for any sequence $\{ x_k\} $ converges to $\bar x$, we have $\partial_{\epsilon} f(x_k) \to \partial_{\epsilon} f(x)$. The minimal norm element also converges as $\lim_{x \to \bar x} \argmin_{v \in \partial_{\epsilon}} f(x) = \argmin_{v \in \partial_{\epsilon}} f(\bar x)$ \cite[Proposition 4.9]{rockafellar2009variational}. This completes the proof.
    \end{itemize} 
\end{proof}

\section{Technical Proofs in \Cref{section:PMB-weakly}}
\label{Apx-section:PMB-weakly}
In this section, we complete the missing proofs in \Cref{section:PMB-weakly}.
\subsection{Proof for \cref{lemma:null-step-improvement}}
\label{subsection:null-step-improvement}
Our analysis is largely inspired by \cite[Section 5.3]{diaz2023optimal}. The proof follows from writing out the solution of the subproblem 
\begin{align}
    \label{eq:essential-subproblem}
    \min_y \left \{ \hat{f}_{j+1}(y) + \frac{\rho}{2}\|y -x_k\|^2 \right \},
\end{align}
where $ \hat{f}_{j+1}$ is the essential model
\begin{align*}
    \hat{f}_{j+1}(y) = \max \left \{ \tilde{f}_j(z_{j+1}) + \innerproduct{s_{j+1}}{ y- z_{j+1}}, f(z_{j+1}) + \frac{m}{2}\|z_{j+1} - x_{k}\|^2 + \innerproduct{ g_{j+1}}{y  - z_{j+1}} \right \}.
\end{align*}
The optimal solution for \cref{eq:essential-subproblem} can be found as (see \cref{subsection:closed-form-sol})
\begin{align}
    y^\star &=  x_k -  \frac{1}{\rho}((1-\theta^\star)s_{j+1} + \theta^\star g_{j+1}), \label{eq:ystar}\\
    \theta^\star &= \min\left \{1, \frac{\rho(f(z_{j+1}) + \frac{m}{2}\|z_{j+1} - x_{k}\|^2 - \tilde{f}_j(z_{j+1}))}{\|s_{j+1} - g_{j+1}\|^2} \right \}. \nonumber
\end{align}
We also note that 
\begin{align}
    \frac{1}{\rho}s_{j+1} & =x_k - z_{j+1}, \label{eq:update-1}  \\
    y^\star - z_{j+1} & = x_k - \frac{1}{\rho}( s_{j+1} + \theta^\star (g_{j+1} - s_{j+1}) ) -z_{j+1} = -\frac{\theta^\star}{\rho} (g_{j+1} - s_{j+1}) .\label{eq:update-2}
\end{align}
Therefore, it follows that 
{
\allowdisplaybreaks
\begin{align*}
     & \eta_{j+1} \\
    \geq\; &  \hat{f}_{j+1}(y^\star) + \frac{\rho}{2}\|y^\star -x_k\|^2 \\
    \geq \; & (1- \theta^\star)(\tilde{f}_j(z_{j+1}) + \innerproduct{s_{j+1}}{ y^\star- z_{j+1}}) \\
     & \quad +  \theta^\star ( f(z_{j+1}) + \frac{m}{2}\|z_{j+1} - x_{k}\|^2 + \innerproduct{ g_{j+1}}{y^\star - z_{j+1}} )  +\frac{\rho}{2}\|y^\star -x_k\|^2\\
    = \; &\tilde{f}_j(z_{j+1}) + \theta^\star (f(z_{j+1}) + \frac{m}{2}\|z_{j+1} - x_{k}\|^2  - \tilde{f}_j(z_{j+1}) ) \\
    & \quad + \innerproduct{s_{j+1} + \theta^\star( g_{j+1} - s_{j+1})}{y^\star- z_{j+1}} +  \frac{\rho}{2}\|y^\star -x_k\|^2 \\
    = \; & \tilde{f}_j(z_{j+1}) + \theta^\star (f(z_{j+1}) + \frac{m}{2}\|z_{j+1} - x_{k}\|^2  - \tilde{f}_j(z_{j+1}) ) \\
   &  \quad + \innerproduct{s_{j+1} + \theta^\star( g_{j+1} - s_{j+1})}{-\frac{\theta^\star}{\rho} (g_{j+1} - s_{j+1})} +  \frac{1}{2\rho}\|s_{j+1} + \theta^\star( g_{j+1} - s_{j+1})\|^2 \\
    = \; & \tilde{f}_j(z_{j+1}) + \theta^\star (f(z_{j+1}) + \frac{m}{2}\|z_{j+1} - x_{k}\|^2  - \tilde{f}_j(z_{j+1}) ) - \frac{(\theta^\star)^2}{2\rho} \|g_{j+1} - s_{j+1}\|^2 + \frac{1}{2\rho}\|s_{j+1}\|^2 \\
     = \; & \tilde{f}_j(z_{j+1}) + \theta^\star (f(z_{j+1}) + \frac{m}{2}\|z_{j+1} - x_{k}\|^2  - \tilde{f}_j(z_{j+1}) ) - \frac{(\theta^\star)^2}{2\rho} \|g_{j+1} - s_{j+1}\|^2 + \frac{\rho}{2}\|z_{j+1} -x_k\|^2 \\
     = \;&\eta_{j} + \underbrace{ \theta^\star (f(z_{j+1}) + \frac{m}{2}\|z_{j+1} - x_{k}\|^2  - \tilde{f}_j(z_{j+1}) ) - \frac{(\theta^\star)^2}{2\rho} \|g_{j+1} - s_{j+1}\|^2}_{\Phi},
\end{align*}
}
where the first inequality is due to $\tilde{f}_{j+1} \geq \hat{f}_{j+1} $, the second equality uses \cref{eq:update-2,eq:ystar}, the fourth equality uses \cref{eq:update-2}, and the last equality uses the definition of $\eta_{j} =\min_{x} \{  \tilde{f}_j(x) + \frac{\rho}{2}\|x-x_k\|^2 \}= \tilde{f}_j(z_{j+1})+  \frac{\rho}{2}\|z_{j+1} -x_k\|^2$.

Using the definition of $\theta^\star$, we can further lower bound the improvement $\Phi$. Specifically, if $\theta^\star = 1$ (in this case, $\rho(f(z_{j+1}) + \frac{m}{2}\|z_{j+1} - x_{k}\|^2 - \tilde{f}_j(z_{j+1})) \geq \|g_{j+1} - s_{j+1}\|^2$), then it holds that
\begin{align*}
    \Phi & =  (f(z_{j+1}) + \frac{m}{2}\|z_{j+1} - x_{k}\|^2  - \tilde{f}_j(z_{j+1}) ) - \frac{1}{2\rho} \|g_{j+1} - s_{j+1}\|^2 \\
    & \geq  \frac{1}{2}(f(z_{j+1}) + \frac{m}{2}\|z_{j+1} - x_{k}\|^2  - \tilde{f}_j(z_{j+1}) ) 
\end{align*}
On the other hand, if $\theta^\star = \frac{\rho(f(z_{j+1}) + \frac{m}{2}\|z_{j+1} - x_{k}\|^2 - \tilde{f}_j(z_{j+1}))}{\|s_{j+1} - g_{j+1}\|^2} $, we have 
\begin{align*}
    \Phi =  \frac{\rho(f(z_{j+1}) + \frac{m}{2}\|z_{j+1} - x_{k}\|^2 - \tilde{f}_j(z_{j+1}))^2}{2\|s_{j+1} - g_{j+1}\|^2}.
\end{align*}
Hence, the improvement $\Phi$ can be bounded as
\begin{align*}
    \Phi \geq \frac{1}{2} \min \left \{f(z_{j+1}) + \frac{m}{2}\|z_{j+1} - x_{k}\|^2  - \tilde{f}_j(z_{j+1})  ,  \frac{\rho(f(z_{j+1}) + \frac{m}{2}\|z_{j+1} - x_{k}\|^2 - \tilde{f}_j(z_{j+1}))^2}{\|s_{j+1} - g_{j+1}\|^2} \right \},
\end{align*}
proving \cref{eq:strict-improvement}.

To prove \cref{eq:strict-improvement-2}, it is sufficient to prove that if \cref{eq:descent-condition-main} is not satisfied at $z_{j+1}$, then 
\begin{align}
    \label{eq:key:approximation-gap}
    f(z_{j+1}) + \frac{m}{2}\|z_{j+1} - x_k\|^2 - \tilde{f}_{j}(z_{j+1}) \geq (1-\beta) \tilde{\Delta}_j.
\end{align}
Suppose \cref{eq:key:approximation-gap} holds, the improvement $\Phi$ then can be relaxed as 
\begin{align*}
    \Phi \geq  \frac{1}{2} \min \left \{(1-\beta) \tilde{\Delta}_j  ,  \frac{\rho (1-\beta)^2 \tilde{\Delta}_j^2}{\|s_{j+1} - g_{j+1}\|^2} \right \}.
\end{align*}
Indeed, if \cref{eq:descent-condition-main} is not satisfied at $z_{j+1}$, then 
\begin{align*}
f(z_{j+1}) + \frac{m}{2}\|z_{j+1} - x_k\|^2 & > f(x_k) - \beta \left(f(x_k) - \tilde{f}_{j}(z_{j+1})\right) \\
& = (1-\beta)(f(x_k) - \tilde{f}(z_{j+1}) )+\tilde{f}_{j}(z_{j+1}) \\
& \geq (1-\beta)(f(x_k) - \tilde{f}(z_{j+1}) - \frac{\rho}{2}\|z_{j+1}-x_k\|^2 )+\tilde{f}_{j}(z_{j+1}) \\
& = (1-\beta) \tilde{\Delta}_j+ \tilde{f}_{j}(z_{j+1}),
\end{align*}
where the first equality adds and subtracts $\tilde{f}(z_{j+1})$, the second inequality subtracts $(1-\beta) \frac{\rho}{2}\|z_{j+1}-x_k\|^2$, and the last equality applies the definition of the approximated gap $\tilde{\Delta}_j =  f(x_{k}) - \eta_j$. The above is then equivalent to \cref{eq:key:approximation-gap}, completing the proof.

\subsection{Optimal solution for the essential subproblem \cref{eq:essential-subproblem}}
\label{subsection:closed-form-sol}
% \noindent \textbf{Optimal solution for \cref{eq:essential-subproblem}:}
This subsection gives the derivation of the closed-form solution to \cref{eq:essential-subproblem}, i.e.,
\begin{align*}
    % \label{eq:essential-subproblem}
    \min_y \left \{ \hat{f}_{j+1}(y) + \frac{\rho}{2}\|y -x_k\|^2 \right \}.
\end{align*}
Another proof that directly verifies the optimality condition can be found in \cite[Appendix B]{diaz2023optimal}. For notational simplicity, we let $f_1 = \tilde{f}_j(z_{j+1})$, $f_2 = f(z_{j+1}) + \frac{m}{2}\|z_{j+1} - x_{k}\|^2$, $v_1 = s_{j+1} $ and $v_2 = g_{j+1}$. For a $\theta \in [0,1]$, we define the shorthands $f(\theta) = (1-\theta) f_1 + \theta f_2$ and $v(\theta) =  (1-\theta) v_1 +  \theta v_2$. The function $\hat{f}_{j+1}$ can be rewritten as 
\begin{align*}
    \hat{f}_{j+1}(y) = \max \left \{ f_1 + \innerproduct{v_1}{ y - z_{j+1}},f_2 + \innerproduct{v_2}{ y - z_{j+1}} \right \}.
\end{align*}
\Cref{eq:essential-subproblem} then becomes 
\begin{align}
     & \min_y \left \{ \hat{f}_{j+1}(y) + \frac{\rho}{2}\|y -x_k\|^2 \right \} \nonumber \\
    =  & \min_{y}  \max \left \{ f_1 + \innerproduct{v_1}{ y - z_{j+1}},f_2 + \innerproduct{v_2}{ y - z_{j+1}} \right \} + \frac{\rho}{2}\|y -x_k\|^2 \nonumber \\
    = &\min_{y}  \max_{\theta \in [0,1]} (1- \theta) (f_1 + \innerproduct{v_1}{ y - z_{j+1}}) + \theta (f_2 + \innerproduct{v_2}{ y - z_{j+1}})  + \frac{\rho}{2}\|y -x_k\|^2 \nonumber\\
    = &\min_{y}  \max_{\theta \in [0,1]}  f(\theta) +  \innerproduct{v(\theta)}{ y- z_{j+1}} + \frac{\rho}{2}\|y -x_k\|^2 \nonumber \\
     = &\max_{\theta \in [0,1]}  \min_{y}   f(\theta) +  \innerproduct{v(\theta)}{ y- z_{j+1}} + \frac{\rho}{2}\|y -x_k\|^2, \label{eq:min-max}
\end{align}
where the second equality uses the fact that $f_1 + \innerproduct{v_1}{ \cdot - z_{j+1}}$ and $f_2 + \innerproduct{v_2}{ y - z_{j+1}}$ are affine, and the last equality swaps the order of minimization and maximization since $[0,1]$ is convex and bounded \cite[Corollary 37.3.2]{rockafellar1997convex}. Fix a $\theta \in [0,1]$ in \cref{eq:min-max}. The optimal solution $y^\star$ in \cref{eq:min-max} can be found as 
\begin{align*}
    y^\star & = x_k - \frac{1}{\rho}v(\theta).
\end{align*}
Plugging this back into \cref{eq:min-max}, we arrive at the optimization
\begin{align*}
   & \max_{\theta \in [0,1]}    f(\theta) +  \innerproduct{v(\theta)}{ x_k - \frac{1}{\rho}v(\theta)- z_{j+1}} + \frac{\rho}{2}\|v(\theta)\|^2 \\
    = &\max_{\theta \in [0,1]}    f(\theta) +  \innerproduct{v(\theta)}{ x_k- z_{j+1}} - \frac{1}{2\rho}\|v(\theta)\|^2
\end{align*}
Taking the gradient with respect to $\theta$ and setting it to zero gives us 
\begin{align*}
    0 & = f_2 - f_1 + \innerproduct{v_2 - v_1}{x_k - z_{j+1}} - \frac{1}{\rho}\innerproduct{v_1 + \theta(v_2-v_1)}{v_2-v_1} \\
    & =f_2 - f_1 +  \innerproduct{v_2 - v_1}{\frac{1}{\rho} v_1} - \frac{1}{\rho}\innerproduct{v_1 + \theta(v_2-v_1)}{v_2-v_1} \\
    & = f_2 - f_1  - \frac{\theta}{\rho}\|v_1 - v_2\|^2, 
\end{align*}
where the second equality uses the update $z_{j+1} = x_k - \frac{1}{\rho}s_{j+1} = x_k - \frac{1}{\rho} v_1$. Considering the constraint $\theta \in [0,1]$, we find the optimal $\theta^\star$ is 
\begin{align*}
    \theta^\star = \min\left \{1, \frac{\rho(f_2 - f_1)}{\|v_1 - v_2\|^2} \right \} = \min\left \{1, \frac{\rho(f(z_{j+1}) + \frac{m}{2}\|z_{j+1} - x_{k}\|^2 - \tilde{f}_j(z_{j+1}))}{\|s_{j+1} - g_{j+1}\|^2} \right \}.
\end{align*}

\subsection{Proof for \cref{lemma:null-step}}
\label{subsection:standard-bundle-proof}
We first introduce a key inequality in this proof.
\begin{lemma} \label{lemma:key-improvement}
    Fix a center point $x_k \in \RR^n$. For all iteration $j\geq 1$ that \cref{eq:descent-condition-main} is not satisfied at $z_{j+1}$, it holds that
    \begin{align}
        \label{eq:key}
        \frac{1}{2\rho }\|s_{j+1}\|^2 \leq \tilde{\Delta}_j \leq \tilde{\Delta}_1  \leq \frac{1}{2\rho}\|g_1\|^2,
    \end{align}
    where $s_{j+1} = \rho (x_k - z_{j+1}) \in \tilde{f}_{j}(z_{j+1})$, $\tilde{\Delta}_j = f(x_k) - \eta_j$, and $g_1 \in \partial f(x_k)$.
\end{lemma}
\begin{proof}
    Let $j\geq 1$ be an iteration that \cref{eq:descent-condition-main} is not satisfied at $z_{j+1}$. We can then bound the step length $\|z_{j+1} - x_{k}\|$ as follows
    \begin{align*}
        \frac{\rho}{2}\|z_{j+1} - x_{k}\|^2 &\overset{(a)}{\leq}  \tilde{f}_j(x_{k}) - \left( \tilde{f}_j(z_{j+1}) +\frac{\rho}{2}\|z_{j+1} - x_{k}\|^2   \right) \\ 
        & \leq f(x_{k}) - \left( \tilde{f}_j(z_{j+1}) +\frac{\rho}{2}\|z_{j+1} - x_{k}\|^2   \right) \\
        & =\Tilde{\Delta}_j \\
        & \overset{(b)}{\leq}\Tilde{\Delta}_{1} \\
        & = f(x_{k}) - \left( \tilde{f}_{1}(z_{2}) +\frac{\rho}{2}\|z_{2} - x_{k}\|^2   \right) \\
        &  \overset{(c)}{=} f(x_{k})  - \left( f(x_{k})  + \innerproduct{g_1}{z_{2} - x_{k}} +\frac{\rho}{2}\|z_{2} - x_{k}\|^2   \right) \\
        & \overset{(d)}{\leq} \frac{1}{2\rho}\|g_{1}\|^2
    \end{align*}
    where $(a)$ comes from the $\rho$-strong convexity of $f_j(\cdot) + \frac{\rho}{2}\|\cdot - x_k\|^2$, $(b)$ uses $\tilde{\Delta}_{i+1} \leq \tilde{\Delta}_{i}$ for all $i \geq 1$, $(c)$ applies the assumption $\tilde{f}_1(\cdot) = f(x_k) + \innerproduct{g_1}{\cdot - x_k}$, and $(d)$ applies the Young's inequality
    $\innerproduct{v}{w} \leq \frac{1}{2a}\|v\|^2 + \frac{a}{2}\|w\|^2$ for all $a > 0$ with $v = g_{1}$ and $w= z_{2} - x_{k}$. The proof is then complete by substituting $s_{j+1} = \rho (x_k - z_{j+1})$.
\end{proof}

From \cref{eq:key}, we see that  
    $$
        \frac{2 \rho \tilde{\Delta}_j }{\|g_1\|^2} \leq 1  , \quad  \|s_{j+1}\|^2 \leq  \|g_1\|^2.
    $$
    The improvement in \cref{eq:strict-improvement-2} can be relaxed as
    \begin{align*}
        \frac{1}{2} \min \left \{(1-\beta) \tilde{\Delta}_j, \frac{(1 - \beta)^2 \rho \tilde{\Delta}_j^2}{  \|g_{j+1} -s_{j+1} \|^2} \right\} & \geq \frac{1}{2} \min \left \{\frac{(1-\beta) 2 \rho \tilde{\Delta}_j^2}{\|g_{1}\|^2}, \frac{(1 - \beta)^2 \rho \tilde{\Delta}_j^2}{ 2 \|g_{j+1}\|^2 +2\| s_{j+1} \|^2} \right\} \\
         &\geq \frac{1}{2} \min \left \{\frac{(1-\beta) 2 \rho \tilde{\Delta}_j^2}{\|g_{1}\|^2}, \frac{(1 - \beta)^2 \rho \tilde{\Delta}_j^2}{ 4 \|g_{j+1}\|^2 } \right\} \\
          & \geq \frac{(1 - \beta)^2 \rho \tilde{\Delta}_j^2}{ 8  G_k^2}
    \end{align*}
    where the first inequality uses $\frac{2 \rho \tilde{\Delta}_1 }{\|g_1\|^2} \leq 1 $ and $\|v+w\|^2 \leq 2\|v\|^2 + 2 \|w\|^2$ for all $v,w \in \RR^n$, the second inequality applies $ \|s_{j+1}\|^2 \leq  \|g_1\|^2$, and the last inequality uses the definition of $G_k$.
    
    Multiplying $(-1)$ and adding $f(x_k)$  on both sides of the above inequality lead to 
    \begin{align}
        \label{eq:improve}
        \tilde{\Delta}_{j+1} \leq \tilde{\Delta}_j - \frac{(1 - \beta)^2 \rho \tilde{\Delta}_j^2}{8 G_k^2},
    \end{align}    
    which is a recursive formula and $\tilde{\Delta}_{j} \geq  \Delta_k \geq 0, \forall j\geq 1$.
    
    We next review a bound for a recurrence relation.
    \begin{lemma}[{\cite[Lemma A.1]{diaz2023optimal}}]
        \label{lemma:recursion}
        Assume that a nonnegative sequence $\{\delta_k\}_{k=1}^\infty$ satisfies the recurrence 
            \[
            \delta_{k+1} \leq \delta_k - \alpha \delta_k^q.
            \]
            Then for any $\epsilon > 0$, the inequality $\delta_k \leq \epsilon$ holds for some
            \[
            k \leq  \frac{1}{(q-1)\alpha \epsilon^{q-1}} .
            \]
    \end{lemma}

    Applying \cref{lemma:recursion} to \cref{eq:improve} with $\epsilon = \Delta_k $, we see that the number of iterations to reach $\tilde{\Delta}_j \leq  \Delta_k $ can not be more than $\frac{8G_k^2}{(1 - \beta)^2 \rho\Delta_k }$, completing the proof.

\subsection{Proof of \cref{lemma:bounds-on-G-k}}
\label{subsection:proof:bounds-on-G-k}
    Recall that $G_k = \sup_{1\leq i \leq T_k} \|g_{i}\|$ and $g_i = v_i + m(z_i-x_k)$ with $v_i \in \partial f(z_i)$ for all $ 1\leq i \leq T_k$. Since $f$ is $L$-Lipschitz continuous, we have 
    \begin{align*}
        \|v_{i} + m (z_{i}- x_{k})\| \leq L + m \|z_{i} - x_{k}\|.
    \end{align*}
    From \cref{eq:key}, we can bound the step length $\|z_{i} - x_{k}\|$ as 
    \begin{align*}
        \|z_{i} - x_{k}\| \leq \frac{1}{\rho}\|g_1\| \leq \frac{1}{\rho}L.
    \end{align*}
    Putting everything together yields
    \begin{align*}
        G_{k} = \sup_{1\leq i \leq T_k} \|v_{i} + m (z_{i}- x_{k})\| \leq (1+  m/\rho ) L.
    \end{align*}

\subsection{Proof of \cref{corollary:main-result}}
\label{subsection:corollary:main-result}
Viewing \cref{lemma:stationary-conversion-R2M-2-main} with $\lambda = \rho$ and $\alpha = m + \rho$, we see that if a point $x$ is a $(\eta,\epsilon)$-inexact stationary point with $\epsilon \leq \frac{2}{\rho}\eta^2$ and $\eta \leq \frac{\delta \rho}{4 \alpha}$, then $x$ satisfies $\|\nabla f_{\alpha}(x)\| \leq \delta$. Combining this observation with \cref{theorem:main-result} shows that \Cref{alg:Proxi-descent} finds a $(\delta,\alpha)$-Moreau stationary point in at most 
\begin{align*}
     \frac{8K_{\max}(\frac{\delta \rho}{4 \alpha},\frac{\delta^2 \rho }{8 \alpha^2})}{(1 - \beta)^2 \rho} \times \left (1 + \frac{m}{\rho} \right)^2L^2 \times \max\left\{ \frac{32 \alpha^2}{\rho \delta^2},\frac{8 \alpha^2}{ \rho \delta^2 } \right\}=  \frac{8K_{\max}(\frac{\delta \rho}{4 \alpha},\frac{\delta^2 \rho }{8 \alpha^2})}{(1 - \beta)^2 \rho} \times \left (1 + \frac{m}{\rho} \right)^2L^2 \times  \frac{32 \alpha^2}{\rho \delta^2}
\end{align*}
subgradient and function evaluations.

\subsection{Analytical solution for the essential model in \cref{subsection:re-interpretation}}
\label{subsection:analytical-sol}
As discussed in \cref{subsection:re-interpretation}, the proximal bundle method is equivalent to \Cref{alg:Proxi-descent}. The only difference between  \Cref{alg:Proxi-descent,alg:PBM-weakly} is the indices. In fact, the proximal subproblem
\begin{align*}
     \argmin_{y}\; f_{k+1}(y) + \frac{\rho}{2}\|y - x_k\|^2
\end{align*}
with $f_{k+1}$ defined as 
\begin{align*}
    f_{k+1}(y) = \max\left\{f_k(z_{k+1})\! +\! \innerproduct{s_{k+1}}{y-z_{k+1}}, f(z_{k+1}) \! +\!  \frac{m}{2}\|z_{k+1} - x_k\|^2 + \innerproduct{g_{k+1}}{y-z_{k+1}} \right\},
\end{align*}
has the optimal solution 
\begin{align*}
    y^\star &=  x_k -  \frac{1}{\rho}((1-\theta^\star)s_{k+1} + \theta^\star g_{k+1}), \\
    \theta^\star &= \min\left \{1, \frac{\rho(f(z_{k+1}) + \frac{m}{2}\|z_{k+1} - x_{k}\|^2 - f_k(z_{k+1}))}{\|s_{k+1} - g_{k+1}\|^2} \right \},
\end{align*}
as established in \cref{subsection:closed-form-sol} by setting a change of notations and indices.

\section{Improved convergence rates in \cref{section:improved-rate}}
\label{Apx-section:improved-rate}
In this section, we establish the faster convergence rates in \cref{section:improved-rate} when under favorable regularity conditions, such as smoothness and quadratic growth.

\subsection{Proof of \cref{proposition:Upper-bound_G}}
\label{subsection:Upper-bound_G}
Our proof is adapted from \cite[Section 5.2]{diaz2023optimal}. Recall the definition of $G_k = \sup_{1\leq j \leq T_k} \|g_{j}\|$ where $g_j \in \partial (f(\cdot) + \frac{m}{2}\|\cdot- x_{k}\|^2 )(z_{j} )$ and $z_1 = x_k$. Since we assume the function $f$ is $M$-smooth, it becomes $g_j = \nabla f(z_j) + m(z_j - x_k)$.
Thus, it holds that
\begin{equation}
    \label{eq:bound-G-step-1}
    \begin{aligned}
        G_k & = \sup_{1\leq j \leq T_k} \|g_{j}\| \\
    & \leq \sup_{1\leq j \leq T_k}  M\| z_j - x_k \|  + \|g_1\|  \\
    & \leq (1+ M/\rho) \|g_1\| \\
    & =  (1+ M/\rho) \|\nabla f(x_k)\|
    \end{aligned}
\end{equation}
where the first inequality uses $M$-smoothness, and the second inequality follows from \cref{eq:key}. 

On the other hand, we have 
\begin{equation}
    \label{eq:bound-G-step-2}
    \begin{aligned}
        \Delta_{k} & = f(x_k) - \min_{y} \left \{f(y) + \frac{ \alpha }{2}\|y - x_k\|^2 \right\} \\
        & \geq f(x_k) - \min_{y} \left \{f(x_k) + \innerproduct{\nabla f(x_k)}{ y- x_k} + \frac{ M+\alpha }{2}\|y - x_k\|^2 \right\} \\
        & =  \frac{1}{2(M+\alpha)}\|\nabla f(x_k)\|^2,
    \end{aligned}
\end{equation}
where the inequality uses the fact that the function $f$ being $M$-smooth implies 
\begin{align*}
    f(x) \leq f(y) + \innerproduct{\nabla f(y)}{ y - x} + \frac{M}{2}\|y-x\|,\; \forall x, y \in \RR^n.
\end{align*}
Combining \cref{eq:bound-G-step-1,eq:bound-G-step-2} yields
\begin{align*}
    \Delta_{k} \leq (1+ M/\rho) \sqrt{2(M+\alpha)\Delta_k},
\end{align*}
completing the proof of \cref{proposition:Upper-bound_G}.

\subsection{Improved convergence rates in \cref{subsection:QG}}
\label{section:improvd-rate-appendix}
In this subsection, we complete the proof in \cref{subsection:QG} where we consider a $m$-weakly convex function $f:\RR^n \to \RR$ that satisfies quadratic growth 
\begin{align}
    \label{eq:QG-appendix}
        f(x) - f^\star \geq \frac{\muq}{2}\cdot \Dist^2(x,S), \; \forall x \in \RR^n,
\end{align}
where $S  = \argmin_x  f(x)$ is the set of optimal solutions and is assumed nonempty, and we further assume the condition $\muq > m$. Under these hypotheses, all the stationary points of $f$ are global minima, as the Polyak-{\L}ojasiewicz (PL) inequality holds true \cite[Theorem 3.1]{liao2024error}. Therefore, it becomes possible to measure the cost value gap $f(x_k) - f^\star$ for \Cref{alg:Proxi-descent}. 

It is known that the proximal point method has been shown to converge linearly for weakly convex optimization when the function satisfies \cref{eq:QG-appendix} and $\muq > m$ \cite[Theorem 4.2]{liao2024error}. On the other hand, as introduced in \cref{section:PBM-inexact}, \Cref{alg:Proxi-descent} is designed to mimic the conceptual proximal point method. Hence, under the same assumptions, we expect that \Cref{alg:Proxi-descent} converges linearly as well. Our analysis can be viewed as a generalization of the analysis for convex \PBM{} in \cite{diaz2023optimal}. 

Recall that we use the notation $\Delta_k := f(x_k) -\inf_x  \{f(x) + \frac{\alpha}{2}\|x - x_k\|^2\} = f(x_k) - f_{\alpha}(x_k)$  to denote the proximal gap at the iterate $x_k$, where $\alpha =m + \rho $ with $\rho > 0$. 
 We first present an important result that quantifies the decrease of the cost value in terms of the proximal gap at every step of \Cref{alg:Proxi-descent}.

\begin{lemma}
    \label{lemma:descent-proximal-gap}
    Suppose $f:\RR^n \to \RR $ is an $m$-weakly convex function. At every iteration $k$ of \Cref{alg:Proxi-descent}, it holds that
    \begin{align*}
          f(x_{k+1})\leq f(x_k) - \beta \Delta_k.
    \end{align*}
\end{lemma}
\begin{proof}
    By the update rule of $\texttt{ProxDescent($x_k,\beta,\rho$)}$, we know that there is a lower approximation function $\tilde{f} \leq f(\cdot) + \frac{m}{2}\|\cdot - x_k\|^2$ such that 
    \begin{align}
        \label{eq:construct-xkp1}
        x_{k+1} = \argmin_{y} \tilde{f}(y) + \frac{\rho}{2}\|y-x_k\|^2
    \end{align} 
    and $x_{k+1}$ satisfies \cref{eq:descent-condition-main}, i.e., 
    $
    \beta \left(f(x_k) - \tilde{f}(x_{k+1})\right) \leq f(x_k) - \left(f(x_{k+1}) + \frac{m}{2}\|x_{k+1} - x_k\|^2\right).
    $
    It follows that 
    \begin{align*}
        \Delta_k & = f(x_k) - \inf_x \left \{f(x) + \frac{\alpha}{2}\|x - x_k\|^2 \right\} \\
        & \leq f(x_k) - \inf_x \left \{ \tilde{f}(x) + \frac{\rho}{2}\|x - x_k\|^2 \right \} \\
        & = f(x_k) - \left (\tilde{f}(x_{k+1}) + \frac{\rho}{2}\|x_{k+1} - x_k\|^2 \right ) \\
        & \leq \frac{1}{\beta}\left( f(x_k) - \left(f(x_{k+1}) + \frac{m}{2}\|x_{k+1} - x_k\|^2\right) \right) - \frac{\rho}{2}\|x_{k+1} - x_k\|^2\\
        & \leq \frac{1}{\beta}(f(x_k) - f(x_{k+1}))
    \end{align*}
    where the first equality is due to, the second equality is from \cref{eq:construct-xkp1}, the second inequality uses the definition of \cref{eq:descent-condition-main}, and the last inequality drops two nonnegative terms. Rearranging the above inequality finishes the proof.    
\end{proof}

Note that the cost value improvement in  \cref{lemma:descent-proximal-gap} is quantified by the proximal gap $\Delta_k$, which depends on the Moreau envelope $f_{\alpha}$. The Moreau envelope $f_{\alpha}$ can be estimated in the following result.

\begin{lemma}
    \label{lemma:Moreau-improvement}
    Let $f:\RR^n \to \RR $ be an $m$-weakly convex function and satisfies \cref{eq:QG-appendix} with $\mu_{\mathrm{q}} > m$. Let $S = \argmin_x f(x)$ and $\alpha = m + \rho $ with $\rho > 0$. Fix $w \in \RR^n \setminus \{S\}$ and $x^\star = \argmin_{y \in S} \|w - y\|$. Define
    $$\Lambda :=f(w) - f(x^\star) - \frac{m}{2}\|w - x^\star\|^2 \geq \frac{\mu_\mathrm{q}}{2} \|w - x^\star\|^2 - \frac{m}{2}\|w - x^\star\|^2  > 0.$$ 
    Then the Moreau envelope can be estimated as 
    \begin{align}   
        \label{eq:Moreau-envelope-estimate}
        f_{\alpha}(w) =  \inf_x  \left \{ f(x) + \frac{\alpha}{2}\|x - w\|^2 \right \}\leq \begin{cases}
            f(w) -  \Lambda + \frac{\rho}{2}\|w - x^\star\|^2,&\text{ if }  \Lambda > \rho\|w - x^\star\|^2 \\
            f(w) - \frac{\Lambda^2}{2\rho\|w - x^\star\|^2} &\text{ if }\Lambda \leq \rho\|w - x^\star\|^2.
        \end{cases}
    \end{align}
\end{lemma}
\begin{proof}
    The proof is inspired by \cite[Lemma 7.12]{ruszczynski2011nonlinear} and follows from a few simple steps. To begin, we first recall a property for weakly convex functions \cite[Lemma 2.1]{davis2019stochastic}. That is, $f$ is $m$-weakly convex if and only if 
    \begin{align}
        \label{eq:weakly-convex-property}
        f(\gamma x + (1-\gamma) y ) \leq \gamma f(x) + (1-\gamma) f(y) + \frac{\gamma(1-\gamma)m}{2}\|x-y\|^2, \; \forall x,y \in \RR^n,\forall \gamma \in [0,1]. 
    \end{align}
    Therefore, it holds that  
    \begin{align*}
          f_{\alpha}(w)= &\inf_x f(x) + \frac{\alpha}{2}\|x - w\|^2 \\
          \leq & \inf_{t \in [0,1]} f((1-t)w + tx^\star) + \frac{\alpha}{2}\|(1-t)w + tx^\star - w\|^2 \\
          \leq & \inf_{t \in [0,1]} (1-t)f(w) + tf(x^\star) +\frac{t(1-t)m}{2}\|w - x^\star \|^2  + \frac{\alpha t^2}{2}\|w - x^\star\|^2 \\
         = &  f(w) + \inf_{t \in [0,1]} -t \Lambda + \frac{\rho}{2}t^2\|x^\star - w\|^2,
    \end{align*}
    where the first inequality restricts the searching space to the line segment between $w$ and $x^\star$, the second inequality uses the property of an $m$-weakly convex function \cref{eq:weakly-convex-property}, and the last equality uses the definition of $\Lambda$ and $\alpha = m+\rho $.
    The minimizer $t^\star$ in the last line can be computed as 
    \begin{align*}
        t^\star = \min \left \{1,\frac{\Lambda}{\rho\|x^\star - w\|^2} \right\}.
    \end{align*}
    Plugging in the minimizer $t^\star$ with some algebra completes the proof.
\end{proof}

In the following content, we will use \cref{lemma:descent-proximal-gap,lemma:Moreau-improvement} to prove \cref{lemma:Prox-descent-linear,theorem:main-result-QG,corollary:main-result-QG}.

\subsubsection{Proof of \cref{lemma:Prox-descent-linear}}
\label{subsection:proof-Prox-descent-linear}
To facilitate readability, we restate \cref{lemma:Prox-descent-linear} here. 
\begin{lemma}
    \label{lemma:Prox-descent-linear-appx}
    Consider the $m$-weakly convex problem \cref{eq:main-problem}. Suppose the function $f$ satisfies quadratic growth \cref{eq:QG} with $\muq > m$. Then  \Cref{alg:Proxi-descent} generates a sequence of iterates, satisfying
    \begin{align*}
        f(x_{k+1}) - f^\star \leq \gamma (f(x_{k}) - f^\star), \; \forall k \geq 1,
    \end{align*}
    where $\gamma \in (0,1)$.
\end{lemma}

The proof of \cref{lemma:Prox-descent-linear-appx} follows from a combination of \cref{lemma:descent-proximal-gap,lemma:Moreau-improvement}. Let $x^\star_k = \argmin_{y \in S} \|x_k - y\|$ (where $S = \argmin_x f(x)$) and $\Lambda_k  =f(x_k) - f(x^\star_k) - \frac{m}{2}\|x^\star_k - x_k  \|^2 > 0 $ (due to the assumption $\muq > m$).  \cref{lemma:Moreau-improvement} shows that 
\begin{align*}
    f_{\alpha}(x_k)   \leq \begin{cases}
            f(x_k) -  \Lambda_k + \frac{\rho}{2}\|x_k - x^\star_k\|^2,&\text{ if }  \Lambda_k > \rho\|x_k- x^\star_k\|^2 \\
            f(x_k) - \frac{\Lambda_k^2}{2\rho\|x_k- x^\star_k\|^2}, &\text{ if }\Lambda_k \leq \rho\|x_k - x^\star_k\|^2.
        \end{cases}
\end{align*}
This leads to the bound on the proximal gap $\Delta_k:= f(x_k) - f_{\alpha}(x_k)$:  
\begin{align}
\label{eq:lower-bound-proximal-gap-linear}
    \Delta_k & \geq \begin{cases}
             \Lambda_k - \frac{\rho}{2}\|x_k - x^\star_k\|^2,&\text{ if }  \Lambda_k > \rho\|x_k- x^\star_k\|^2 \\
            \frac{\Lambda^2_k}{2\rho\|x_k- x^\star_k\|^2} &\text{ if }\Lambda_k \leq \rho\|x_k - x^\star_k\|^2
        \end{cases} \nonumber \\
         & \geq \begin{cases}
             \frac{1}{2}\Lambda_k,&\text{ if }  \Lambda_k > \rho\|x_k- x^\star_k\|^2 \\
            \frac{\bar{\mu}_{\mathrm{q}}}{4\rho}\Lambda_k &\text{ if }\Lambda_k \leq \rho\|x_k - x^\star_k\|^2
        \end{cases} \nonumber\\
        & \geq \kappa \Lambda_k,
\end{align}
where in the second inequality we set $\bar{\mu}_{\mathrm{q}} = \muq - m > 0$ and use the fact that $\Lambda_k \geq \frac{\bar{\mu}_{\mathrm{q}}}{2}\|x_k - x_k^\star\|^2$ due to \cref{eq:QG-appendix}, and the last inequality lets  $\kappa = \min \{ \frac{1}{2}, \frac{\bar{\mu}_{\mathrm{q}}}{4\rho}\} >0$.

Combining \cref{lemma:descent-proximal-gap,eq:lower-bound-proximal-gap-linear} yields
\begin{align*}
    f(x_{k+1}) \leq f(x_k) - \beta \kappa \Lambda_k,
\end{align*}
implying
\begin{align*}
    f(x_{k+1}) - f^\star & \leq f(x_k) -f^\star - \beta \kappa \Lambda_k \\
    & = (1-\beta \kappa)(f(x_k) -f^\star ) + \beta \kappa \frac{m}{2}\|x^\star_k - x_k  \|^2 \\
    & \leq \left(1- \beta\kappa(1  -m /\muq) \right)(f(x_k) -f^\star),
\end{align*}
where the first equality uses the definition of $\Lambda_k$, and the last inequality applies the quadratic growth \cref{eq:QG-appendix}. By the assumption, we have $m/\muq < 1$. The above then shows that the cost value gap is reduced at least by a factor $ \left(1- \beta\kappa(1  -m /\muq) \right) \in (0,1)$. Hence, the reduction factor $\gamma$ and constant $C$ in \cref{lemma:Prox-descent-linear-appx} can be chosen as $ \left(1- \beta\kappa(1  -m /\muq) \right) \in (0,1)$ and $1$, respectively.

\subsubsection{Proof of \cref{theorem:main-result-QG}}
\label{subsection:proof-theorem:main-result-QG}

From \cref{eq:inexact-subgradient-bound-linear}, we know that $\frac{2\alpha^2}{m + \beta \rho} (f(x_k) - f^\star) \leq \eta^2$ (resp. $\frac{1-\beta}{\beta} (f(x_k) - f^\star) \leq \epsilon $) implies   $\|\tilde{g}_{k+1}\| \leq \eta$ (resp. $\epsilon_{k+1} \leq \epsilon$). Let $N =  (\log(\gamma^{-1}))^{-1} \log\left (  (f(x_1) - f^\star) \max\left\{\frac{(1-\beta)}{\beta\epsilon}, \frac{2 \alpha^2}{(m+\beta \rho) \eta^2} \right\}\right) +1 $ be the iteration number in \cref{eq:descent-bound-linear}. From \cref{lemma:proximal-descent-linear}, we know that 
\begin{align*}
    f(x_k) - f^\star \leq \sigma : =\min\left \{ \frac{m + \beta     \rho}{2\alpha^2}\eta^2 , \frac{\beta}{1- \beta} \epsilon \right\},\; \forall k \geq N.
\end{align*}
Let $T \leq N$ be the first index such that $f(x_{T+1}) - f^\star \leq  \sigma$. It then holds that 
\begin{align}
    \label{eq:chain}
    \sigma < f(x_T) - f^\star \leq \gamma (f(x_{T-1}) - f^\star) \leq \gamma^2 (f(x_{T-2}) - f^\star) \leq \cdots \leq \gamma^{T-1} (f(x_{1}) - f^\star).
\end{align}
On the other hand, using the assumptions of quadratic growth \cref{eq:QG-appendix} and $\muq > m$, \cref{eq:lower-bound-proximal-gap-linear} can be relaxed as 
\begin{align}
    \label{eq:lower-bound-relaxed}
    \Delta_k \geq \kappa \Lambda_k \geq \kappa (1-m/\muq)( f(x_k) -f^\star),\; \forall k\geq 1.
\end{align}
Combining \cref{eq:chain,eq:lower-bound-relaxed} yields 
\begin{align*}
    \sigma < \gamma^{T-k} (f(x_k) - f^\star) \leq \frac{\gamma^{T-k}}{\kappa(1-m/\muq)} \Delta_k,\; \forall 1\leq  k \leq T,
\end{align*}
implying 
\begin{align*}
    \frac{\kappa(1-m/\muq)}{\gamma^{T-k}} \sigma \leq \Delta_k, \; \forall 1\leq  k \leq T. 
\end{align*}
Plugging this into the bound in \cref{lemma:null-step} and using \cref{lemma:bounds-on-G-k} yields  
\begin{align*}
    T_k \leq  \frac{8 (1 + m/\rho)^2L^2}{(1 - \beta)^2 \rho } \frac{\gamma^{T-k}}{\kappa(1-m/\muq)\sigma},  \; \forall 1\leq  k \leq T.
\end{align*}
This leads to the bound on the total function value and subgradient evaluation
\begin{align*}
    \sum_{k=1}^T T_{k} & \leq \frac{8 (1 + m/\rho)^2L^2}{(1 - \beta)^2 \rho } \frac{1}{\kappa(1-m/\muq)\sigma} \sum_{k=1}^T  \gamma^{T-k} \\
    & \leq  \frac{8 (1 + m/\rho)^2L^2}{(1 - \beta)^2 \rho } \frac{1}{\kappa(1-m/\muq)\sigma (1-\gamma)} \\
    & \leq \frac{8 (1 + m/\rho)^2L^2}{(1 - \beta)^2 \rho } \frac{1}{\kappa(1-m/\muq)(1-\gamma) \min \left \{ \frac{m + \beta     \rho}{2\alpha^2} , \frac{\beta}{1- \beta}  \right\} } \max \left \{\frac{1}{\eta^2},\frac{1}{\epsilon}\right \}.
\end{align*}
This completes the proof.

\subsubsection{Proof of \cref{corollary:main-result-QG}}
\label{subsection:corollary:main-result-QG}
Similar to the proof of \cref{corollary:main-result} in \cref{subsection:corollary:main-result}, we choose $\epsilon \leq \frac{2}{\rho}\eta^2$ and $\eta \leq \frac{\delta \rho}{4 \alpha}$ with $\delta \geq  0$. \Cref{theorem:main-result-QG} guarantees to find a $(\eta,\epsilon)$-inexact stationary point in at most $\bigO\left(\max\{\frac{1}{\eta^2}, \frac{1}{\epsilon}\}\right)$ subgradient and function evaluations. This with \cref{lemma:stationary-conversion-R2G-main-text} shows that \Cref{theorem:main-result-QG} finds a $(\delta,\alpha)$-Moreau stationary point in at most $\bigO(1/\delta^2)$ subgradient and function evaluation.

\section{Details in the numerical experiment \cref{section:numerics}}
\label{Apx-section:numerics}
In this section, we show that the two functions
\begin{align}
    \label{eq:Phase-retrieval-apx}
    f(x) = \frac{1}{n} \sum_{i = 1}^n |\innerproduct{a_i}{x}^2 - b_i|
\end{align}
and 
\begin{align}
    \label{eq:Blind-deconvolution-apx}
    f(x,y) = \frac{1}{n} \sum_{i = 1}^n |\innerproduct{u_i}{x}\innerproduct{v_i}{y} - b_i|
\end{align}
in \cref{subsection:phase,subsection:blind} are weakly convex. We first start by reviewing a result from \cite[Lemma 4.2]{drusvyatskiy2019efficiency} that shows the composition of a convex Lipschitz function and a smooth function is weakly convex.

\begin{lemma}[{\cite[Lemma 4.2]{drusvyatskiy2019efficiency}}]
    \label{lemma:comp-weakly}
    Let $h:\RR^m \to \RR$ be a convex $L$-Lischitz continous function and $c:\RR^d \to \RR^m$ be a $\beta$-smooth function. Then the function $h(c(\cdot))$ is $L\beta$-weakly convex.
\end{lemma}

We next present two straightforward results that show that the summation of finite smooth functions is smooth and the scale of a smooth function is smooth.
\begin{proposition}
    \label{prop:sum-smmoth}
    Suppose $f_i:\RR^d \to \RR $ is $\beta_i$-smooth for $i = 1,\ldots,n$. Then the function $f = \sum_{i=1}^n f_i$ is $\sum_{i=1}^n \beta_i$-smooth.
\end{proposition}
\begin{proof}
    By the definition, we know that for each $i = 1,\ldots,n$,
    \begin{align*}
        \|\nabla f_i(x) - \nabla f_i (y)\| \leq \beta_i \|x -y \|, \; \forall x, y \in \RR^n.
    \end{align*}
    It follows that 
    \begin{align*}
        \|\nabla f(x) - \nabla f(y)\| = \|\sum_{i=1}^n \nabla f_i(x) - \nabla f_i(y)\| \leq \sum_{i=1}^n \|\nabla f_i(x) - \nabla f_i(y)\| \leq \sum_{i=1}^n \beta_i \|x-y\|,  
    \end{align*}
    where the first inequality uses the triangle inequality. Thus, $f$ is $(\sum_{i=1}^n \beta_i)$-smooth.
\end{proof}

\begin{proposition}
    \label{prop:scaled-smooth}
    Suppose $f:\RR^d \to \RR$ is $M$-smooth and $c \in \RR$. Then $cf$ is $|c|M$-smooth.
\end{proposition}
\begin{proof}
    The proof is straightforward as 
    \begin{align*}
        \|\nabla cf(x) - \nabla cf(y)\| = \|c (\nabla f(x) - \nabla f(y)) \| \leq |c| \|\nabla f(x) - \nabla f(y)\| \leq |c|M\|x-y\|. 
    \end{align*}
\end{proof}

With \cref{lemma:comp-weakly,prop:sum-smmoth,prop:scaled-smooth}, the weak convexity of the functions in \cref{eq:Phase-retrieval-apx,eq:Blind-deconvolution-apx} can be established readily.
In particular, the functions in \cref{eq:Phase-retrieval-apx,eq:Blind-deconvolution-apx} can all be written in the form of 
\begin{align*}
    f(\cdot) = \frac{1}{n} \sum_{i= 1}^n h_i(c_i(\cdot)),
\end{align*}
where $h_i:\RR^m \to \RR$ is a convex and $L_i$-Lischitz continous function and $c_i:\RR^d \to \RR^m$ is a $\beta_i$-smooth function for each $i$. Thus, the function $h_i(c_i(\cdot))$ is $L_i\alpha_i$-weakly convex by \cref{lemma:comp-weakly}, and the function $f$ is $(\sum_{i=1}^n L_i \alpha_i)/n$-weakly by \cref{prop:sum-smmoth,prop:scaled-smooth}. Precisely, if we let $h_i (\cdot) = |\cdot|$ and $c_i(\cdot) = \innerproduct{a_i}{x}^2 - b_i $ for each $i =1,\ldots,n$, then the function in \cref{eq:Phase-retrieval-apx} can be written as $\frac{1}{n}\sum_{i=1}^n h_i(c_i(\cdot))$. Similarly, if we let $h_i (\cdot) = |\cdot|$ and $c_i(x,y) = \innerproduct{u_i}{x}\innerproduct{v_i}{y} - b_i$ for each $i=1,\ldots,n$, then the function in \cref{eq:Blind-deconvolution-apx} can be written as $\frac{1}{n}\sum_{i=1}^n h_i(c_i(x,y))$ as well. 

It is clear that the absolution value function $h(\cdot) = |\cdot|$ is convex and $1$-Lipschiz continuous. What remains to show is that the function $c_i$ is smooth.
\begin{proposition}
    \label{prop:smooth-phase}
    The function $c:\RR^d \to \RR$ as $c(x) = \innerproduct{a_i}{x}^2 - b_i$ is $2 \|a_i\|^2$ smooth.    
\end{proposition}
\begin{proof}
Note that the gradient can be computed as $\nabla c(x) = 2 a_i a_i^\tr x.$ Thus, we have 
\begin{align*}
    \|\nabla c(x_1) - \nabla c(x_2)\| = 2\|a_i a_i^\tr (x_1 - x_2)\| \leq 2 \|a_i a_i^\tr\|\|x_1 - x_2\| = 2 \|a_i\|^2\|x_1 - x_2\|,
\end{align*}
finishing the proof.
\end{proof}

\begin{proposition}
    \label{prop:smooth-blind}
    The function $c:\RR^d \times \RR^d \to \RR$ as $c(x,y) =\innerproduct{u_i}{x}\innerproduct{v_i}{y} - b_i$ is $| v_i^\tr u_i |$ smooth.
\end{proposition}
\begin{proof}
    Note that $\nabla f(x,y) = \begin{bmatrix} u_i v_i^\tr y  \\ v_i u_i^\tr x \end{bmatrix} $.
    Let $x_1,y_1,x_2,y_2 \in \RR^n$. We then have  
    \begin{align*}
        \|\nabla f(x_1,y_1) - \nabla f(x_2,y_2)\| & = \| \begin{bmatrix}
            u_i v_i^\tr  (y_1 - y_2) \\
            v_i u_i^\tr (x_1 - x_2)
        \end{bmatrix}\| = \| \begin{bmatrix}
            v_i u_i^\tr (x_1 - x_2) \\
            u_i v_i^\tr  (y_1 - y_2) 
        \end{bmatrix}\| \\
        & = \| \begin{bmatrix}
            v_i u_i^\tr & 0 \\
            0 & u_i v_i^\tr 
        \end{bmatrix} \begin{bmatrix}
            x_1 - x_2 \\
            y_1 - y_2
        \end{bmatrix} \| \\
        & \leq \| \begin{bmatrix}
            v_i u_i^\tr & 0 \\
            0 & u_i v_i^\tr 
        \end{bmatrix} \| \| \begin{bmatrix}
            x_1 - x_2 \\
            y_1 - y_2
        \end{bmatrix}  \|  \\
        & =\|v_i u_i^\tr\| \| \begin{bmatrix}
            x_1 - x_2 \\
            y_1 - y_2
        \end{bmatrix}  \| \\
        & = | u_i^\tr v_i | \| \begin{bmatrix}
            x_1 - x_2 \\
            y_1 - y_2
        \end{bmatrix}  \|,
    \end{align*}
    showing that the function $c$ is $| u_i^\tr v_i |$-smooth.
\end{proof}

From \cref{prop:smooth-phase,prop:smooth-blind} and the discussion above, we see that \cref{eq:Phase-retrieval-apx,eq:Blind-deconvolution-apx} is $2(\sum_{i=1}^n \|a_i \|^2)/n$ and $(\sum_{i=1}^n | u_i^\tr v_i |)/n$-smooth, respectively.

\newpage